\PassOptionsToPackage{unicode}{hyperref}
\PassOptionsToPackage{hyphens}{url}
\PassOptionsToPackage{dvipsnames,svgnames,x11names}{xcolor}
\documentclass[
  11pt,
  a4paper,
]{article}

\usepackage{amsmath,amssymb}
\usepackage{iftex}
\ifPDFTeX
  \usepackage[T1]{fontenc}
  \usepackage[utf8]{inputenc}
  \usepackage{textcomp} % provide euro and other symbols
\else % if luatex or xetex
  \usepackage{unicode-math}
  \defaultfontfeatures{Scale=MatchLowercase}
  \defaultfontfeatures[\rmfamily]{Ligatures=TeX,Scale=1}
\fi
\usepackage[]{libertinus}
\ifPDFTeX\else  
\fi
\IfFileExists{upquote.sty}{\usepackage{upquote}}{}
\IfFileExists{microtype.sty}{% use microtype if available
  \usepackage[]{microtype}
  \UseMicrotypeSet[protrusion]{basicmath} % disable protrusion for tt fonts
}{}
\makeatletter
\@ifundefined{KOMAClassName}{% if non-KOMA class
  \IfFileExists{parskip.sty}{%
    \usepackage{parskip}
  }{% else
    \setlength{\parindent}{0pt}
    \setlength{\parskip}{6pt plus 2pt minus 1pt}}
}{% if KOMA class
  \KOMAoptions{parskip=half}}
\makeatother
\usepackage{xcolor}
\usepackage[a4paper,textheight=24cm,textwidth=15.5cm]{geometry}
\makeatletter
\ifx\paragraph\undefined\else
  \let\oldparagraph\paragraph
  \renewcommand{\paragraph}{
    \@ifstar
      \xxxParagraphStar
      \xxxParagraphNoStar
  }
  \newcommand{\xxxParagraphStar}[1]{\oldparagraph*{#1}\mbox{}}
  \newcommand{\xxxParagraphNoStar}[1]{\oldparagraph{#1}\mbox{}}
\fi
\ifx\subparagraph\undefined\else
  \let\oldsubparagraph\subparagraph
  \renewcommand{\subparagraph}{
    \@ifstar
      \xxxSubParagraphStar
      \xxxSubParagraphNoStar
  }
  \newcommand{\xxxSubParagraphStar}[1]{\oldsubparagraph*{#1}\mbox{}}
  \newcommand{\xxxSubParagraphNoStar}[1]{\oldsubparagraph{#1}\mbox{}}
\fi
\makeatother

\usepackage{color}
\usepackage{fancyvrb}

\DefineVerbatimEnvironment{Highlighting}{Verbatim}{commandchars=\\\{\}}
\newenvironment{Shaded}{}{}

\newcommand{\AttributeTok}[1]{\textcolor[rgb]{0.84,0.23,0.29}{#1}}

\newcommand{\ConstantTok}[1]{\textcolor[rgb]{0.00,0.36,0.77}{#1}}

\newcommand{\DecValTok}[1]{\textcolor[rgb]{0.00,0.36,0.77}{#1}}

\newcommand{\FloatTok}[1]{\textcolor[rgb]{0.00,0.36,0.77}{#1}}
\newcommand{\FunctionTok}[1]{\textcolor[rgb]{0.44,0.26,0.76}{#1}}

\newcommand{\NormalTok}[1]{\textcolor[rgb]{0.14,0.16,0.18}{#1}}

\newcommand{\OtherTok}[1]{\textcolor[rgb]{0.44,0.26,0.76}{#1}}

\newcommand{\SpecialCharTok}[1]{\textcolor[rgb]{0.00,0.36,0.77}{#1}}

\usepackage{longtable,booktabs,array}
\usepackage{calc} % for calculating minipage widths
\usepackage{etoolbox}
\makeatletter
\patchcmd\longtable{\par}{\if@noskipsec\mbox{}\fi\par}{}{}
\makeatother
\IfFileExists{footnotehyper.sty}{\usepackage{footnotehyper}}{\usepackage{footnote}}
\makesavenoteenv{longtable}
\usepackage{graphicx}
\makeatletter
\def\maxwidth{\ifdim\Gin@nat@width>\linewidth\linewidth\else\Gin@nat@width\fi}
\def\maxheight{\ifdim\Gin@nat@height>\textheight\textheight\else\Gin@nat@height\fi}
\makeatother
\setkeys{Gin}{width=\maxwidth,height=\maxheight,keepaspectratio}
\makeatletter
\def\fps@figure{htbp}
\makeatother

\usepackage{tikz}
\usepackage{xcolor}
\usepackage{tabularx}
\usepackage{orcidlink}
\usepackage{lineno}
\usepackage{libertinus}
\usepackage{draftwatermark}

\usepackage{amsmath}
\makeatletter
\@ifpackageloaded{tcolorbox}{}{\usepackage[skins,breakable]{tcolorbox}}
\@ifpackageloaded{fontawesome5}{}{\usepackage{fontawesome5}}
\definecolor{quarto-callout-color}{HTML}{909090}
\definecolor{quarto-callout-note-color}{HTML}{0758E5}
\definecolor{quarto-callout-important-color}{HTML}{CC1914}
\definecolor{quarto-callout-warning-color}{HTML}{EB9113}
\definecolor{quarto-callout-tip-color}{HTML}{00A047}
\definecolor{quarto-callout-caution-color}{HTML}{FC5300}
\definecolor{quarto-callout-color-frame}{HTML}{acacac}
\definecolor{quarto-callout-note-color-frame}{HTML}{4582ec}
\definecolor{quarto-callout-important-color-frame}{HTML}{d9534f}
\definecolor{quarto-callout-warning-color-frame}{HTML}{f0ad4e}
\definecolor{quarto-callout-tip-color-frame}{HTML}{02b875}
\definecolor{quarto-callout-caution-color-frame}{HTML}{fd7e14}
\makeatother
\makeatletter
\@ifpackageloaded{caption}{}{\usepackage{caption}}
\AtBeginDocument{%
\ifdefined\contentsname
  \renewcommand*\contentsname{Table of contents}
\else
  \newcommand\contentsname{Table of contents}
\fi
\ifdefined\listfigurename
  \renewcommand*\listfigurename{List of Figures}
\else
  \newcommand\listfigurename{List of Figures}
\fi
\ifdefined\listtablename
  \renewcommand*\listtablename{List of Tables}
\else
  \newcommand\listtablename{List of Tables}
\fi
\ifdefined\figurename
  \renewcommand*\figurename{Figure}
\else
  \newcommand\figurename{Figure}
\fi
\ifdefined\tablename
  \renewcommand*\tablename{Table}
\else
  \newcommand\tablename{Table}
\fi
}
\@ifpackageloaded{float}{}{\usepackage{float}}
\floatstyle{ruled}
\@ifundefined{c@chapter}{\newfloat{codelisting}{h}{lop}}{\newfloat{codelisting}{h}{lop}[chapter]}
\floatname{codelisting}{Listing}

\usepackage{amsthm}
\theoremstyle{plain}
\newtheorem{proposition}{Proposition}[section]
\theoremstyle{plain}
\newtheorem{theorem}{Theorem}[section]
\theoremstyle{definition}
\newtheorem{example}{Example}[section]
\theoremstyle{plain}
\newtheorem{corollary}{Corollary}[section]
\theoremstyle{definition}
\newtheorem{definition}{Definition}[section]
\theoremstyle{remark}
\AtBeginDocument{}

\newtheorem{refremark}{Remark}[section]

\makeatother
\makeatletter
\@ifpackageloaded{caption}{}{\usepackage{caption}}
\@ifpackageloaded{subcaption}{}{\usepackage{subcaption}}
\makeatother
\makeatletter
\@ifpackageloaded{algorithm}{}{\usepackage{algorithm}}
\makeatother
\makeatletter
\@ifpackageloaded{algpseudocode}{}{\usepackage{algpseudocode}}
\makeatother
\makeatletter
\@ifpackageloaded{caption}{}{\usepackage{caption}}
\makeatother
    \makeatletter
    \newcommand\fs@nocaption{
      \def\@fs@cfont{\bfseries}
      \let\@fs@capt\floatc@plain
      \def\@fs@pre{}%
      \def\@fs@post{\kern2pt\hrule}%
      \def\@fs@mid{\hrule\kern2pt}%
      \let\@fs@iftopcapt\iftrue}
    \makeatother
  
\newcounter{quartocallouttipno}
\newcommand{\quartocallouttip}[1]{\refstepcounter{quartocallouttipno}\label{#1}}
\newcounter{quartocalloutcauno}
\newcommand{\quartocalloutcau}[1]{\refstepcounter{quartocalloutcauno}\label{#1}}
\ifLuaTeX
  \usepackage{selnolig}  % disable illegal ligatures
\fi
\usepackage[round]{natbib}
\usepackage{bookmark}

\IfFileExists{xurl.sty}{\usepackage{xurl}}{} % add URL line breaks if available
\hypersetup{
  pdftitle={Fast confidence bounds for the false discovery proportion over a path of hypotheses},
  pdfauthor={Guillermo Durand},
  pdfkeywords={multiple testing, algorithmic, post hoc inference, false
discovery proportion, confidence bound},
  colorlinks=true,
  linkcolor={blue},
  filecolor={Maroon},
  citecolor={Blue},
  urlcolor={Blue},
  pdfcreator={LaTeX via pandoc}}

\title{Fast confidence bounds for the false discovery proportion over a
path of hypotheses}
\author{Guillermo Durand}
\date{2025-10-09}

\begin{document}
\definecolor{computo-blue}{HTML}{034E79}

\begin{tikzpicture}[remember picture,overlay]
\fill[computo-blue]
  (current page.north west) -- (current page.north east) --
  ([yshift=-5cm]current page.north east|-current page.north east) --
  ([yshift=-5cm]current page.north west|-current page.north west) -- cycle;
\node[anchor=north west, xshift=.75cm,
  yshift=-.75cm] at (current page.north west) {\includegraphics[height=3cm]{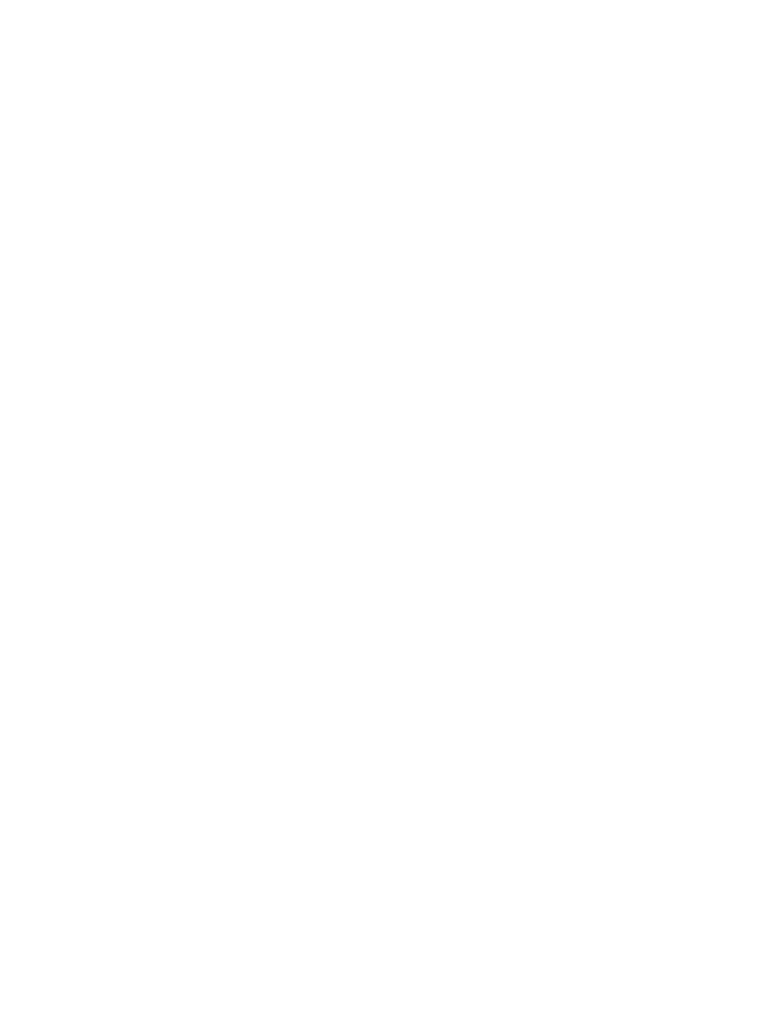}};
\node[font=\sffamily\bfseries\color{white},anchor=north west, xshift=.75cm,
  yshift=-4.25cm] at (current page.north
  west) {\fontsize{10}{12}\selectfont ISSN 2824-7795};
\node[font=\sffamily\bfseries\color{white},anchor=west,
  xshift=4.25cm,yshift=-2.75cm] at (current page.north west)
  {\begin{minipage}{15cm}
    \fontsize{25}{30}\selectfont
    Fast confidence bounds for the false discovery proportion over a
    path of hypotheses
    \vspace{.5cm}
    \\
    \fontsize{15}{18}\selectfont
    
  \end{minipage}};
\end{tikzpicture}

\vspace*{2.5cm}
\begin{center}
          Guillermo
Durand~\orcidlink{0000-0003-4056-5631}\footnote{Corresponding author: \href{mailto:guillermo.durand@universite-paris-saclay.fr}{guillermo.durand@universite-paris-saclay.fr}}\quad
             \href{https://www.universite-paris-saclay.fr/}{Université
Paris-Saclay}, \href{https://www.cnrs.fr}{CNRS},
\href{https://team.inria.fr/celeste/}{Inria},
\href{https://www.imo.universite-paris-saclay.fr}{Laboratoire de
Mathématiques d'Orsay}, 91405, Orsay, France\\
           
  \bigskip
  
  Date published: 2025-10-09 \quad Last modified: 2026-03-06
\end{center}
      
\bigskip
\begin{abstract}
This paper presents a new algorithm (and an additional trick) that
allows to compute fastly an entire curve of post hoc bounds for the
False Discovery Proportion when the underlying bound
\(V^*_{\mathfrak{R}}\) construction is based on a reference family
\(\mathfrak{R}\) with a forest structure à la \citet{MR4178188}. By an
entire curve, we mean the values
\(V^*_{\mathfrak{R}}(S_1),\dotsc,V^*_{\mathfrak{R}}(S_m)\) computed on a
path of increasing selection sets \(S_1\subsetneq\dotsb\subsetneq S_m\),
\(|S_t|=t\). The new algorithm leverages the fact that going from
\(S_t\) to \(S_{t+1}\) is done by adding only one hypothesis. Compared
to a more naive approach, the new algorithm has a complexity in
\(O(|\mathcal K|m)\) instead of \(O(|\mathcal K|m^2)\), where
\(|\mathcal K|\) is the cardinality of the family.
\end{abstract}

\noindent%
\vfill

\DraftwatermarkOptions{stamp=false}

\floatname{algorithm}{Algorithm}

\renewcommand*\contentsname{Contents}
{
\hypersetup{linkcolor=}
\setcounter{tocdepth}{3}
\tableofcontents
}
\newcommand{\comp}[1]{{#1}^{\mathsf{c}}}
 \newcommand{\Pro}[1]{\mathbb{P}\left(#1\right)} 
 \newcommand{\Esp}[1]{\mathbb{E}\left[ #1 \right]}
 \newcommand{\ind}[1]{\mathbb{1}_{\left\{#1 \right\}}}
 \newcommand{\cH}{\mathcal{H}}
 \newcommand{\cK}{\mathcal{K}}
 \newcommand{\cP}{\mathcal{P}}
 \newcommand{\FDP}{\mathrm{FDP}}
 \newcommand{\FDR}{\mathrm{FDR}}
 \newcommand{\JER}{\mathrm{JER}}
 \newcommand{\Rfam}{\mathfrak{R}}
 \newcommand{\statfam}{\mathfrak{F}}
 \newcommand{\Hoi}{H_{0,i}}
 \newcommand{\Vhat}{\widehat V}
 \newcommand{\Vstar}{V^*_{\Rfam}}
 \newcommand{\Nm}{\mathbb{N}_m}
 \newcommand{\pr}{\mathfrak{pr}}
 \newcommand{\kth}[2]{k^{(#1,#2)}}
 \newcommand{\RR}{\mathbb{R}}
 \newcommand{\telque}{\,:\,}

\section{Introduction}\label{introduction}

Multiple testing theory is often used for exploratory analysis, like in
Genome-Wide Association Studies, where multiple features are tested to
find promising ones. Classical multiple testing theory like Family-Wise
Error Rate (FWER) control or False Discovery Rate (FDR) control
\citep{MR1325392} can be used, but a more recent trend consists in the
computation of confidence upper bounds for the number of false
discoveries, or, equivalently, for the False Discovery Proportion (FDP).
This approach is notably advocated in the context of exploratory
research by \citep[Section 1]{MR2951390}.

Mathematically speaking, assume that we observe some data \(X\) that is
formally a random variable defined on some probability space equipped of
the probability measure \(\mathbb P\), and that the distribution of
\(X\), denoted by \(P\), belongs to a model \(\mathfrak{F}\). We want to
test \(m\) null hypotheses
\(H_{0,1},\dotsc,H_{0,m} \subset \mathfrak{F}\). A confidence upper
bound (also frequently named post hoc bound, post selection bound or
confidence envelope) is then a function
\(\widehat V: \mathcal{P}(\mathbb{N}_m^*) \to \mathbb{N}_m\), where
\(\mathbb{N}_m=\{0,\dotsc,m\}\), \(\mathbb{N}_m^*=\{1,\dotsc,m\}\), such
that \begin{equation}
\forall \alpha \in (0,1),\; \mathbb{P}\left(\forall S \subseteq \mathbb{N}_m^*, |S\cap \mathcal{H}_0|\leq \widehat V(S)\right)\geq 1-\alpha.
\label{eq_confidence}
\end{equation} Here, \(\alpha\) is a target error rate and
\(\mathcal{H}_0=\{i: P\in H_{0,i}\}\) is the set of indices of the null
hypotheses that are true. Note that the construction of \(\widehat V\)
depends on \(\alpha\) and on the data \(X\), and the dependence is
omitted to lighten notation and because there is no ambiguity. The
meaning of Equation \eqref{eq_confidence} is that \(\widehat V\)
provides an upper bound for the number of null hypotheses in \(S\) for
any selection set \(S\subseteq \mathbb{N}_m^*\), that is, the number
\(|S\cap \mathcal{H}_0|\) of false discoveries in \(S\). This allows the
user to perform post hoc selection on their data without breaching the
statistical guarantee. Also note that by dividing by \(|S|\vee 1\) in
Equation \eqref{eq_confidence} we also get a confidence bound for the
FDP: \begin{equation}
\forall \alpha \in (0,1),\; \mathbb{P}\left(\forall S \subseteq \mathbb{N}_m^*, \mathrm{FDP}(S)\leq \frac{\widehat V(S)}{|S|\vee 1}\right)\geq 1-\alpha.
\label{eq_confidence_fdp}
\end{equation}

So post hoc bounds provide ways to construct FDP-controlling sets
instead of FDR-controlling sets, which is much more desirable given the
nature of the FDR as an expected value. See for example \citep[Figure
4]{MR3418717} for a credible example where the FDR is controlled but the
FDP has a highly undesirable behavior (either 0 because no discoveries
at all are made, either higher than the target level). The construction
is the following: one can compute the largest \(S\) such that
\(\frac{\widehat V(S)}{|S|\vee 1}\) is less than or equal to a nominal
level \(q\), and \eqref{eq_confidence_fdp} ensures that, with high
probability, the FDP of \(S\) is upper-bounded by \(q\). The control of
the FDP with high probability is sometimes called False Discovery
Exceedance (FDX) control.

Post hoc bounds have notably been applied to genetic data. For example,
in \citet{MR2951390} and \citet{10.1093/bioinformatics/btac693}, the
authors apply post hoc bounds to an Urothelial Bladder Carcinoma RNA
sequencing dataset to detect genes differentially expressed between
stage II and stage III of the disease. Furthermore, the \texttt{R}
\citep{R-base} package \texttt{IIDEA} (\citet{IIDEA}, see also
\citet{enjalbertcourrech:tel-05034928}, Chapter 3) implements a
user-friendly \texttt{shiny} application \citep{shiny} that computes
post hoc bounds for differential expression analyses, where the user can
upload their own microarray or bulk RNAseq data file (the application
also comes with the aforementioned dataset).

Another field where post hoc bounds have been successfully applied is
functional Magnetic Resonance Imaging (fMRI) studies, where each voxel
of an image is tested to detect activation of the corresponding brain
region during a given task. Using the aforementioned FDP-controlling
construction, in \citet{blain22notip} and \citet{NEURIPS2023_f6712d51},
the authors construct rejection sets with a high number of true
positives.

The first confidence bounds are found in \citet{MR2279468} and
\citet{MR2279639}, although, in the latter, only for selection sets of
the form \(\{i\in\mathbb{N}_m: p_i\leq s\}\) where \(p_i\) is the
\(p\)-value associated to the null hypothesis \(H_{0,i}\) and
\(s\in[0,1]\) is a threshold. In \citet{MR2951390} the authors re-wrote
the generic construction of \citet{MR2279468} in terms of closed testing
(a framework first introduced for the FWER control by \citet{MR468056}),
proposed several practical constructions and sparked a new interest in
multiple testing procedures based on confidence bounds. This work was
followed by a prolific series of works like \citet{MR3305943} or
\citet{MR4731977}. In \citet{MR4124323}, the authors introduce the new
point of view of references families to construct post hoc bounds, and
show the links between this meta-technique and the closed testing one,
along with new bounds. Reference families are families of couples
\((R_k,\zeta_k)_{k\in\mathcal{K}}\) where \(R_k\) is a subset of
hypotheses (called a region), and \(\zeta_k\) an over-estimator of the
number of null hypotheses inside \(R_k\), that is, of
\(|R_k\cap\mathcal{H}_0|\). From a statistical guarantee on the family,
called the Joint Error Rate (JER) control, one is able to build a post
hoc bound, denoted \(V^*_{\mathfrak{R}}\) in the remainder, by
interpolation (see Section~\ref{sec-reference-fam} for all the details).

Following the reference family trail, in \citet{MR4178188}, the authors
introduce new reference families with a special set-theoretic constraint
that allows an efficient computation of the bound
\(V^*_{\mathfrak{R}}(S)\) for a given, single selection set \(S\). The
constraint, named ``forest structure'', is that two regions of
hypotheses \(R_k\) and \(R_{k'}\) are either disjoint, or nested:
\(R_k\cap R_{k'}\in\{R_k,R_{k'},\varnothing\}\). This structure arises
when the object of study naturally presents different levels of
hierarchy. For example, in genomic studies, where each hypothesis tests
the association of a Single Nucleotid Polymorphism (SNP) with a given
character, we can exploit the natural grouping of SNPs into genes or
intergenic regions, and then the grouping of genes into genomic
pathways, or into chromosomes. In proteomic studies, where the smallest
unit is usually the peptide, we can exploit the natural grouping of
peptides into proteins, and the grouping of proteins into proteomic
pathways. In brain imagery, known brain anatomy can be used to build the
regions.

The problem is that one often wants to compute \(V^*_{\mathfrak{R}}\) on
a whole path of selection sets \((S_t)_{t\in\mathbb{N}_m^*}\), for
example the hypotheses attached to the \(t\) smallest \(p\)-values:
\(S_t=\{\sigma(1),\dots,\sigma(t)\}\), where \(\sigma\) is a (random)
permutation ordering the \(p\)-values:
\(p_{\sigma(1)}\leq\dotsb\leq p_{\sigma(m)}\). Whereas the algorithm
provided in the aforementioned work \citep[Algorithm 1]{MR4178188},
which is reproduced here (see Algorithm~\ref{algo-vstar}) is fast for a
single evaluation, it is slow and inefficient to repeatedly call it to
compute each \(V^*_{\mathfrak{R}}(S_t)\). If the \(S_t\)'s are nested,
and growing by one, that is \(S_1\subsetneq\dotsb\subsetneq S_m\) and
\(|S_t|=t\), there is a way to efficiently compute
\(\left(V^*_{\mathfrak{R}}(S_t)\right)_{t\in\mathbb{N}_m}\) by
leveraging the nested structure.

This is the main contribution of the present paper: a new and fast
algorithm (Algorithm~\ref{algo-formal-curve}) computing the curve
\(\left(V^*_{\mathfrak{R}}(S_t)\right)_{t\in\mathbb{N}_m}\) for a nested
path of selection sets, that is presented in
Section~\ref{sec-fast-curve}. An additional pruning algorithm, that can
speed up computations both for the single-evaluation algorithm and the
new curve-evaluation algorithm, is also presented in
Section~\ref{sec-pruning}. Notably, a detailed example illustrating how
the new algorithms work is provided in Section~\ref{sec-example}. In
Section~\ref{sec-notation}, all necessary notation and vocabulary is
re-introduced, most of it being the same as in \citet{MR4178188}. In
Section~\ref{sec-implementation}, we discuss the current implementations
of all the presented algorithms in the \texttt{R} \citep{R-base} package
\texttt{sanssouci} \citep{sanssouci}, with an example code. A few
numerical experiments are presented in Section~\ref{sec-numeric} to
demonstrate the computation time gain. We reproduce here, in
Figure~\ref{fig-benchmark_intro}, the striking results of one of those
experiments, where the combination of the two new algorithms improves
the computation time by a factor \(33000\). Finally, after some
concluding remarks in Section~\ref{sec-conclusion}, the proofs of all
results, including the proof that Algorithm~\ref{algo-formal-curve}
indeed computes correctly the curve, are presented in
Section~\ref{sec-proofs}.

\begin{figure}

\centering{

\includegraphics{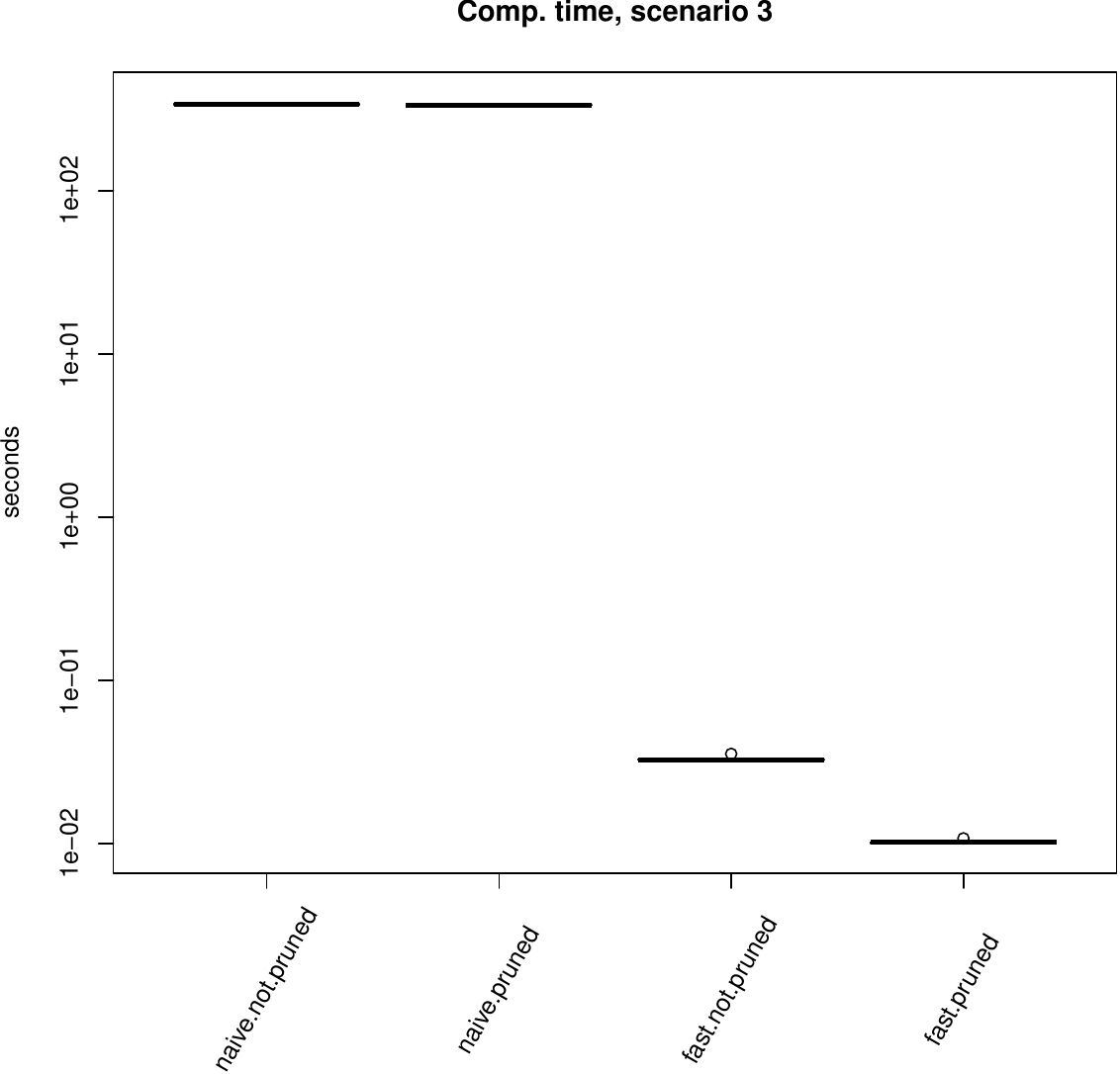}

}

\caption{\label{fig-benchmark_intro}Computation time in scenario 3 for
the new fast algorithm versus the previous, naive approach, in seconds
(using a logarithmic scale)}

\end{figure}%

\section{Notation and reference family methodology}\label{sec-notation}

\subsection{Multiple testing notation}\label{multiple-testing-notation}

As is usual in multiple testing theory, we consider a probability space
\((\Omega,\mathcal A, \mathbb P)\), a model \(\mathfrak{F}\) on a
measurable space \((\mathcal{X},\mathfrak{X})\), and data that is
represented by a random variable
\(X:(\Omega,\mathcal A)\to(\mathcal{X},\mathfrak{X})\) with
\(X\sim P\in \mathfrak{F}\), that is, the law of \(X\) is comprised in
the model \(\mathfrak{F}\).

Then we consider \(m\geq1\) null hypotheses \(H_{0,1}, \dotsc, H_{0,m}\)
which formally are submodels, that is subsets of \(\mathfrak{F}\). The
associated alternative hypotheses \(H_{1,1}, \dotsc, H_{1,m}\) are
submodels such that \(H_{0,i}\cap H_{1,i}=\varnothing\) for all
\(i\in\mathbb{N}_m^*\). We denote by \(\mathcal{H}_0=\mathcal{H}_0(P)\)
(the dependence in \(P\) will be dropped when there is no ambiguity) the
set of all null hypotheses that are true, that is
\(\mathcal{H}_0(P)=\{i\in\mathbb{N}_m^* : P\in H_{0,i}\}\). In other
words, \(H_{0,i}\) is true if and only if \(i\in\mathcal{H}_0\). For
testing each \(H_{0,i}, i\in\mathbb{N}_m^*\), we have at hand a
\(p\)-value \(p_i=p_i(X)\) (the dependence in \(X\) will be dropped when
there is no ambiguity) which is a random variable with the following
property : if \(i\in\mathcal{H}_0\), then the law of \(p_i\) is
super-uniform, which is sometimes denoted
\(\mathcal L(p_i)\succeq \mathcal{U}([0,1])\). This means that in such
case, the cumulative distribution function (cdf) of \(p_i\) is always
smaller than or equal to the cdf of a random variable
\(U\sim \mathcal{U}([0,1])\) : \begin{equation}
\forall x \in \mathbb{R},\; \mathbb{P}\left(p_i\leq x\right)\leq \mathbb{P}\left(U\leq x\right) = 0\vee(x\wedge 1).
\label{eq_super_unif}
\end{equation}

For every subset of hypotheses \(S\subseteq\mathbb{N}_m^*\), let
\(V(S)=|S\cap\mathcal{H}_0|\). If we think of \(S\) as a selection set
of hypotheses deemed significant, \(V(S)\) is then the number of false
discoveries, or false positives, in \(S\). \(V(S)\) is our main object
of interest and the quantity that we wish to over-estimate with
confidence upper bounds (see Equation \eqref{eq_confidence} or the more
formal Equation \eqref{eq_confidence_formal} below).

Finally let us consider the following toy example, that will re-appear
slightly simplified in Section~\ref{sec-numeric}.

\begin{example}[Gaussian
one-sided]\protect\hypertarget{exm-gauss}{}\label{exm-gauss}

In this case we assume that \(X=(X_1,\dotsc,X_m)\) is a Gaussian vector
and the null hypotheses refer to the nullity of the means in contrast to
their positivity. That is, formally,
\((\mathcal{X},\mathfrak{X})=(\mathbb R^m, \mathcal B\left(\mathbb R^m  \right))\),
\(\mathcal P=\{ \mathcal N(\boldsymbol{\mu}, \Sigma) : \forall j \in\mathbb{N}_m^*, \mu_j\geq 0, \Sigma \text{ positive semidefinite}  \}\),
for each \(i\in\mathbb{N}_m^*\),
\(H_{0,i}= \{ \mathcal N(\boldsymbol{\mu}, \Sigma) \in \mathcal P :\mu_i=0 \}\)
and
\(H_{1,i}=\{ \mathcal N(\boldsymbol{\mu}, \Sigma) \in \mathcal P :\mu_i>0 \}\).
Then we can construct \(p\)-values by letting \(p_i(X)=1-\Phi(X_i)\),
where \(\Phi\) denotes the cdf of \(\mathcal N(0,1)\).

\end{example}

\subsection{Post hoc bounds with reference
families}\label{sec-reference-fam}

With the formalism introduced in last section, a confidence upper bound
is a functional
\(\widehat V:\mathcal X\times (0,1)\to(\mathcal P(\mathbb{N}_m^*) \to \mathbb{N}_m)\)
such that, \begin{equation}
\forall P\in\mathcal P, \forall X\sim P, \forall \alpha \in (0,1),\; \mathbb{P}\left(\forall S \subseteq \mathbb{N}_m^*, V(S)\leq \widehat V(X,\alpha)(S)\right)\geq 1-\alpha.
\label{eq_confidence_formal}
\end{equation} In the remainder, the dependence in \((X,\alpha)\) will
be dropped when there is no ambiguity and
\(\widehat V(X,\alpha)(\cdot)\) will simply be written \(\widehat V\).

As said in the Introduction, many constructions, ultimately
theoretically equivalent but differing by the practical steps involved,
exist, and in this paper we focus on the meta-construction of
\citet{MR4124323} based on reference families. A reference family is a
finite family
\(\mathfrak{R}=\mathfrak{R}(X,\alpha)=(R_k,\zeta_k)_{k\in \mathcal K}\)
with \(R_k\subseteq\mathbb{N}_m^*\),
\(\zeta_k\in\left\{0,\dotsc,|R_k|\right\}\) where everything (that is,
\(\mathcal K\) and all the \(R_k\) and \(\zeta_k\)) depends on
\((X,\alpha)\) but the dependency is not explicitly written. The \(R_k\)
are assumed all distinct almost surely (see Remark~\ref{rem-distinct}).

The intuition of the concept of reference families is the following. The
\(R_k\)'s are subsets, or regions, of hypotheses, and the associated
\(\zeta_k\)'s are over-estimators of \(V(R_k)\). That is, building a
reference family amounts to building a collection of regions, that can
be much smaller than all possible subsets of hypotheses (note that
\(\left|\mathcal P(\mathbb{N}_m^*)\right|=2^m\) is likely to be very
large), for which we have a confidence upper bound. This bound, holding
only on the \(R_k\)'s in the first place, will then be extended to a
simultaneous confidence bound over all subsets \(S\) (in the sense of
Equation \eqref{eq_confidence_formal}) by an interpolation scheme
explained below.

The statistical guarantee over the \(\zeta_k\)'s, as over-estimators of
\(V(R_k)\), is written in terms of the following error criterion for a
reference family, named Joint Error Rate (JER): \begin{equation}
\mathrm{JER}(\mathfrak{R}) = \mathbb{P}\left(\exists k\in\mathcal K, |R_k\cap\mathcal{H}_0| > \zeta_k \right) = \mathbb{P}\left(\exists k\in\mathcal K, V(R_k) > \zeta_k \right).
\label{eq_jer}
\end{equation} We say that the reference family \(\mathfrak{R}\)
controls the JER if the following is true: \begin{equation}
\forall P\in\mathcal P, \forall X\sim P, \forall \alpha \in (0,1),\; 1-\mathrm{JER}(\mathfrak{R}(X,\alpha))=\mathbb{P}\left(\forall k\in\mathcal K, V(R_k)\leq \zeta_k\right) \geq 1-\alpha.
\label{eq_jer_control}
\end{equation} Note that Equation \eqref{eq_jer_control} is, as
foretold, really similar to Equation \eqref{eq_confidence_formal} except
that the uniform guarantee, instead of being over all
\(S\subseteq \mathbb{N}_m^*\), is only over all the
\(R_k\subseteq \mathbb{N}_m^*, k\in\mathcal K\), with \(\mathcal K\)
having cardinality potentially much smaller than \(2^m\). A ``global''
confidence bound is then derived from the JER-controlling reference
family with the following two steps. First let \begin{equation}
\mathcal A(\mathfrak{R})= \left\{A\subseteq \mathbb{N}_m^*:  \forall k\in\mathcal K, |R_k\cap A| \leq \zeta_k \right\}.
\label{eq_a}
\end{equation} The JER control says that, with high probability,
\(\mathcal{H}_0\in\mathcal A(\mathfrak{R})\). We then leverage this
information by interpolation, with the following construction:
\begin{equation}
V^*_{\mathfrak{R}}(S) = \max_{A\in\mathcal A(\mathfrak{R})}|S\cap A|.
\label{eq-vstar}
\end{equation} By Proposition 2.1 of \citet{MR4124323}, the JER control
of the family in Equation \eqref{eq_jer_control} implies that
\(V^*_{\mathfrak{R}}\) is indeed a confidence bound as required by
Equation \eqref{eq_confidence_formal}. The same Proposition also
establishes that \(V^*_{\mathfrak{R}}\) optimally uses the information
provided by the JER control of the reference family.

Note that, because of the \(\max_{A\in\mathcal A(\mathfrak{R})}\), the
computation of \(V^*_{\mathfrak{R}}(S)\) is generally intractable (see
Proposition 2.2 of \citet{MR4124323}), but for specific structures of
reference families, a polynomial computation can be derived. This is the
topic of \citet{MR4178188} and of the remainder of this paper.

\begin{refremark}
The specific computation of the \(R_k\)'s and the \(\zeta_k\)'s such
that Equation \eqref{eq_jer_control} holds is outside the scope of the
present paper, but different constructions can be found in
\citet{MR4124323}, \citet{MR4178188}, \citet{blain22notip} or
\citet{JMLR:v25:23-1025}, for example.

\label{rem-zeta}

\end{refremark}

\begin{refremark}
Some reference family constructions can yield \(R_k=R_{k'}\) for
\(k\neq k'\), for example in the setting of \citet{JMLR:v25:23-1025}
with discrete \(p\)-values. But in that scenario we can always prune the
duplicate and keep only one index so that \(k\mapsto R_k\) is injective.
We implicitly consider that this operation is always done in practice
and in the remainder of this article. Of course, if \(R_k=R_{k'}\) with
\(k\neq k'\), we keep the index with the lower value of \(\zeta\), that
is we keep \(\tilde k\in\arg\min_{\ell\in\{k,k'\}}\zeta_{\ell}\) (not
doing so would change the bound defined by \eqref{eq-vstar} and decrease
its power in terms of type-II error). Similarly, some constructions can
yield empty regions, which can always be pruned without changing the
bound. This will also be assumed to be the case in the following.
Finally, note that the constraint that \(\zeta_k\leq |R_k|\) always hold
no matter how \(\zeta_k\) was computed, up to replacing \(\zeta_k\) by
\(\zeta_k\wedge |R_k|\): it is clear on Equation \eqref{eq_a} that this
doesn't change the bound.

\label{rem-distinct}

\end{refremark}

\subsection{Regions with a forest structure}\label{sec-forest-structure}

The core concept of this section is to assume that the regions \(R_k\)'s
of the reference family are what we called in \citet{MR4178188} a forest
structure, that is two regions are either disjoint or nested:
\begin{equation}
\forall k,k'\in\mathcal{K} , R_k \cap R_{k'} \in \{ R_k,  R_{k'} , \varnothing \}.
\label{eq-forest}
\end{equation} Representing the \(R_k\)'s with a directed graph, where
there is an oriented edge \(R_k \leftarrow R_{k'}\) if and only if
\(R_k \subseteq R_{k'}\) and there is no \(R_{k''}\) such that
\(R_k \subsetneq R_{k''}\subsetneq R_{k'}\) gives a forest, hence the
name. See Example~\ref{exm-toy-forest} and its representation in
Figure~\ref{fig-forest-exm}.

We also need to introduce the notion of depth with the following
function: \begin{equation}
\phi \:  : \: \left\{
\begin{array}{l  c l  }
 \mathcal{K}& \to & \mathbb{N}^*\\
k & \mapsto & 1 + \left| \{k'\in\mathcal{K}: R_k\subsetneq R_{k'} \} \right|   .
\end{array}
\right.
\label{eq-depth}
\end{equation} This definition matches the intuition of depth because we
assumed the \(R_k\) are distinct, see Remark~\ref{rem-distinct}.

In all the remainder, \(H\) refers to the maximum depth in the
structure: \(H=\max_{k\in\mathcal{K}}\phi(k)\).

\begin{example}[]\protect\hypertarget{exm-toy-forest}{}\label{exm-toy-forest}

Let \(m=25\), \(R_1 = \{1, \dotsc , 20 \}\), \(R_2  =  \{1, 2  \}\),
\(R_3   =   \{3 , \dotsc , 10 \}\), \(R_4  =    \{11, \dotsc , 20 \}\),
\(R_5 =  \{5, \dotsc , 10 \}\), \(R_6   =     \{11, \dotsc , 16 \}\),
\(R_7  =   \{17, \dotsc ,20  \}\), \(R_8=\{21,22\}\), \(R_9 = \{22\}\).
This is the same example as Example 2 of \citet{MR4178188} and it is
graphically depicted in Figure~\ref{fig-forest-exm}. The sets \(R_1\),
\(R_8\) are of depth \(1\); the sets \(R_2,R_3,R_4,R_9\) are of depth
\(2\); the sets \(R_5,R_6,R_7\) are of depth \(3\).

\end{example}

\begin{figure}

\centering{

\includegraphics{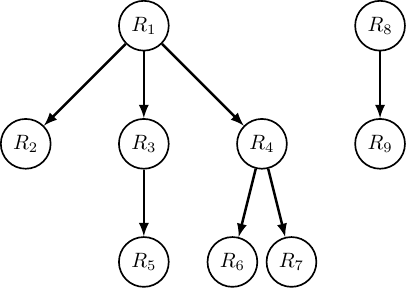}

}

\caption{\label{fig-forest-exm}The regions of
Example~\ref{exm-toy-forest}.}

\end{figure}%

Another tool of \citet{MR4178188} that will be used is its Lemma 2, that
is the identification of \(\mathfrak{R}\) with a set
\(\mathcal C\subset \left\{(i,j)\in \left(\mathbb N_N^*\right)^2 \: : i\leq j\right\}\)
such that for \((i,j), (i',j')\in\mathcal C\),
\(\{i,\dotsc, j\}\cap\{i',\dotsc,j'\}\in\left\{\varnothing, \{i,\dotsc, j\},\{i',\dotsc j'\}  \right\}\).
With this identification, each \(R_k=R_{(i,j)}\) can be written as
\(P_{i:j}=\bigcup_{i\leq n\leq j}P_n\) where \((P_n)_{1\leq n \leq N}\)
is a partition of \(\mathbb{N}_m^*\). The \(P_n\)'s were called atoms in
\citet{MR4178188} because they have the thinnest granularity in the
structure, but to continue the analogy with graphs, forests and trees,
they can also be called leafs. See Example~\ref{exm-toy-leaves} for a
concrete example.

\begin{example}[Continuation of
Example~\ref{exm-toy-forest}]\protect\hypertarget{exm-toy-leaves}{}\label{exm-toy-leaves}

For the reference family given in Example~\ref{exm-toy-forest}, a
partition of atoms is given by \(P_1 =R_2\),
\(P_2  =   R_3\setminus R_5\), \(P_3  =   R_5\), \(P_4=R_6\),
\(P_5=R_7\), \(P_6=R_8\setminus R_9\), \(P_7=R_9\),
\(P_8=\mathbb{N}_m^* \setminus \{R_1 \cup R_8 \}\). Then
\(R_1=P_{1:5}\), \(R_3=P_{2:3}\), \(R_4=P_{4:5}\) and \(R_8=P_{6:7}\).
Note that not all atoms are regions of the family. Those new labels are
graphically depicted in Figure~\ref{fig-leaves-exm}. The nodes that
correspond to atoms that are not in the family are depicted with a
dashed circle, and all atoms are depicted in gray. This is the same
example as Example 3 of \citet{MR4178188}.

\end{example}

\begin{figure}

\centering{

\includegraphics{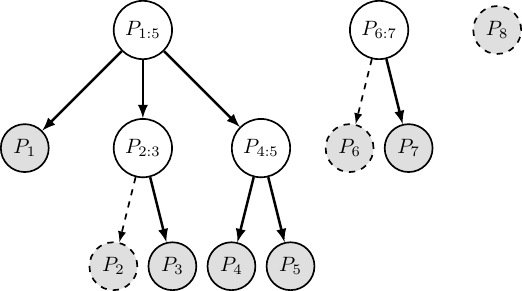}

}

\caption{\label{fig-leaves-exm}The regions of
Example~\ref{exm-toy-forest} but with the labels of
Example~\ref{exm-toy-leaves}.}

\end{figure}%

When all leaves are regions of the family, it is said that the family is
complete. If this is not the case, the family can easily be completed by
adding the missing leaves (and using their cardinality as associated
\(\zeta\)) without changing the value \(V^*_{\mathfrak{R}}\). See
Definition 2, Lemma 6 and Algorithm 2 of \citet{MR4178188} for the
details.

\citet{MR4178188} also proved in their Theorem 1 that: \begin{equation}
V^*_{\mathfrak{R}}(S)=\min_{Q\subseteq\mathcal{K}}\left(\sum_{k'\in Q}\zeta_{k'}\wedge|S\cap R_{k'}|+\left| S\setminus\bigcup_{k'\in Q} R_{k'}   \right|\right)
\label{eq_vstar_Q} 
\end{equation} and, even better, in their Corollary 1 \emph{(iii)} that:
\begin{equation}
V^*_{\mathfrak{R}}(S) = \min_{Q\in \mathfrak P}\sum_{k'\in Q}\zeta_{k'}\wedge|S\cap R_{k'}|,
\label{eq_vstar_Qpartition}
\end{equation} provided that the family is complete. Here,
\(\mathfrak P \subseteq \mathcal P(\mathcal{K})\) is the set of subsets
of \(\mathcal{K}\) that realize a partition, that is, the set of
elements \(Q\subseteq\mathcal{K}\) such that the \(R_k\), \(k\in Q\),
form a partition of \(\mathbb{N}_m^*\). So the minimum in Equation
\eqref{eq_vstar_Qpartition} is over way less elements than in Equation
\eqref{eq_vstar_Q}.

Finally, that paper provides a polynomial algorithm to
\(V^*_{\mathfrak{R}}(S)\) for a single \(S\subseteq\mathbb{N}_m^*\),
which we reproduce here in Algorithm~\ref{algo-vstar}. The family is
assumed complete, otherwise the first step would be to complete it. In
the original paper, \(\mathcal{K}^h\) used to designate the elements of
\(\mathcal{K}\) at depth \(h\) plus the atoms at depth \(\leq h\).
Actually, including those atoms is not needed for this algorithm to
perform exactly the same, and produces redundant computations. If we
don't include them, the only difference is that sometimes \(Succ_k\) can
be empty, in which case we simply let
\(newVec_k=\zeta_k\wedge|S\cap R_k|\). Thus, here in this paper, we
define \(\mathcal{K}^h\) as only the elements of \(\mathcal{K}\) at
depth \(h\) (the previous intricate definition may still be necessary
for the proof of Theorem 1 of \citet{MR4178188}):
\(\mathcal{K}^h=\{ (i,j)\in\mathcal{K}: \phi(i,j)=h      \}, \:\:\:h\geq 1.\)
This is the only deviation from the notation of \citet{MR4178188}.
Finally note that in the ongoing analogy with graph theory, the elements
of \(\mathcal{K}^1\) are the roots of the different trees making up the
forest.

\floatstyle{ruled}
\restylefloat{algorithm}
\begin{algorithm}[htb!]
\caption{\label{algo-vstar}Computation of a given $V^*_{\mathfrak{R}}(S)$ with a complete family}
\begin{algorithmic}[1]
\Procedure{Vstar}{S, $\mathfrak{R}=(R_{k},\zeta_{k})_{k\in\mathcal{K}}$  with $\mathfrak{R}$ complete}
  \State $ H \gets \max_{k\in\mathcal{K}} \phi(k)  $ \Comment{maximum depth}
  \State $Vec\gets (\zeta_{k}\wedge|S\cap R_k|)_{k \in  \mathcal{K}^H}$ \Comment{initialization}
  \For{$h = H-1, \dotsc, 1$}
    \State $\mathcal{K}^h\gets \{ k\in\mathcal{K} : \phi(k) =h  \}$
    \State $newVec\gets (0)_{k \in  \mathcal{K}^h}$
    \For{$k \in  \mathcal{K}^h$}
      \State $Succ_k \gets \{ k' \in  \mathcal{K}^{h+1} : R_{k'}\subseteq R_k\}$
      \If{$Succ_k=\varnothing$}
        \State $newVec_k \gets \zeta_k\wedge|S\cap R_k|$
      \Else
        \State $newVec_k \gets \min\left( \zeta_{k}\wedge|S\cap R_k| ,  \sum_{k'\in Succ_k} Vec_{k'}   \right)$
      \EndIf
    \EndFor
    \State $Vec\gets newVec$
  \EndFor
  \State\Return $\sum_{k\in\mathcal{K}^1} Vec_k  $
\EndProcedure
\end{algorithmic}
\end{algorithm}
\floatstyle{plain}

A step by step description of Algorithm~\ref{algo-vstar} is provided at
the end of Section 3 and in the Figure 9 of \citet{MR4178188}.

The computation time of the algorithm is in \(O(|\mathcal{K}||S|)\),
which is fast for a single evaluation, but calling it repeatedly on a
path of selection sets \((S_t)_{t\in\mathbb{N}_m^*}\) has complexity
\(O(|\mathcal{K}|m^2)\), which is not desirable and makes computations
difficult in practice, hence the need for a new, faster algorithm.

\begin{tcolorbox}[enhanced jigsaw, left=2mm, coltitle=black, toptitle=1mm, colback=white, bottomtitle=1mm, breakable, title=\textcolor{quarto-callout-tip-color}{\faLightbulb}\hspace{0.5em}{Tip \ref*{tip-algo1}: Tip}, titlerule=0mm, arc=.35mm, opacityback=0, opacitybacktitle=0.6, colbacktitle=quarto-callout-tip-color!10!white, toprule=.15mm, bottomrule=.15mm, rightrule=.15mm, colframe=quarto-callout-tip-color-frame, leftrule=.75mm]

\quartocallouttip{tip-algo1} 

In the practical implementation of this algorithm (and of the following
Algorithm~\ref{algo-pruning}), \(Vec\) and \(newVec\) are always of size
\(N\) (the number of leaves) instead of the cardinality of
\(\mathcal{K}^h\). And the sum \(\sum_{k'\in Succ_k} Vec_{k'}\) is
really easy to compute: if
\(R_k= R_{(i_0,i_{p}-1)}= \bigcup_{ j=1}^{p} R_{(i_{ j-1}, i_{ j}-1)}=\bigcup_{i_0\leq n\leq i_{p}-1}P_n\in\mathcal{K}^h\)
for some \(p\geq2\), a strictly increasing sequence
\((i_0,\dotsc,i_{p})\) and
\(R_{(i_{ j-1}, i_{ j}-1)}\in\mathcal{K}^{h+1}\) for all
\(1\leq j\leq p\), then we simply sum \(Vec\) over the indices from
\(i_{0}\) to \(i_{p}-1\). After that, the computed quantity is set in
\(newVec\) at index \(i_0\). So actually computing \(Succ_k\) is not
needed and not done.

Furthermore, computing \(|S\cap R_k|\) for each \(k\) is not necessary,
it is sufficient to compute \(|S\cap P_n|\) for each leaf \(P_n\), which
can be done in \(O(N|S|)\).

By the two previous points, we can actually refine the complexity result
of Algorithm~\ref{algo-vstar}: it is in
\(O(N|S|+|\mathcal{K}|)\leq O(|\mathcal{K}||S|)\) because
\(N\leq|\mathcal{K}|\) for a complete forest, and if the forest is not
complete, the Algorithm 2 of \citet{MR4178188} constructs a partition
\((P_n)_{1\leq n\leq N}\) such that \(N\leq|\mathcal{K}|\), and so the
cardinality of the completed forest is \(\leq 2|\mathcal{K}|\).

\end{tcolorbox}

Speaking of complexity, we have the following result regarding \(m\),
\(|\mathcal{K}|\) and \(N\):

\begin{proposition}[]\protect\hypertarget{prp-cardinals}{}\label{prp-cardinals}

For any reference family \((R_k,\zeta_k)_{k\in\mathcal{K}}\) with a
forest structure, we have \(N\leq m\), \(H\leq N\),
\(|\mathcal{K}|\leq 2N-1\) and these three bounds can be achieved
simultaneously. In particular, \(|\mathcal{K}|\leq 2m-1\).

\end{proposition}

The proof of Proposition~\ref{prp-cardinals} is given in
Section~\ref{sec-cardinals}.

\section{New algorithms}\label{new-algorithms}

\subsection{Pruning the forest}\label{sec-pruning}

We remark the simple fact that if, for example,
\((1,1), (2,2), (1,2)\in\mathcal{K}\), and
\(\zeta_{(1,2)}\geq \zeta_{(1,1)}+\zeta_{(2,2)}\), then \(R_{(1,2)}\)
never contributes to the computation of any \(V^*_{\mathfrak{R}}(S)\)
and it could just be removed from \(\mathfrak{R}\). We now formalize and
prove this pruning scheme.

\begin{definition}[Pruning]\protect\hypertarget{def-pruning}{}\label{def-pruning}

We define by \(\mathcal{K}^{\mathfrak{pr}}\) (\(\mathcal{K}\) pruned)
the set of elements of \(\mathcal{K}\) from which we removed all
\((i,i')\) such that there exists \(p\geq2\) and integers
\(i_1,\dotsc,i_{p-1}\) such that, when setting \(i_0=i\) and
\(i_{p}=i'+1\), the sequence \((i_0,\dotsc,i_{p})\) is strictly
increasing, \((i_{j-1},i_{j}-1)\in\mathcal{K}\) for all
\(1\leq j\leq p\) and finally
\(\zeta_{(i,i')}=\zeta_{(i_0,i_{p}-1)}\geq \sum_{j=1}^{p} \zeta_{(i_{j-1}, i_{j}-1)}\).

\end{definition}

An important note is that for a removed
\((i,i')\in\mathcal{K}\setminus\mathcal{K}^{\mathfrak{pr}}\), we can
always choose the indices \(i_1,\dotsc,i_{p-1}\) such that actually
\((i_j,i_{j+1}-1)\in\mathcal{K}^{\mathfrak{pr}}\) and not only
\(\mathcal{K}\), because if
\((i_j,i_{j+1}-1)\in\mathcal{K}\setminus\mathcal{K}^{\mathfrak{pr}}\) it
can itself be fragmented, and this decreasing recursion eventually ends
(the later possible being at the atoms of the forest structure). Also
note that removing elements from \(\mathcal{K}\) does not alter the fact
that we have at hand a forest structure, that is, the reference family
defined by
\(\mathfrak{R}^{\mathfrak{pr}}=(R_k,\zeta_k)_{k\in\mathcal{K}^{\mathfrak{pr}}}\)
has a forest structure. Because pruning a forest structure does not
touch the atoms, note finally that if \(\mathcal{K}\) is complete then
so is \(\mathcal{K}^{\mathfrak{pr}}\).

The following proposition states that pruning the forest does not alter
the bound.

\begin{proposition}[]\protect\hypertarget{prp-pruning}{}\label{prp-pruning}

For any \(S\subseteq \mathbb{N}_m^*\),
\(V^*_{\mathfrak{R}}(S)=V^*_{\mathfrak{R}^{\mathfrak{pr}}}(S)\).

\end{proposition}

The proof of Proposition~\ref{prp-pruning} is given in
Section~\ref{sec-pruning-proofs-pruning}.

This gives a practical way to speed up computations by first pruning the
family before computing any \(V^*_{\mathfrak{R}}(S)\), because
\(\mathcal{K}^{\mathfrak{pr}}\) is smaller than \(\mathcal{K}\), and by
the above Proposition there is no theoretical loss in doing so.

Furthermore, pruning can be done really simply by following
Algorithm~\ref{algo-vstar} for \(S=\mathbb{N}_m^*\), and pruning when
appropriate. This gives the following Algorithm~\ref{algo-pruning},
assuming, for simplicity, that the family is complete. Note that the
only differences between Algorithm~\ref{algo-pruning} and
Algorithm~\ref{algo-vstar} are the pruning step and \(\zeta_k\)
replacing \(\zeta_k\wedge|S\cap R_k|\), because \(\zeta_k\leq|R_k|\) and
\(S=\mathbb{N}_m^*\) here, so
\(\zeta_k\wedge|\mathbb{N}_m^*\cap R_k|=\zeta_k\).

\floatstyle{ruled}
\restylefloat{algorithm}
\begin{algorithm}[htb!]
\caption{\label{algo-pruning}Pruning of a complete $\mathfrak{R}$}
\begin{algorithmic}[1]
\Procedure{Pruning}{$\mathfrak{R}=(R_{k},\zeta_{k})_{k\in\mathcal{K}}$  with $\mathfrak{R}$ complete}
  \State $\mathcal{L}\gets\mathcal{K}$
  \State $ H \gets \max_{k\in\mathcal{K}} \phi(k)  $ \Comment{maximum depth}
  \For{$h = H-1, \dotsc, 1$}
    \State $\mathcal{K}^h\gets \{ k\in\mathcal{K} : \phi(k) =h  \}$
    \State $newVec\gets (0)_{k \in  \mathcal{K}^h}$
    \For{$k \in  \mathcal{K}^h$}
      \State $Succ_k \gets \{ k' \in  \mathcal{K}^{h+1} : R_{k'}\subseteq R_k\}$
      \If{$Succ_k=\varnothing$}
        \State $newVec_k \gets \zeta_k$
      \Else
        \If{$\zeta_{k} \geq  \sum_{k'\in Succ_k} Vec_{k'}$}
          \State $\mathcal{L}\gets \mathcal{L}\setminus \{ k \}$ \Comment{pruning of the region indexed by $k$}
        \EndIf
        \State $newVec_k \gets \min\left( \zeta_{k} ,  \sum_{k'\in Succ_k} Vec_{k'}   \right)$
      \EndIf
    \EndFor
    \State $Vec\gets newVec$
  \EndFor
  \State\Return $(\mathcal{L},\sum_{k\in\mathcal{K}^1} Vec_k  )$
\EndProcedure
\end{algorithmic}
\end{algorithm}
\floatstyle{plain}

Also note that the algorithm returns
\(V^*_{\mathfrak{R}}(\mathbb{N}_m^*)\) as a by-product. The following
proposition states that Algorithm~\ref{algo-pruning} indeed produces the
pruned region as in Definition~\ref{def-pruning}.

\begin{proposition}[]\protect\hypertarget{prp-pruning-correct}{}\label{prp-pruning-correct}

The final \(\mathcal{L}\) returned by Algorithm~\ref{algo-pruning} is
equal to \(\mathcal{K}^{\mathfrak{pr}}\):
\(\mathcal{L}=\mathcal{K}^{\mathfrak{pr}}\).

\end{proposition}

The proof of Proposition~\ref{prp-pruning-correct} is given in
Section~\ref{sec-pruning-proofs-pruning-correct}.

\begin{tcolorbox}[enhanced jigsaw, left=2mm, coltitle=black, toptitle=1mm, colback=white, bottomtitle=1mm, breakable, title=\textcolor{quarto-callout-tip-color}{\faLightbulb}\hspace{0.5em}{Tip \ref*{tip-pruning-complexity}: Tip}, titlerule=0mm, arc=.35mm, opacityback=0, opacitybacktitle=0.6, colbacktitle=quarto-callout-tip-color!10!white, toprule=.15mm, bottomrule=.15mm, rightrule=.15mm, colframe=quarto-callout-tip-color-frame, leftrule=.75mm]

\quartocallouttip{tip-pruning-complexity} 

We saw that Algorithm~\ref{algo-vstar} has \(O(N|S|+|\mathcal{K}|)\)
complexity, the \(N|S|\) term coming from the evaluation of the
\(|S\cap P_i|\) terms, \(1\leq i\leq N\). Here,
\(|S\cap P_i|=|\mathbb{N}_m^*\cap P_i|=|P_i|\) can be accessed in
\(O(1)\), and \(N\leq |\mathcal{K}|\) for a complete family, so
Algorithm~\ref{algo-pruning} simply has \(O(|\mathcal{K}|)\) complexity.

\end{tcolorbox}

\subsection{Fast algorithm to compute a curve of confidence bounds on a
path of selection sets}\label{sec-fast-curve}

Let \((i_1,\dotsc, i_m)\) a permutation of \(\mathbb{N}_m^*\),
eventually random, and, for all \(t\in\mathbb{N}_m^*\), let
\(S_t=\{i_1,\dotsc,i_t\}\) and \(S_0=\varnothing\). For example,
\((i_1,\dotsc, i_m)\) can be the permutation ordering the \(p\)-values
in increasing order and in that case \(S_t\) becomes the set of indices
of the \(t\) smallest \(p\)-values. Assume that we want to compute all
\(V^*_{\mathfrak{R}}(S_t)\) for all \(t\in\{ 0,\dotsc,m\}\), this is
what we call the curve of confidence bounds indexed by
\((i_1,\dotsc, i_m).\) Applying Algorithm~\ref{algo-vstar} to compute
\(V^*_{\mathfrak{R}}(S_t)\) for a given \(t\) has complexity
\(O(|\mathcal{K}|t)\), so using it to sequentially compute the full
curve has complexity
\(O\left(  |\mathcal{K}|\sum_{t=0}^m t\right)=O\left(|\mathcal{K}|m^2\right)\).
In this section, we present a new algorithm that computes the curve with
a \(O\left(|\mathcal{K}|m\right)\) complexity. The algorithm will need
that \(\mathfrak{R}\) is complete, so if that is not the case we first
need to complete \(\mathfrak{R}\) following the Algorithm 2 of
\citet{MR4178188}, which has a \(O(|\mathcal{K}|m)\) complexity. In the
remainder of this section we assume that \(\mathfrak{R}\) is complete.

We first recall and introduce some notation. Recall that \(\phi\) is the
depth function inside of \(\mathfrak{R}\), that
\(\mathfrak P \subseteq \mathcal P(\mathcal{K})\) is the set of subsets
of \(\mathcal{K}\) that realize a partition, recall the important result
stated by Equation \eqref{eq_vstar_Qpartition}, and that
\(\mathcal{K}^h=\{ k\in\mathcal{K}: \phi(k)=h  \}\) for all
\(1\leq h\leq H\) where \(H=\max_{k\in\mathcal{K}}\phi(k)\). For any
\(t\in\mathbb{N}_m^*\) and \(1\leq h\leq H\), we denote by \(k^{(t,h)}\)
the element of \(\mathcal{K}^h\) such that \(i_t\in R_{k^{(t,h)}}\) if
it exists, and we denote by \(h_{\max}(t)\) the highest \(h\) such that
\(k^{(t,h)}\) exists.

\begin{example}[Continuation of Example~\ref{exm-toy-forest} and
Example~\ref{exm-toy-leaves}]\protect\hypertarget{exm-kth}{}\label{exm-kth}

Assume that the reference family of Example~\ref{exm-toy-forest} has
been labeled as in Example~\ref{exm-toy-leaves} and completed. Let
\((i_1,\dotsc, i_{25})\) such that \(i_1=7\), \(i_2=1\) and \(i_3=24\).
Then for \(t=1\), \(k^{(t,1)}=(1,5)\), \(k^{(t,2)}=(2,3)\),
\(k^{(t,3)}=(3,3)\) and \(h_{\max}(t)=H=3\). For \(t=2\),
\(k^{(t,1)}=(1,5)\), \(k^{(t,2)}=(1,1)\), \(k^{(t,3)}\) does not exist
and \(h_{\max}(t)=2\). For \(t=3\), \(k^{(t,1)}=(8,8)\), \(k^{(t,2)}\)
does not exist and \(h_{\max}(t)=1\).

\end{example}

We will now present the new algorithm and the proof that it computes the
curve \((V^*_{\mathfrak{R}}(S_t))_{t\in\mathbb{N}_m}\). We present two
versions of the algorithm (strictly equivalent): one very formal
(Algorithm~\ref{algo-formal-curve}), to introduce additional notation
used in the proof of Theorem~\ref{thm-curve-path}, and, later, a simpler
version that is the one actually implemented
(Algorithm~\ref{algo-curve}). Recall that a detailed illustration of the
steps of the algorithms will be provided in Section~\ref{sec-example}.

In addition to the computation of all \(V^*_{\mathfrak{R}}(S_t)\),
Algorithm~\ref{algo-formal-curve} also computes partitions
\(\mathcal{P}^t\) that realize the minimum in
\eqref{eq_vstar_Qpartition} for \(V^*_{\mathfrak{R}}(S_t)\). The
initialization of \(\mathcal{P}^t\) has to be done carefully, for that
we let \(\mathcal{K}_0^-=\{k\in\mathcal{K}: \zeta_k=0  \}\) and
\begin{equation}
E=\left\{k\in\mathcal{K}_0^-:\forall k'\in \mathcal{K}_0^-, R_k\subseteq R_{k'}\Rightarrow k'=k \right\},
\label{E}
\end{equation} \begin{equation}
F=\left\{(i,i), 1\leq i\leq N:\forall k\in \mathcal{K}_0^-, R_{(i,i)}\not\subseteq R_k  \right\},
\label{F}
\end{equation} and finally \begin{equation}
\mathcal{P}^0=E\cup F.
\label{cP0}
\end{equation} \(E\) is the set of indices of the maximal elements \(k\)
such that \(\zeta_k=0\), \(F\) is the set of indices of all leaves that
are not a subset of a region indexed by \(E\). \(\mathcal{P}^0\) is the
disjoint union of the two.

\floatstyle{ruled}
\restylefloat{algorithm}
\begin{algorithm}[htb!]
\caption{\label{algo-formal-curve}Formal computation of $(V^*_{\mathfrak{R}}(S_t))_{0\leq t\leq m}$ with a complete family}
\begin{algorithmic}[1]
\Procedure{Curve}{$\mathfrak{R}=(R_{k},\zeta_{k})_{k\in\mathcal{K}}$  with $\mathfrak{R}$ complete, path $(S_t)_{1\leq t \leq m}$ with $S_t=\{i_1, \dotsc, i_t\}$}
  \State $\mathcal{P}^0\gets E\cup F$ \Comment{see \eqref{E} and \eqref{F}} 
  \State $\mathcal{K}^-_0\gets\{k\in\mathcal{K} : \zeta_k=0  \}$
  \State $\eta^0_k\gets0$ for all $k\in\mathcal{K}$
  \For{$t=1,\dotsc, m$}
    \If{$i_t\in\bigcup_{\kappa\in\mathcal{K}^-_{t-1}}R_{\kappa}$}
      \State $\mathcal{P}^t \gets \mathcal{P}^{t-1}$
      \State $\mathcal{K}^-_t \gets \mathcal{K}^-_{t-1}$
      \State $\eta^t_k\gets\eta^{t-1}_k$ for all $k\in\mathcal{K}$
    \Else
      \For{$h=1,\dotsc,h_{\max}(t)$}
        \State $\eta^t_{k^{(t,h)}}\gets\eta^{t-1}_{k^{(t,h)}} + 1$
        \If{$\eta^t_{k^{(t,h)}}<\zeta_k$}
          \State Pass
        \Else
          \State $h^f_t \gets h$ \Comment{final depth}
          \State $\mathcal{P}^t \gets\left( \mathcal{P}^{t-1}\setminus \{ k\in \mathcal{P}^{t-1} : R_k\subseteq R_{k^{(t,h^f_t)}} \}\right)\cup \{ k^{(t,h^f_t)} \}$
          \State $\mathcal{K}^-_t \gets \mathcal{K}^-_{t-1} \cup \{k^{(t,h^f_t)}\}$
          \State Break the loop
        \EndIf
      \EndFor
      \If{the loop has been broken}
        \State $\eta^t_k\gets\eta^{t-1}_k$ for all $k\in\mathcal{K}$ not visited during the loop, that is all $k\not\in\{k^{(t,h)}, 1\leq h\leq h^f_t   \}$
      \Else
        \State $\mathcal{P}^t \gets \mathcal{P}^{t-1}$
        \State $\mathcal{K}^-_t \gets \mathcal{K}^-_{t-1}$
        \State $\eta^t_k\gets\eta^{t-1}_k$ for all $k\in\mathcal{K}$ not visited during the loop, that is all $k\not\in\{k^{(t,h)}, 1\leq h\leq h_{\max}(t)   \}$
      \EndIf
    \EndIf
  \EndFor
  \State\Return $\mathcal{P}^t, \eta^t_k$ for all $t=1,\dotsc, m$ and $k\in\mathcal{K}$
\EndProcedure
\end{algorithmic}
\end{algorithm}
\floatstyle{plain}

The core idea of the algorithm is that, as we increase \(t\) and add new
hypotheses in \(S_t\), we inflate a counter \(\eta_k^t\) for each region
\(R_k\), by 1 if \(i_t\in R_k\) (line 12), by 0 if not (lines 23 and
27), but only until the counter reaches \(\zeta_k\) (line 13). After
this point, the hypotheses in \(R_k\) don't contribute to
\(V^*_{\mathfrak{R}}(S_t)\), we keep track of those hypotheses with
\(\mathcal{K}^-_t\) (line 6), so as soon as
\(\eta^t_{k^{(t,h)}}=\zeta_k\) we update \(\mathcal{K}^-_t\) by adding
\(k^{(t,h)}\) (line 18) to it and we update \(\mathcal{P}^t\)
accordingly (line 17).

We will see in the following Theorem~\ref{thm-curve-path} how this
algorithm allows to compute \(V^*_{\mathfrak{R}}(S_t)\). We first need a
final notation. Let \begin{equation*}
\mathcal{K}_t=\{k\in\mathcal{K}: \exists k'\in \mathcal{P}^t : R_{k'}\subseteq R_k   \}.
\end{equation*} The elements of \(\mathcal{K}_t\) index the regions of
the forest that ``are above'\,' the regions of the current
partition-realizing \(\mathcal{P}^t\). In particular, we always have,
for any \(t\in\mathbb{N}_m\), \(\mathcal{K}^1\subseteq\mathcal{K}_t\)
and \(\mathcal{P}^t\subseteq \mathcal{K}_t\). We can also remark that
the sequence \((\mathcal{K}_t)_{0\leq t \leq m}\) is non-increasing for
the inclusion relation, and that \(\mathcal{K}_0=\mathcal{K}\).

\begin{theorem}[Fast curve
computation]\protect\hypertarget{thm-curve-path}{}\label{thm-curve-path}

Assume that \(\mathfrak{R}\) is complete in the sense of
Section~\ref{sec-forest-structure}.

Let any \(t\in\mathbb{N}_m\). Then, \(\mathcal{P}^t\in\mathfrak P\), and
for all \(k\in\mathcal{K}_t\), we have \begin{equation}
V^*_{\mathfrak{R}}(S_t\cap R_k) = \eta_k^t
\label{eq_vstar_inter_Rk_equal_eta}
\end{equation} and \begin{equation}
V^*_{\mathfrak{R}}(S_t\cap R_k) = \sum_{\substack{k'\in \mathcal{P}^t\\ R_{k'}\subseteq R_k}} \zeta_{k'}\wedge|S_t \cap R_{k'}|.
\label{eq_Pt_good_partition}
\end{equation} Furthermore, \begin{equation}
V^*_{\mathfrak{R}}(S_t)  = \sum_{{k\in \mathcal{P}^t}} \zeta_{k}\wedge|S_t \cap R_{k}|= \sum_{k\in\mathcal{K}^1} \eta_k^t.
\label{eq_vstar_equal_sum_eta}
\end{equation}

\end{theorem}

The proof of this Theorem is given in Section~\ref{sec-proof}. The first
equality of Equation \eqref{eq_vstar_equal_sum_eta} states that the
minimum in \eqref{eq_vstar_Qpartition} is indeed realized on the
partition \(\mathcal{P}^t\), and the last equality of the same Equation
is the basis of the following light corollary.

\begin{corollary}[Easy
computation]\protect\hypertarget{cor-easy-impl}{}\label{cor-easy-impl}

Assume that \(\mathfrak{R}\) is complete in the sense of
Section~\ref{sec-forest-structure}.

For \(t\in\{0,\dotsc, m-1 \}\),
\(V^*_{\mathfrak{R}}(S_{t+1})=V^*_{\mathfrak{R}}(S_{t})\) if
\(i_{t+1}\in \bigcup_{k\in\mathcal{K}^-_t}R_k\), and
\(V^*_{\mathfrak{R}}(S_{t+1})=V^*_{\mathfrak{R}}(S_{t}) + 1\) if not.

\end{corollary}

\begin{proof}
From \eqref{eq_vstar_equal_sum_eta},
\(V^*_{\mathfrak{R}}(S_{t+1})=\sum_{k\in\mathcal{K}^1} \eta_k^{t+1}\)
and \(V^*_{\mathfrak{R}}(S_{t})=\sum_{k\in\mathcal{K}^1} \eta_k^{t}\).
If \(i_{t+1}\in \bigcup_{k\in\mathcal{K}^-_t}R_k\),
\(\eta_k^{t+1}=\eta_k^{t}\) for all \(k\in\mathcal{K}^1\). If not,
\(\eta_k^{t+1}=\eta_k^{t}\) for all \(k\in\mathcal{K}^1\),
\(k\neq k^{(t+1,1)}\), whereas for \(k= k^{(t+1,1)}\),
\(\eta_k^{t+1}=\eta_k^{t}+1\).
\end{proof}

We note that, from Theorem~\ref{thm-curve-path} and
Corollary~\ref{cor-easy-impl}, if one is only interested in the
computation of the curve
\(\left(V^*_{\mathfrak{R}}(S_{t})\right)_{1\leq t\leq m}\), tracking
\(\mathcal{P}^t\) is actually useless, what is important is to track and
update \(\mathcal{K}^-_t\) correctly. Hence the simpler, alternative
Algorithm~\ref{algo-curve}. Note that Algorithm~\ref{algo-curve} is less
formal than Algorithm~\ref{algo-formal-curve}: as in
Algorithm~\ref{algo-vstar} and Algorithm~\ref{algo-pruning}, it recycles
notation (mimicking the actual code implementation) so the \(t\)
subscript or superscript is dropped from all \(\mathcal{K}^-_t\) and
\(\eta_k^t\). In Algorithm~\ref{algo-curve}, the notation \(V_t\) is
actually equal to \(V^*_{\mathfrak{R}}(S_{t})\) by
Corollary~\ref{cor-easy-impl}.

\floatstyle{ruled}
\restylefloat{algorithm}
\begin{algorithm}[htb!]
\caption{\label{algo-curve}Practical computation of $(V^*_{\mathfrak{R}}(S_t))_{0\leq t\leq m}$}
\begin{algorithmic}[1]
\Procedure{Curve}{$\mathfrak{R}=(R_{k},\zeta_{k})_{k\in\mathcal{K}}$ with $\mathfrak{R}$ complete, path $(S_t)_{1\leq t \leq m}$ with $S_t=\{i_1, \dotsc, i_t\}$}
  \State $V_0\gets 0$
  \State $\mathcal{K}^-\gets\{k\in\mathcal{K} : \zeta_k=0  \}$
  \State $\eta_k\gets 0$ for all $k\in\mathcal{K}$
  \For{$t=1,\dotsc, m$}
    \If{$i_t\in\bigcup_{\kappa\in\mathcal{K}^-}R_{\kappa}$}
      \State $V_{t}\gets V_{t-1}$
    \Else
      \For{$h=1,\dotsc,h_{\max}(t)$}
        \State find $k^{(t,h)}\in\mathcal{K}^{h}$ such that $i_t\in R_{k^{(t,h)}}$
        \State $\eta_{k^{(t,h)}}\gets\eta_{k^{(t,h)}} + 1$
        \If{$\eta_{k^{(t,h)}}<\zeta_k$}
          \State pass
        \Else
          \State $\mathcal{K}^- \gets \mathcal{K}^-\cup \{ k^{(t,h)} \}$
          \State break the loop
        \EndIf
      \EndFor
     \State $V_{t}\gets V_{t-1} + 1$
    \EndIf
  \EndFor
  \State\Return $(V_t)_{1\leq t \leq m}$
\EndProcedure
\end{algorithmic}
\end{algorithm}
\floatstyle{plain}

Stocking, for each \(i_t\), the indices \(k\) such that \(i_t\in R_k\),
is done by scanning the forest structure so it has complexity in
\(O(|\mathcal{K}|)\). Once this information is available, finding
\(k^{(t,h)}\) and updating \(\mathcal{K}^-\) can be done in \(O(1)\).
Then each step \(t\) of the \texttt{for} loop consists in two successive
scans of the \(k^{(t,h)}\), \(1\leq h\leq H\), the first to check if
\(i_t\in\bigcup_{k\in\mathcal{K}^-}R_k\), and the second to update the
\(\eta_{k^{(t,h)}}\) if \(i_t\not\in\bigcup_{k\in\mathcal{K}^-}R_k\). So
each step has complexity in \(O(H)\) and finally the complexity of
Algorithm~\ref{algo-curve} is in
\(O(Hm+|\mathcal{K}|)\leq O(|\mathcal{K}|m)\), and even only \(O(Hm)\)
after the first call if the necessary information has been stocked.

\subsection{Illustration on a detailed example}\label{sec-example}

In this section, we follow Algorithm~\ref{algo-formal-curve} during its
first steps in a detailed fashion.

We keep the structure of Example~\ref{exm-toy-forest} and
Example~\ref{exm-toy-leaves}. Recall that \(m=25\),
\(P_{1:5}=R_1 = \{1, \dotsc , 20 \}\), \(P_1=R_2  =  \{1, 2  \}\),
\(P_{2:3}=R_3   =   \{3 , \dotsc , 10 \}\),
\(P_{4:5}=R_4  =    \{11, \dotsc , 20 \}\), \(P_2=\{3,4\}\),
\(P_3=R_5 =  \{5, \dotsc , 10 \}\),
\(P_4=R_6   =     \{11, \dotsc , 16 \}\),
\(P_5=R_7  =   \{17, \dotsc ,20  \}\), \(P_{6:7}=R_8=\{21,22\}\),
\(P_6=\{21\}\), \(P_7=R_9 = \{22\}\) and \(P_8=\{23,24,25\}\).

Now assume that we have the following values for the \(\zeta_k\)'s:
\(\zeta_{(1,5)}=5\), \(\zeta_{(1, 1)}=2\), \(\zeta_{(2, 3)}=0\),
\(\zeta_{(3, 3)}=0\), \(\zeta_{(4, 5)}=4\), \(\zeta_{(4, 4)}=2\),
\(\zeta_{(5, 5)}=3\), \(\zeta_{(6, 7)}=2\), \(\zeta_{(7, 7)}=0\).
Because \(P_2\), \(P_6\) and \(P_8\) come from the completion operation
(see Section~\ref{sec-forest-structure}), we also have
\(\zeta_{(2, 2)}=|P_2|=2\), \(\zeta_{(6, 6)}=|P_6|=1\) and
\(\zeta_{(8, 8)}=|P_8|=3\). Theses values are summarized in
Figure~\ref{fig-zetas}.

\begin{figure}

\centering{

\includegraphics{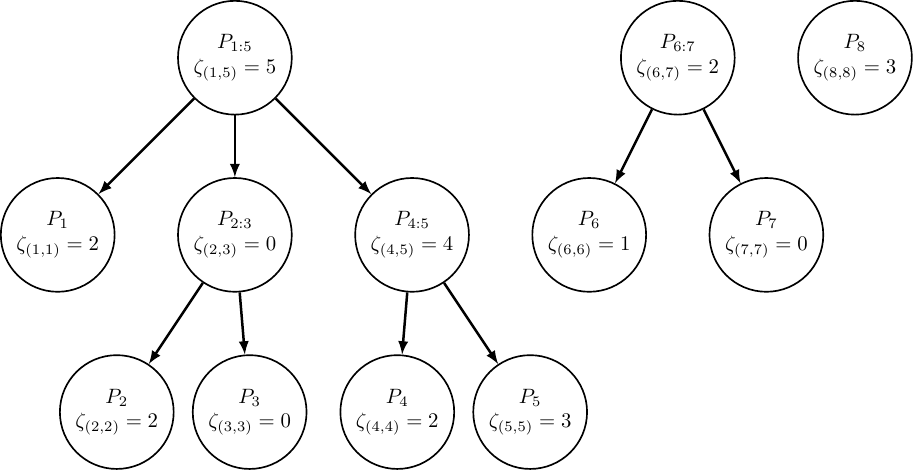}

}

\caption{\label{fig-zetas}The regions of Example~\ref{exm-toy-forest}
with the \(\zeta_k\) values.}

\end{figure}%

We want to compute the curve
\(\left(V^*_{\mathfrak{R}}\left(S_t\right)\right)_{1\leq t\leq 9}\) with
\(S_t=\{i_1,\dotsc, i_t\}\) and \(i_1=11\), \(i_2=17\), \(i_3=12\),
\(i_4=13\), \(i_5=18\), \(i_6=24\), \(i_7=19\), \(i_8=22\) and
\(i_9=5\).

First, we apply Algorithm~\ref{algo-pruning} to the family. This results
in pruning \(P_{6:7}\) (and only this region), because
\(2=\zeta_{(6, 7)}\geq \zeta_{(6, 6)}+\zeta_{(7, 7)}=1+0\). This gives
Figure~\ref{fig-pruned}.

\begin{figure}

\centering{

\includegraphics{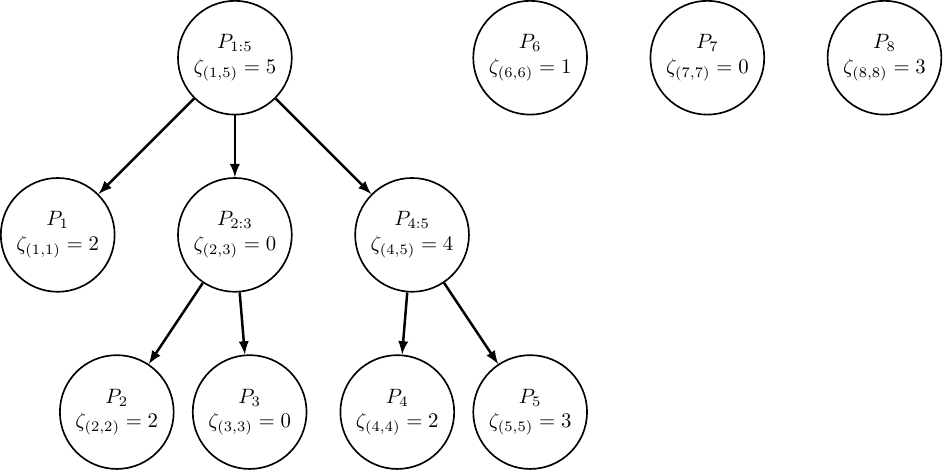}

}

\caption{\label{fig-pruned}The regions of Example~\ref{exm-toy-forest}
after pruning.}

\end{figure}%

Now we initialize Algorithm~\ref{algo-formal-curve}, that is we let
\(t=0\). Because \(\zeta_{(2, 3)}=\zeta_{(3, 3)}=\zeta_{(7,7)}=0\),
\((2,3)\), \((3,3)\) and \((7, 7)\) are added to \(\mathcal{K}^-_t\):
\(\mathcal{K}^-_0=\{(2, 3), (3, 3), (7, 7)\}\). We define
\(\mathcal{P}^0\) according to \eqref{E}, \eqref{F} and \eqref{cP0}.
Here, \(E=\{(2,3), (7,7)\}\) and so
\(F=\{(1,1), (4,4), (5,5), (6,6), (8,8)\}\) and
\(\mathcal{P}^0=E\cup F\). Furthermore, all \(\eta_k^t\) are set to 0.
The initial state of Algorithm~\ref{algo-formal-curve} is shown in
Figure~\ref{fig-t0}, with the elements of \(\mathcal{K}^-_t\) being in
red to show that they will not contribute to the computations, and the
elements of \(\mathcal{P}^t\) as squares.

\begin{figure}

\centering{

\includegraphics{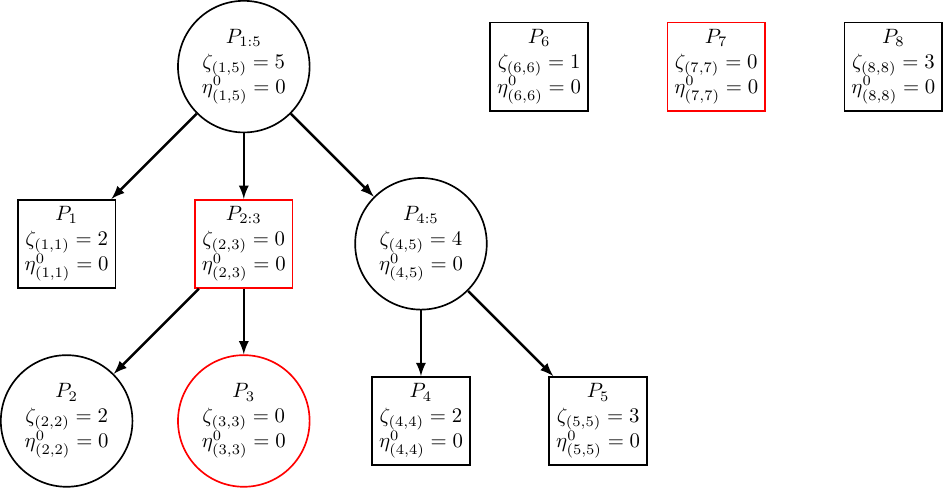}

}

\caption{\label{fig-t0}The regions of Example~\ref{exm-toy-forest} at
\(t=0\) in Algorithm~\ref{algo-formal-curve}.}

\end{figure}%

We move on to \(t=1\), with \(i_1=11\).
\(i_1\in P_4\subseteq P_{4:5}\subseteq P_{1:5}\). The appropriate
\(\eta_k^t\) are increased by one, and by \eqref{eq_vstar_equal_sum_eta}
we have
\(V^*_{\mathfrak{R}}(S_1)=\eta_{(1, 5)}^1+\eta_{(6, 6)}^1+\eta_{(7, 7)}^1+\eta_{(8, 8)}^1=1+0+0+0=1\).
The state of the step is summarized in Figure~\ref{fig-t1}.

\begin{figure}

\centering{

\includegraphics{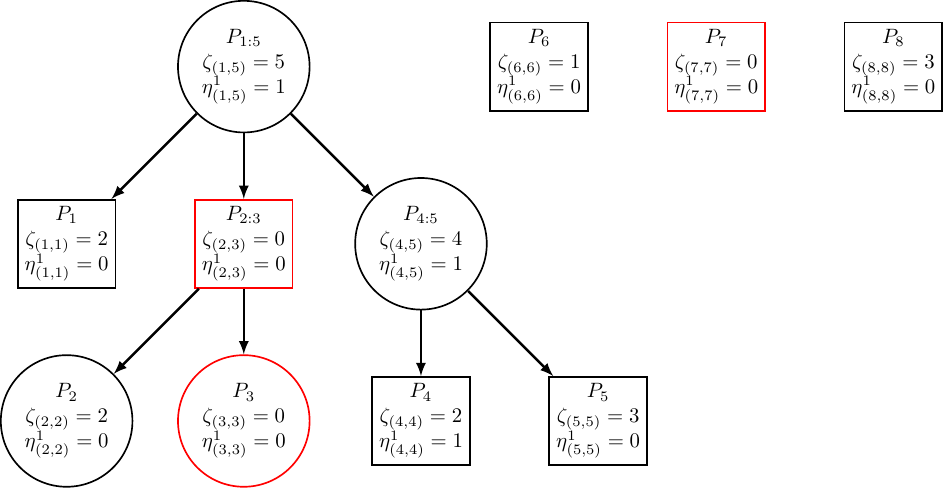}

}

\caption{\label{fig-t1}The regions of Example~\ref{exm-toy-forest} at
\(t=1\) in Algorithm~\ref{algo-formal-curve}.}

\end{figure}%

We move on to \(t=2\), with \(i_2=17\).
\(i_1\in P_5\subseteq P_{4:5}\subseteq P_{1:5}\). The appropriate
\(\eta_k^t\) are increased by one, and by \eqref{eq_vstar_equal_sum_eta}
we have \(V^*_{\mathfrak{R}}(S_2)=2\). The state of the step is
summarized in Figure~\ref{fig-t2}.

\begin{figure}

\centering{

\includegraphics{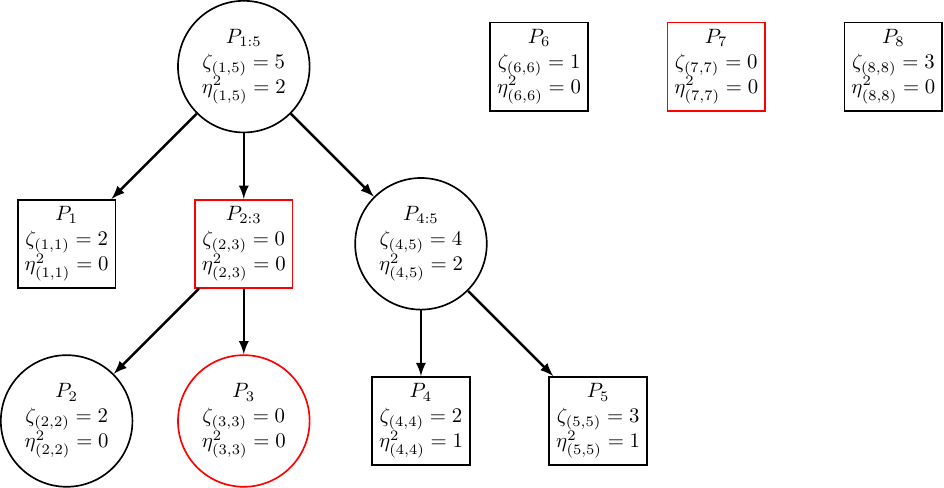}

}

\caption{\label{fig-t2}The regions of Example~\ref{exm-toy-forest} at
\(t=2\) in Algorithm~\ref{algo-formal-curve}.}

\end{figure}%

We move on to \(t=3\), with \(i_3=12\).
\(i_3\in P_4\subseteq P_{4:5}\subseteq P_{1:5}\). The appropriate
\(\eta_k^t\) are increased by one, and we notice that
\(\eta_{(4, 4)}^3=2=\zeta_{(4, 4)}\). So \(P_4\) will stop contributing,
we add it to \(\mathcal{K}^-_t\):
\(\mathcal{K}^-_3=\{(2,3), (3,3), (4, 4), (7, 7)\}\). Following line 17
of Algorithm~\ref{algo-formal-curve}, \(\mathcal{P}^t\) does not change
(we remove then add \((4,4)\) from it) and
\(\mathcal{P}^3=\mathcal{P}^0\). By \eqref{eq_vstar_equal_sum_eta}, we
have \(V^*_{\mathfrak{R}}(S_3)=3\). The state of the step is summarized
in Figure~\ref{fig-t3}, with \(P_4\) now also in red.

\begin{figure}

\centering{

\includegraphics{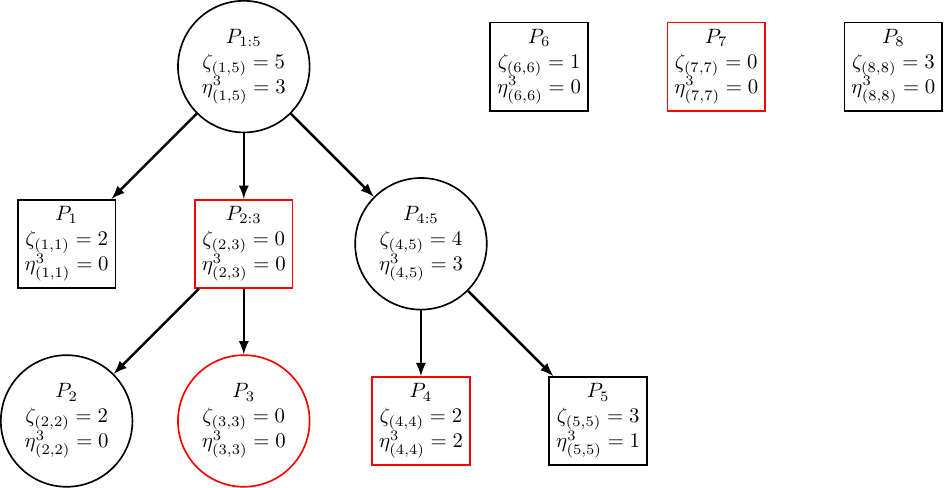}

}

\caption{\label{fig-t3}The regions of Example~\ref{exm-toy-forest} at
\(t=3\) in Algorithm~\ref{algo-formal-curve}.}

\end{figure}%

We move on to \(t=4\), with \(i_4=13\).
\(i_4\in P_4\in \bigcup_{k\in\mathcal{K}^-_3}R_k\). No \(\eta_k^t\) is
increased (see line 9 of Algorithm~\ref{algo-formal-curve}), and by
\eqref{eq_vstar_equal_sum_eta}, we have \(V^*_{\mathfrak{R}}(S_4)=3\).

We move on to \(t=5\), with \(i_5=18\).
\(i_5\in P_5\subseteq P_{4:5}\subseteq P_{1:5}\). We first increase
\(\eta_{(1,5)}^t\): \(\eta_{(1,5)}^5=4<\zeta_{(1,5)}\), then
\(\eta_{(4,5)}^t\): \(\eta_{(4,5)}^5=4\), and we stop there because
\(\eta_{(4,5)}^5=4=\zeta_{(4,5)}\). \(P_{4:5}\) will stop contributing,
we add it to \(\mathcal{K}^-_t\):
\(\mathcal{K}^-_5=\{(2,3), (4,5), (3,3), (4, 4), (7, 7)\}\). We also add
\((4,5)\) and remove \((4,4)\) and \((5,5)\) from \(\mathcal{P}^t\):
\(\mathcal{P}^5=\{(1,1), (2,3), (4,5), (6,6), (7,7), (8,8)\}\). Note
that \(\eta_{(5,5)}^t\) is not updated because we stopped the loop
before, see line 23 of Algorithm~\ref{algo-formal-curve}. By
\eqref{eq_vstar_equal_sum_eta}, we have \(V^*_{\mathfrak{R}}(S_5)=4\).
The state of the step is summarized in Figure~\ref{fig-t5}, with
\(P_{4:5}\) now also in red.

\begin{figure}

\centering{

\includegraphics{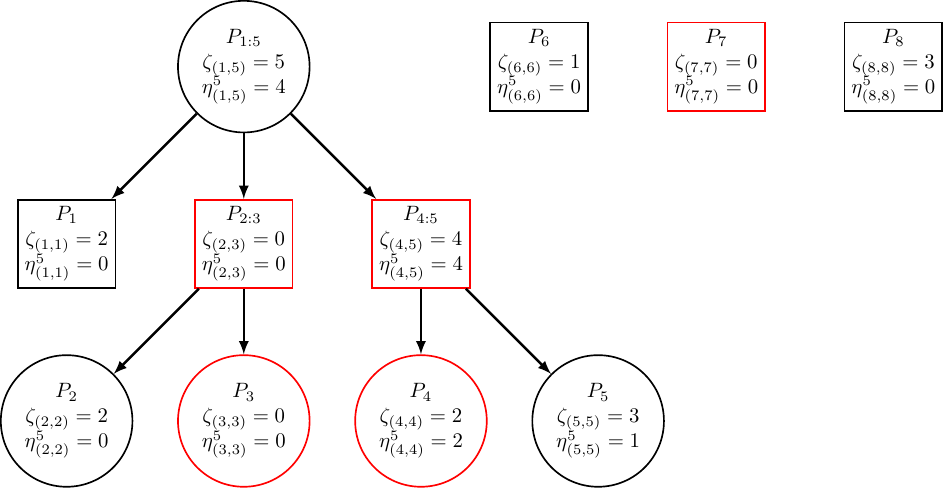}

}

\caption{\label{fig-t5}The regions of Example~\ref{exm-toy-forest} at
\(t=5\) in Algorithm~\ref{algo-formal-curve}.}

\end{figure}%

We move on to \(t=6\), with \(i_6=24\). \(i_6\in P_8\). The appropriate
\(\eta_k^t\) is increased by one: \(\eta_{(8, 8)}^6=1<\zeta_{(8,8)}\),
and by \eqref{eq_vstar_equal_sum_eta} we have
\(V^*_{\mathfrak{R}}(S_6)=\eta_{(1, 5)}^6+\eta_{(6, 6)}^6+\eta_{(7, 7)}^6+\eta_{(8, 8)}^6=4+0+0+1=5\).
The state of the step is summarized in Figure~\ref{fig-t6}.

\begin{figure}

\centering{

\includegraphics{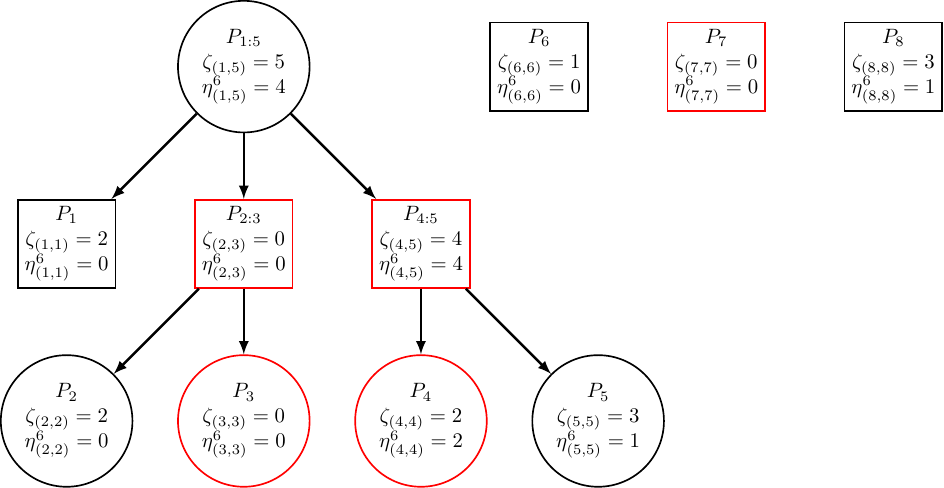}

}

\caption{\label{fig-t6}The regions of Example~\ref{exm-toy-forest} at
\(t=6\) in Algorithm~\ref{algo-formal-curve}.}

\end{figure}%

We move on to the remaining steps. \(i_7=19\in P_{4:5}\),
\(i_8=22\in P_{7}\) and \(i_9=5\in P_{2:3}\) are all in
\(\bigcup_{k\in\mathcal{K}^-_6}R_k\) so no \(\eta_k^t\) is increased at
their step (see line 9 of Algorithm~\ref{algo-formal-curve}), and by
\eqref{eq_vstar_equal_sum_eta}, we have
\(V^*_{\mathfrak{R}}(S_7)=V^*_{\mathfrak{R}}(S_8)=V^*_{\mathfrak{R}}(S_9)=5\).

\section{Implementation}\label{sec-implementation}

All algorithms discussed in this manuscript are already implemented in
the \texttt{R} \citep{R-base} package \texttt{sanssouci}
\citep{sanssouci} which is available on GitHub (see the References for
the link) and is dedicated to the computation of confidence bounds for
the number of false positives. It also hosts the implementation of the
methods described in \citet{MR4124323} and
\citet{10.1093/bioinformatics/btac693}. Algorithm~\ref{algo-vstar} is
implemented as the \texttt{V.star} function,
Algorithm~\ref{algo-pruning} is implemented as the \texttt{pruning}
function, and Algorithm~\ref{algo-curve} is implemented as the
\texttt{curve.V.star.forest.fast} function (whereas the
\texttt{curve.V.star.forest.naive} function just repeatedly calls
\texttt{V.star}). Note that the \texttt{pruning} function has a
\texttt{delete.gaps} option that speeds up the computation even more by
removing unnecessary gaps introduced in the data structure by the
pruning operation, those gaps being due to the specific structure that
is used to store the information of \(\mathcal{K}\).

Speaking of the data structure, we briefly describe it, with an example.
We represent \((R_k)_{k\in\mathcal{K}}\) by two lists, \texttt{C} and
\texttt{leaf\_list}. \texttt{leaf\_list} is a list of vectors, where
\texttt{leaf\_list{[}{[}i{]}{]}} is the vector listing the hypotheses in
the atom \(P_i\). \texttt{C} is a list of lists. For \(1\leq h\leq H\),
\texttt{C{[}{[}h{]}{]}} lists the regions at depth \(h\), using the
index bounds of the atoms they are composed of. That is, the elements of
the list \texttt{C{[}{[}h{]}{]}} are vectors of size two, and if there
is \texttt{k}, \texttt{i} and \texttt{j} such that
\texttt{C{[}{[}h{]}{]}{[}{[}k{]}{]}\ =\ c(i,\ j)}, it means that
\((i, j)\in\mathcal{K}\), or in other words that \(R_{(i,j)}=P_{i:j}\)
is one of the regions, and that \(\phi((i, j))=h\).

\begin{example}[Implementation of
Example~\ref{exm-toy-leaves}]\protect\hypertarget{exm-implementation}{}\label{exm-implementation}

For the reference family given in Example~\ref{exm-toy-forest} and
completed in Example~\ref{exm-toy-leaves}, \(H=3\). For \(h=1\), we have
\texttt{C{[}{[}1{]}{]}{[}{[}1{]}{]}\ =\ c(1,\ 5)},
\texttt{C{[}{[}1{]}{]}{[}{[}2{]}{]}\ =\ c(6,\ 7)},
\texttt{C{[}{[}1{]}{]}{[}{[}3{]}{]}\ =\ c(8,\ 8)}. For \(h=2\), we have
\texttt{C{[}{[}2{]}{]}{[}{[}1{]}{]}\ =\ c(1,\ 1)},
\texttt{C{[}{[}2{]}{]}{[}{[}2{]}{]}\ =\ c(2,\ 3)},
\texttt{C{[}{[}2{]}{]}{[}{[}3{]}{]}\ =\ c(4,\ 5)},
\texttt{C{[}{[}2{]}{]}{[}{[}4{]}{]}\ =\ c(6,\ 6)},
\texttt{C{[}{[}2{]}{]}{[}{[}5{]}{]}\ =\ c(7,\ 7)}. For \(h=3\), we have
\texttt{C{[}{[}3{]}{]}{[}{[}1{]}{]}\ =\ c(2,\ 2)},
\texttt{C{[}{[}3{]}{]}{[}{[}2{]}{]}\ =\ c(3,\ 3)},
\texttt{C{[}{[}3{]}{]}{[}{[}3{]}{]}\ =\ c(4,\ 4)},
\texttt{C{[}{[}3{]}{]}{[}{[}4{]}{]}\ =\ c(5,\ 5)}.

And then for the atoms, we have
\texttt{leaf\_list{[}{[}1{]}{]}\ =\ c(1,\ 2)},
\texttt{leaf\_list{[}{[}2{]}{]}\ =\ c(3,\ 4)},
\texttt{leaf\_list{[}{[}3{]}{]}\ =\ c(5,\ 6,\ 7,\ 8,\ 9,\ 10)},
\texttt{leaf\_list{[}{[}4{]}{]}\ =\ c(11,\ 12,\ 13,\ 14,\ 15,\ 16)},
\texttt{leaf\_list{[}{[}5{]}{]}\ =\ c(17,\ 18,\ 19,\ 20)},
\texttt{leaf\_list{[}{[}6{]}{]}\ =\ 21},
\texttt{leaf\_list{[}{[}7{]}{]}\ =\ 22} and finally
\texttt{leaf\_list{[}{[}8{]}{]}\ =\ c(23,\ 24,\ 25)}.

\end{example}

\begin{tcolorbox}[enhanced jigsaw, left=2mm, coltitle=black, toptitle=1mm, colback=white, bottomtitle=1mm, breakable, title=\textcolor{quarto-callout-caution-color}{\faFire}\hspace{0.5em}{Caution \ref*{cau-impl}: Caution}, titlerule=0mm, arc=.35mm, opacityback=0, opacitybacktitle=0.6, colbacktitle=quarto-callout-caution-color!10!white, toprule=.15mm, bottomrule=.15mm, rightrule=.15mm, colframe=quarto-callout-caution-color-frame, leftrule=.75mm]

\quartocalloutcau{cau-impl} 

We emphasize that the 1D structure of the hypotheses has to be respected
by the user as the current implementation implicitly uses it: that is,
\(P_1\) has to contain the hypotheses labeled \(1, 2, \dotsc, p\),
\(P_2\) has to contain the hypotheses labeled \(p+1, \dotsc\), and so
on. Also, the hypotheses have to be in increasing order:
\texttt{leaf\_list{[}{[}1{]}{]}} has to be equal to
\texttt{c(1,\ 2,\ 3,\ ...,\ p)} and not, say,
\texttt{c(2,\ 1,\ 3,\ ...,\ p)}.

\end{tcolorbox}

\begin{tcolorbox}[enhanced jigsaw, left=2mm, coltitle=black, toptitle=1mm, colback=white, bottomtitle=1mm, breakable, title=\textcolor{quarto-callout-tip-color}{\faLightbulb}\hspace{0.5em}{Tip \ref*{tip-leaves}: Tip}, titlerule=0mm, arc=.35mm, opacityback=0, opacitybacktitle=0.6, colbacktitle=quarto-callout-tip-color!10!white, toprule=.15mm, bottomrule=.15mm, rightrule=.15mm, colframe=quarto-callout-tip-color-frame, leftrule=.75mm]

\quartocallouttip{tip-leaves} 

We see that the current implementation requires to provide a partition
\((P_n)_{1\leq n \leq N}\) of leaves compatible with the reference
family. So, in a way, we always provide a complete family. However, not
all \texttt{c(i,i)} have to be in \texttt{C} for the various algorithms
implemented to function properly. It is not needed for the
implementations of Algorithm~\ref{algo-vstar} and
Algorithm~\ref{algo-pruning} given that, as stated in
Tip~\ref{tip-algo1}, they start by computing \(|S\cap P_n|\) for each
leaf \(P_n\) anyway. Having a complete \texttt{C} is not needed either
for Algorithm~\ref{algo-curve}. Indeed, if \(R_k\) is a region such that
\(\zeta_k=|R_k|\), the condition of line 12 of
Algorithm~\ref{algo-curve} is always true, except for the last element
of \(R_k\) that is added to \(S_t\). At that point, because the elements
of \(R_k\) have been exhausted, \(R_k\) won't ever be visited again, so
adding it to \(\mathcal{K}^-\) is irrelevant. In the end, tracking
\(R_k\) and \(\eta_k\) for such \(k\) was not necessary in the first
place. Furthermore, if the implementation does not find \(k^{(t,h)}\)
(see line 10), it simply pass to the next iteration of the \texttt{for}
loop.

\end{tcolorbox}

The functions \texttt{dyadic.from.leaf\_list},
\texttt{dyadic.from.window.size}, and \texttt{dyadic.from.height} return
the appropriate data structure to represent a \(\mathcal{K}\) that can
be described as a dyadic tree, based on some entry parameters that can
be inferred from the names of the functions. As said in
Tip~\ref{tip-leaves}, the completion of \texttt{C} given
\texttt{leaf\_list} is not necessary, but can be done by the
\texttt{forest.completion} function. Finally, the \(\zeta_k\)'s are
computed as in \citet{MR4178188} by the \texttt{zetas.tree} function
with \texttt{method=zeta.DKWM}. Using \texttt{method=zeta.trivial} just
yields \(\zeta_k=|R_k|\).

The following \texttt{R} snippet constructs the family of
Example~\ref{exm-toy-forest}, draws uniform \(p\)-values
\(p_1,\dotsc,p_{25}\), computes the \(\zeta_k\)'s with
\texttt{method=zeta.trivial}, completes the family with
\texttt{forest.completion} to get the family of
Example~\ref{exm-toy-leaves}, allowing to verify the claim of
Example~\ref{exm-implementation}, prunes it with \texttt{pruning} (which
prunes everything except the leaves because of \texttt{zeta.trivial}),
and finally computes the curve of confidence bounds on the path
\(S_t=\{\sigma(1),\dots,\sigma(t)\}\), using both the pruned and
non-pruned complete family, where \(\sigma\) is a permutation ordering
the \(p\)-values, using the fast implementation
\texttt{curve.V.star.forest.fast}.

\begin{Shaded}
\begin{Highlighting}[]
\FunctionTok{library}\NormalTok{(sanssouci)}

\NormalTok{leaf\_list }\OtherTok{\textless{}{-}} \FunctionTok{list}\NormalTok{(}\FunctionTok{c}\NormalTok{(}\DecValTok{1}\NormalTok{, }\DecValTok{2}\NormalTok{), }
                  \FunctionTok{c}\NormalTok{(}\DecValTok{3}\NormalTok{, }\DecValTok{4}\NormalTok{), }
                  \FunctionTok{c}\NormalTok{(}\DecValTok{5}\NormalTok{, }\DecValTok{6}\NormalTok{, }\DecValTok{7}\NormalTok{, }\DecValTok{8}\NormalTok{, }\DecValTok{9}\NormalTok{, }\DecValTok{10}\NormalTok{), }
                  \FunctionTok{c}\NormalTok{(}\DecValTok{11}\NormalTok{, }\DecValTok{12}\NormalTok{, }\DecValTok{13}\NormalTok{, }\DecValTok{14}\NormalTok{, }\DecValTok{15}\NormalTok{, }\DecValTok{16}\NormalTok{), }
                  \FunctionTok{c}\NormalTok{(}\DecValTok{17}\NormalTok{, }\DecValTok{18}\NormalTok{, }\DecValTok{19}\NormalTok{, }\DecValTok{20}\NormalTok{), }
                  \DecValTok{21}\NormalTok{, }
                  \DecValTok{22}\NormalTok{, }
                  \FunctionTok{c}\NormalTok{(}\DecValTok{23}\NormalTok{, }\DecValTok{24}\NormalTok{, }\DecValTok{25}\NormalTok{))}
\NormalTok{C }\OtherTok{\textless{}{-}} \FunctionTok{list}\NormalTok{(}\FunctionTok{list}\NormalTok{(}\FunctionTok{c}\NormalTok{(}\DecValTok{1}\NormalTok{, }\DecValTok{5}\NormalTok{), }\FunctionTok{c}\NormalTok{(}\DecValTok{6}\NormalTok{,}\DecValTok{7}\NormalTok{)),}
          \FunctionTok{list}\NormalTok{(}\FunctionTok{c}\NormalTok{(}\DecValTok{1}\NormalTok{, }\DecValTok{1}\NormalTok{), }\FunctionTok{c}\NormalTok{(}\DecValTok{2}\NormalTok{, }\DecValTok{3}\NormalTok{), }\FunctionTok{c}\NormalTok{(}\DecValTok{4}\NormalTok{, }\DecValTok{5}\NormalTok{), }\FunctionTok{c}\NormalTok{(}\DecValTok{7}\NormalTok{, }\DecValTok{7}\NormalTok{)),}
          \FunctionTok{list}\NormalTok{(}\FunctionTok{c}\NormalTok{(}\DecValTok{3}\NormalTok{, }\DecValTok{3}\NormalTok{), }\FunctionTok{c}\NormalTok{(}\DecValTok{4}\NormalTok{, }\DecValTok{4}\NormalTok{), }\FunctionTok{c}\NormalTok{(}\DecValTok{5}\NormalTok{, }\DecValTok{5}\NormalTok{)))}
\NormalTok{pvalues }\OtherTok{\textless{}{-}} \FunctionTok{runif}\NormalTok{(}\DecValTok{25}\NormalTok{)}
\NormalTok{o }\OtherTok{\textless{}{-}} \FunctionTok{order}\NormalTok{(pvalues)}
\NormalTok{ZL }\OtherTok{\textless{}{-}} \FunctionTok{zetas.tree}\NormalTok{(C, leaf\_list, zeta.trivial, pvalues, }\AttributeTok{alpha =} \FloatTok{0.05}\NormalTok{)}
\NormalTok{complete.res }\OtherTok{\textless{}{-}} \FunctionTok{forest.completion}\NormalTok{(C, ZL, leaf\_list)}
\FunctionTok{curve.V.star.forest.fast}\NormalTok{(o, complete.res}\SpecialCharTok{$}\NormalTok{C, complete.res}\SpecialCharTok{$}\NormalTok{ZL, leaf\_list)}
\end{Highlighting}
\end{Shaded}

\begin{verbatim}
 [1]  1  2  3  4  5  6  7  8  9 10 11 12 13 14 15 16 17 18 19 20 21 22 23 24 25
\end{verbatim}

\begin{Shaded}
\begin{Highlighting}[]
\NormalTok{pruning.res }\OtherTok{\textless{}{-}} \FunctionTok{pruning}\NormalTok{(complete.res}\SpecialCharTok{$}\NormalTok{C, complete.res}\SpecialCharTok{$}\NormalTok{ZL, leaf\_list, }\AttributeTok{delete.gaps =} \ConstantTok{TRUE}\NormalTok{)}
\FunctionTok{curve.V.star.forest.fast}\NormalTok{(o, pruning.res}\SpecialCharTok{$}\NormalTok{C, pruning.res}\SpecialCharTok{$}\NormalTok{ZL, leaf\_list)}
\end{Highlighting}
\end{Shaded}

\begin{verbatim}
 [1]  1  2  3  4  5  6  7  8  9 10 11 12 13 14 15 16 17 18 19 20 21 22 23 24 25
\end{verbatim}

\section{Numerical experiments}\label{sec-numeric}

In this Section, we present some numerical experiments aiming to
demonstrate the impact of the pruning of Algorithm~\ref{algo-pruning}
(using the \texttt{delete.gaps} option mentioned in
Section~\ref{sec-implementation}) and of the fast
Algorithm~\ref{algo-curve}, in terms of computation time, compared to
the only previously available method to compute a curve of confidence
bounds. As mentioned in Section~\ref{sec-forest-structure} and
Section~\ref{sec-implementation}, this naive method simply consisted in
a \texttt{for} loop repeatedly applying Algorithm~\ref{algo-vstar}.

To compare the computation time, we use the \texttt{R} package
\texttt{microbenchmark} version 1.5.0 \citep{microbenchmark} with
\texttt{R} version 4.4.0 (2024-04-24) and \texttt{sanssouci} version
0.14.1, on a MacBook Air M1 (2020) running macOS 15.1.1. The package
\texttt{microbenchmark} allows to run code snippets a given number
\texttt{n\_repl} of times, and to compute summary statistics on the
computation time. The script executing the computation can be found in
the same repository as this manuscript.

Four scenarios are studied, all based on a common setting which we first
describe. A number \(m\) of hypotheses is tested. We use a reference
family \((R_k,\zeta_k)\) such that the \(R_k\)'s have a forest structure
of maximal depth \(H=10\). The graph of the inclusion relations between
the \(R_k\)'s is a binary tree, hence there are \(2^H-1=1023\) \(R_k\)'s
and in particular \(2^{H-1}=512\) atoms. The \(p\)-values are generated
in a gaussian one-sided fashion (see Example~\ref{exm-gauss}) where
\(H_{0,i}= \{ \mathcal N(\boldsymbol{\mu}, \mathrm{Id})  :\mu_i=0 \}\),
\(H_{1,i}=\{ \mathcal N(\boldsymbol{\mu}, \mathrm{Id}) :\mu_i=4 \}\).
\(\mathcal{H}_1\) is comprised of the leafs 1, 5, 9 and 10, that is
\(\mathcal{H}_1=P_1\cup P_5\cup P_9\cup P_{10}\). For each scenario, the
curve
\(\left(V^*_{\mathfrak{R}}(\{1,\dotsc,t \})\right)_{t\in\mathbb{N}_m^*}\)
is computed. For the experiments including pruning, the pruning is done
once before the \texttt{n\_repl} replications, to mimick the practice
where pruning only needs to be done once and for all, while the user may
be interested in computing multiple bounds and curves after that.

In scenarios 1 and 2, \(m=1024\) (so the atoms are of size 2), in
scenarios 3 and 4, \(m=10240\) (so the atoms are of size 10). In
scenarios 1 and 3, the \(\zeta_k\)'s are estimated trivially by
\(\zeta_k=|R_k|\), and in scenarios 2 and 4, they are computed as in
\citet{MR4178188} with the DKWM inequality \citep{MR0083864, MR1062069}.
Because of the size of \(m\) and the poor performances of the naive
approach, we set \texttt{n\_repl=100} in scenarios 1 and 2 and
\texttt{n\_repl=10} only in scenarios 3 and 4. The differences between
the scenarios are summarized in Table~\ref{tbl-scenarios}.

\begin{longtable}[]{@{}lllll@{}}
\caption{Differences between the
scenarios}\label{tbl-scenarios}\tabularnewline
\toprule\noalign{}
parameter & Scenario 1 & Scenario 2 & Scenario 3 & Scenario 4 \\
\midrule\noalign{}
\endfirsthead
\toprule\noalign{}
parameter & Scenario 1 & Scenario 2 & Scenario 3 & Scenario 4 \\
\midrule\noalign{}
\endhead
\bottomrule\noalign{}
\endlastfoot
\(m\) & 1024 & 1024 & 10240 & 10240 \\
zeta computation & trivial & DKWM & trivial & DKWM \\
\texttt{n\_repl} & 100 & 100 & 10 & 10 \\
\end{longtable}

For the trivial \(\zeta_k\) computation of scenarios 1 and 3, the
pruning obviously deletes all non-atom regions so
\(|\mathcal{K}^{\mathfrak{pr}}|=512\). Whereas, for the particular
instance \(\omega\in\Omega\) in the experiments,
\(|\mathcal{K}^{\mathfrak{pr}}|=541\) for scenario 2, and
\(|\mathcal{K}^{\mathfrak{pr}}|=573\) for scenario 4. Those results
alone illustrate the benefits of pruning with respect to the reduction
of the cardinality of the reference family: the regions above atoms with
no signal (or no detectable signal in the trivial scenarios) are pruned.
The fact that the regions above atoms with detectable signal are not
pruned means that they are relevant for the confidences bounds (which
had already been demonstrated in the simulation study of
\citet{MR4178188}).

The summary statistics of the computation time in each scenario are
presented in Table~\ref{tbl-benchmark01}, Table~\ref{tbl-benchmark02},
Table~\ref{tbl-benchmark03}, and Table~\ref{tbl-benchmark04}, and they
are also presented as boxplots in Figure~\ref{fig-benchmark01}. The time
unit is the second.

\begin{longtable}[]{@{}
  >{\raggedright\arraybackslash}p{(\columnwidth - 14\tabcolsep) * \real{0.2048}}
  >{\raggedleft\arraybackslash}p{(\columnwidth - 14\tabcolsep) * \real{0.1205}}
  >{\raggedleft\arraybackslash}p{(\columnwidth - 14\tabcolsep) * \real{0.1205}}
  >{\raggedleft\arraybackslash}p{(\columnwidth - 14\tabcolsep) * \real{0.1205}}
  >{\raggedleft\arraybackslash}p{(\columnwidth - 14\tabcolsep) * \real{0.1205}}
  >{\raggedleft\arraybackslash}p{(\columnwidth - 14\tabcolsep) * \real{0.1205}}
  >{\raggedleft\arraybackslash}p{(\columnwidth - 14\tabcolsep) * \real{0.1205}}
  >{\raggedleft\arraybackslash}p{(\columnwidth - 14\tabcolsep) * \real{0.0723}}@{}}

\caption{\label{tbl-benchmark01}Scenario 1}

\tabularnewline

\toprule\noalign{}
\begin{minipage}[b]{\linewidth}\raggedright
expr
\end{minipage} & \begin{minipage}[b]{\linewidth}\raggedleft
min
\end{minipage} & \begin{minipage}[b]{\linewidth}\raggedleft
lq
\end{minipage} & \begin{minipage}[b]{\linewidth}\raggedleft
mean
\end{minipage} & \begin{minipage}[b]{\linewidth}\raggedleft
median
\end{minipage} & \begin{minipage}[b]{\linewidth}\raggedleft
uq
\end{minipage} & \begin{minipage}[b]{\linewidth}\raggedleft
max
\end{minipage} & \begin{minipage}[b]{\linewidth}\raggedleft
neval
\end{minipage} \\
\midrule\noalign{}
\endhead
\bottomrule\noalign{}
\endlastfoot
naive.not.pruned & 3.6708287 & 3.8028650 & 3.8149199 & 3.8221756 &
3.8362092 & 3.9022797 & 100 \\
naive.pruned & 3.3147519 & 3.4198975 & 3.4353463 & 3.4470054 & 3.4657886
& 3.5459636 & 100 \\
fast.not.pruned & 0.0035286 & 0.0035779 & 0.0046194 & 0.0036011 &
0.0036321 & 0.1014023 & 100 \\
fast.pruned & 0.0011960 & 0.0012314 & 0.0012535 & 0.0012430 & 0.0012703
& 0.0013603 & 100 \\

\end{longtable}

\begin{longtable}[]{@{}
  >{\raggedright\arraybackslash}p{(\columnwidth - 14\tabcolsep) * \real{0.2024}}
  >{\raggedleft\arraybackslash}p{(\columnwidth - 14\tabcolsep) * \real{0.1190}}
  >{\raggedleft\arraybackslash}p{(\columnwidth - 14\tabcolsep) * \real{0.1190}}
  >{\raggedleft\arraybackslash}p{(\columnwidth - 14\tabcolsep) * \real{0.1190}}
  >{\raggedleft\arraybackslash}p{(\columnwidth - 14\tabcolsep) * \real{0.1190}}
  >{\raggedleft\arraybackslash}p{(\columnwidth - 14\tabcolsep) * \real{0.1190}}
  >{\raggedleft\arraybackslash}p{(\columnwidth - 14\tabcolsep) * \real{0.1310}}
  >{\raggedleft\arraybackslash}p{(\columnwidth - 14\tabcolsep) * \real{0.0714}}@{}}

\caption{\label{tbl-benchmark02}Scenario 2}

\tabularnewline

\toprule\noalign{}
\begin{minipage}[b]{\linewidth}\raggedright
expr
\end{minipage} & \begin{minipage}[b]{\linewidth}\raggedleft
min
\end{minipage} & \begin{minipage}[b]{\linewidth}\raggedleft
lq
\end{minipage} & \begin{minipage}[b]{\linewidth}\raggedleft
mean
\end{minipage} & \begin{minipage}[b]{\linewidth}\raggedleft
median
\end{minipage} & \begin{minipage}[b]{\linewidth}\raggedleft
uq
\end{minipage} & \begin{minipage}[b]{\linewidth}\raggedleft
max
\end{minipage} & \begin{minipage}[b]{\linewidth}\raggedleft
neval
\end{minipage} \\
\midrule\noalign{}
\endhead
\bottomrule\noalign{}
\endlastfoot
naive.not.pruned & 3.7152477 & 3.8110591 & 3.9803535 & 3.8483790 &
3.9549886 & 10.1338336 & 100 \\
naive.pruned & 3.3277028 & 3.4592016 & 3.5465768 & 3.5060270 & 3.6059210
& 5.4159371 & 100 \\
fast.not.pruned & 0.0032789 & 0.0033216 & 0.0067553 & 0.0033482 &
0.0033857 & 0.1978229 & 100 \\
fast.pruned & 0.0013597 & 0.0013884 & 0.0014134 & 0.0014056 & 0.0014298
& 0.0017731 & 100 \\

\end{longtable}

\begin{longtable}[]{@{}
  >{\raggedright\arraybackslash}p{(\columnwidth - 14\tabcolsep) * \real{0.1789}}
  >{\raggedleft\arraybackslash}p{(\columnwidth - 14\tabcolsep) * \real{0.1263}}
  >{\raggedleft\arraybackslash}p{(\columnwidth - 14\tabcolsep) * \real{0.1263}}
  >{\raggedleft\arraybackslash}p{(\columnwidth - 14\tabcolsep) * \real{0.1263}}
  >{\raggedleft\arraybackslash}p{(\columnwidth - 14\tabcolsep) * \real{0.1263}}
  >{\raggedleft\arraybackslash}p{(\columnwidth - 14\tabcolsep) * \real{0.1263}}
  >{\raggedleft\arraybackslash}p{(\columnwidth - 14\tabcolsep) * \real{0.1263}}
  >{\raggedleft\arraybackslash}p{(\columnwidth - 14\tabcolsep) * \real{0.0632}}@{}}

\caption{\label{tbl-benchmark03}Scenario 3}

\tabularnewline

\toprule\noalign{}
\begin{minipage}[b]{\linewidth}\raggedright
expr
\end{minipage} & \begin{minipage}[b]{\linewidth}\raggedleft
min
\end{minipage} & \begin{minipage}[b]{\linewidth}\raggedleft
lq
\end{minipage} & \begin{minipage}[b]{\linewidth}\raggedleft
mean
\end{minipage} & \begin{minipage}[b]{\linewidth}\raggedleft
median
\end{minipage} & \begin{minipage}[b]{\linewidth}\raggedleft
uq
\end{minipage} & \begin{minipage}[b]{\linewidth}\raggedleft
max
\end{minipage} & \begin{minipage}[b]{\linewidth}\raggedleft
neval
\end{minipage} \\
\midrule\noalign{}
\endhead
\bottomrule\noalign{}
\endlastfoot
naive.not.pruned & 336.0473732 & 336.7254511 & 338.6804399 & 337.0286221
& 340.9009506 & 344.4716282 & 10 \\
naive.pruned & 332.4762463 & 332.8188433 & 334.4660587 & 334.1282526 &
335.5376761 & 337.6580202 & 10 \\
fast.not.pruned & 0.0323725 & 0.0324755 & 0.0328789 & 0.0325803 &
0.0328097 & 0.0354455 & 10 \\
fast.pruned & 0.0099485 & 0.0100272 & 0.0101948 & 0.0101886 & 0.0102164
& 0.0107677 & 10 \\

\end{longtable}

\begin{longtable}[]{@{}
  >{\raggedright\arraybackslash}p{(\columnwidth - 14\tabcolsep) * \real{0.1789}}
  >{\raggedleft\arraybackslash}p{(\columnwidth - 14\tabcolsep) * \real{0.1263}}
  >{\raggedleft\arraybackslash}p{(\columnwidth - 14\tabcolsep) * \real{0.1263}}
  >{\raggedleft\arraybackslash}p{(\columnwidth - 14\tabcolsep) * \real{0.1263}}
  >{\raggedleft\arraybackslash}p{(\columnwidth - 14\tabcolsep) * \real{0.1263}}
  >{\raggedleft\arraybackslash}p{(\columnwidth - 14\tabcolsep) * \real{0.1263}}
  >{\raggedleft\arraybackslash}p{(\columnwidth - 14\tabcolsep) * \real{0.1263}}
  >{\raggedleft\arraybackslash}p{(\columnwidth - 14\tabcolsep) * \real{0.0632}}@{}}

\caption{\label{tbl-benchmark04}Scenario 4}

\tabularnewline

\toprule\noalign{}
\begin{minipage}[b]{\linewidth}\raggedright
expr
\end{minipage} & \begin{minipage}[b]{\linewidth}\raggedleft
min
\end{minipage} & \begin{minipage}[b]{\linewidth}\raggedleft
lq
\end{minipage} & \begin{minipage}[b]{\linewidth}\raggedleft
mean
\end{minipage} & \begin{minipage}[b]{\linewidth}\raggedleft
median
\end{minipage} & \begin{minipage}[b]{\linewidth}\raggedleft
uq
\end{minipage} & \begin{minipage}[b]{\linewidth}\raggedleft
max
\end{minipage} & \begin{minipage}[b]{\linewidth}\raggedleft
neval
\end{minipage} \\
\midrule\noalign{}
\endhead
\bottomrule\noalign{}
\endlastfoot
naive.not.pruned & 340.4702704 & 341.4280632 & 344.8181652 & 344.3519587
& 348.6074564 & 350.4574742 & 10 \\
naive.pruned & 337.2238865 & 338.1905987 & 340.4203488 & 340.4743933 &
342.7488030 & 344.1039957 & 10 \\
fast.not.pruned & 0.0294732 & 0.0296390 & 0.0299436 & 0.0298885 &
0.0300172 & 0.0307673 & 10 \\
fast.pruned & 0.0124157 & 0.0126186 & 0.0137188 & 0.0127803 & 0.0130546
& 0.0194847 & 10 \\

\end{longtable}

\begin{figure}

\centering{

\includegraphics{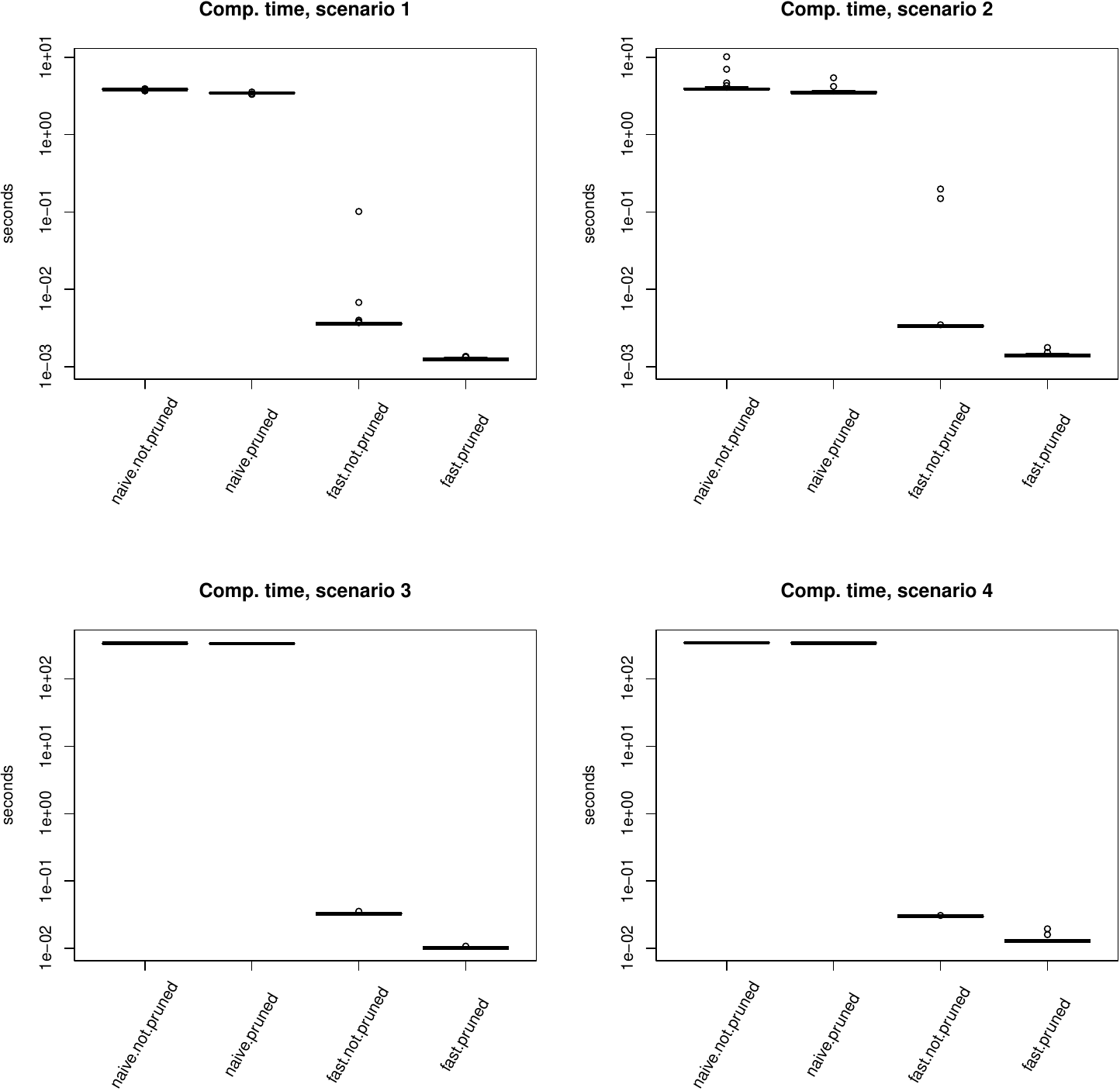}

}

\caption{\label{fig-benchmark01}Computation times in each scenario, in
seconds (using a logarithmic scale)}

\end{figure}%

In each scenario, using the fast algorithm is much faster than the naive
approach, with a speed factor of at least 1000. Using the naive
approach, pruning always gives a slight improvement over not pruning.
Using the fast algorithm, the benefits of pruning are significant, with
a speed factor of at least 2, and sometimes 3.

Comparing scenarios 1 and 2 first, we see that, as expected, there is no
significant change in computation time for \texttt{naive.not.pruned}.
Methods \texttt{naive.pruned} and \texttt{fast.pruned} are faster in
scenario 1, given that we prune more. But, on the other hand,
\texttt{fast.not.pruned} is slightly faster in scenario 2. This is
because, for the regions with signal, said signal is detected and so
those regions are quickly saturated, in the sense that we quickly have
\(\eta_k^t=\zeta_k\) and \(k\) added to \(\mathcal{K}^-_k\), which saves
a lot of time.

The comparison between scenarios 3 and 4 is similar. Although, with only
\texttt{n\_repl=10}, the statistics seem less accurate, this can be
confirmed with additional experiments (\texttt{n\_repl} can also be set
to 100 without problem if we don't include \texttt{naive} methods).

Finally, comparing scenarios 3 \& 4 with scenarios 1 \& 2, we see that
multiplying the number of hypotheses by 10 effectively multiplies the
computation time by \(\sim10\) when using Algorithm~\ref{algo-curve} and
by \(\sim100\) when using Algorithm~\ref{algo-vstar} naively, which
illustrates the stated complexities of \(O(|\mathcal{K}|m)\) and
\(O(|\mathcal{K}|m^2)\), respectively.

\section{Conclusion}\label{sec-conclusion}

In conclusion, we effectively introduced a new algorithm to compute a
curve of confidence upper bounds for the false discoveries, or,
equivalently, for the FDP, that is much faster than the previous
alternative, with one power of \(m\) less in the complexity. This
algorithm can be applied as soon as the confidence upper bound is built
according to the JER framework, when the reference family exhibit a
forest structure.

To develop new confidence upper bound methods and test them on
simulations, it was previously not conceivable to replicate experiments
a sufficient number of times while computing whole curves. For instance,
in the simulation study of \citet{MR4178188}, the number of replications
chosen was 10 and the whole curve was not computed, only ten values
along the curve were computed, for an \texttt{m} set to 12800, that is
0.078\% of the curve had been computed. Now, simulation studies with an
adequate number of replications and 100\% of the curve become feasible.

A lot of work remains to be done on the \texttt{sanssouci} package. For
example, to make the data format of a forest structure
\((R_k)_{k\in\mathcal{K}}\) less convoluted and more user-friendly is an
interesting project. Another one would be to implement inside the
package the methods of the paper \citet{blain22notip}, which are
currently only available in the Python language \citep{10.5555/1593511},
and the methods of the paper \citet{JMLR:v25:23-1025}.

Other current works include the development of new reference families
with theoretically proven JER control that could better account for
realistic models, such as models with dependence between the
\(p\)-values, see for example \citet{perrot2023selective}, or models
with discreteness.

\section{Proofs}\label{sec-proofs}

\subsection{\texorpdfstring{Proofs of
Section~\ref{sec-pruning}}{Proofs of Section~}}\label{sec-pruning-proofs}

\subsubsection{\texorpdfstring{Proof of
Proposition~\ref{prp-pruning}}{Proof of Proposition~}}\label{sec-pruning-proofs-pruning}

Recall Equation \eqref{eq_vstar_Q} and, because
\(\mathfrak{R}^{\mathfrak{pr}}\) also has a forest structure,
\begin{equation}
V^*_{\mathfrak{R}^{\mathfrak{pr}}}(S)=\min_{Q\subseteq\mathcal{K}^{\mathfrak{pr}}}\left(\sum_{k'\in Q}\zeta_{k'}\wedge|S\cap R_{k'}|+\left| S\setminus\bigcup_{k'\in Q} R_{k'}   \right|\right),
\label{eq_vstarpruned_Q}
\end{equation} so we immediately get that
\(V^*_{\mathfrak{R}}(S)\leq V^*_{\mathfrak{R}^{\mathfrak{pr}}}(S)\).

Let any \(Q\subseteq \mathcal{K}\). We split \(Q\) in \(A\) elements of
\(\mathcal{K}\setminus\mathcal{K}^{\mathfrak{pr}}\), denoted
\((i_{0,a}, i_{p_a,a}-1)\), \(1\leq a\leq A\) for some \(p_a\geq2\), and
\(B\) elements of \(\mathcal{K}^{\mathfrak{pr}}\), simply denoted
\(k_b\), \(1\leq b\leq B\). By the definition of
\(\mathcal{K}^{\mathfrak{pr}}\) and the previous remarks, for any
\(1\leq a \leq A\), there exist integers \(i_{1,a},\dotsc,i_{p_a-1,a}\)
such that \(i_{0,a}<i_{1,a}<\dotsb<i_{p_a-1,a} < i_{p_a,a}\),
\((i_{j-1,a},i_{j,a}-1)\in\mathcal{K}^{\mathfrak{pr}}\) for all
\(1\leq j\leq p_a\), and
\(\zeta_{(i_{0,a}, i_{p_a,a}-1)}\geq \sum_{j=1}^{p_a}\zeta_{(i_{j-1,a},i_{j,a}-1)}\).
Now let \begin{equation}
Q^{\mathfrak{pr}}=\{k_b : 1\leq b\leq B \} \cup \{ (i_{j-1,a},i_{j,a}-1) :  1\leq a\leq A, 1\leq j\leq p_a  \}.
\label{eq_Qpr}
\end{equation} We have that
\(Q^{\mathfrak{pr}}\subseteq \mathcal{K}^{\mathfrak{pr}}\) and
\(\bigcup_{k\in Q}R_k=\bigcup_{k\in Q^{\mathfrak{pr}}}R_k\). Then,
\begin{align*}
\sum_{k\in Q}\zeta_k\wedge|S\cap R_k|+\left| S\setminus\bigcup_{k\in Q} R_k   \right|&=\sum_{b=1}^B\zeta_{k_b}\wedge|S\cap R_{k_b}| \\
&\qquad+\sum_{a=1}^A\zeta_{(i_{0,a}, i_{p_a,a}-1)}\wedge |S\cap R_{(i_{0,a}, i_{p_a,a}-1)}| \\
&\qquad+ \left| S\setminus\bigcup_{k\in Q} R_k   \right|    , 
\end{align*} but for all \(1\leq a\leq A\), \begin{align*}
\zeta_{(i_{0,a}, i_{p_a,a}-1)}&\geq  \sum_{j=1}^{p_a}\zeta_{(i_{j-1,a},i_{j,a}-1)} \\
&\geq   \sum_{j=1}^{p_a}\zeta_{(i_{j-1,a},i_{j,a}-1)} \wedge |S\cap R_{(i_{j-1,a}, i_{j,a}-1)}| ,
\end{align*} so the term
\(\sum_{a=1}^A\zeta_{(i_{0,a}, i_{p_a,a}-1)}\wedge |S\cap R_{(i_{0,a}, i_{p_a,a}-1)}|\)
is greater than or equal to \begin{equation*}
\sum_{a=1}^A\left(  \sum_{j=1}^{p_a}\zeta_{(i_{j-1,a},i_{j,a}-1)} \wedge |S\cap R_{(i_{j-1,a}, i_{j,a}-1)}| \right)\wedge |S\cap R_{(i_{0,a}, i_{p_a,a}-1)}| ,
\end{equation*} which is simply equal to \begin{equation*}
 \sum_{a=1}^A  \sum_{j=1}^{p_a}\zeta_{(i_{j-1,a},i_{j,a}-1)} \wedge |S\cap R_{(i_{j-1,a}, i_{j,a}-1)}|.
\end{equation*} Furthermore
\(\left|S\setminus\bigcup_{k\in Q} R_k\right|= \left|S\setminus\bigcup_{k\in Q^{\mathfrak{pr}}} R_k\right|\)
so finally: \begin{align}
\sum_{k\in Q}\zeta_k\wedge|S\cap R_k|+\left| S\setminus\bigcup_{k\in Q} R_k   \right| &\geq \sum_{k\in Q^{\mathfrak{pr}}}\zeta_k\wedge|S\cap R_k|+\left| S\setminus\bigcup_{k\in Q^{\mathfrak{pr}}} R_k   \right|   \label{pruning_ineq}\\
&\geq V^*_{\mathfrak{R}^{\mathfrak{pr}}}(S). \notag
\end{align} Note that Equation \eqref{pruning_ineq} is true even if
there are some
\(b\in\{ 1,\dotsc,B\}, a\in\{ 1,\dotsc,A\}, j\in\{ 1,\dotsc,p_a\}\) such
that \(k_b=(i_{j-1,a}, i_{j,a}-1)\). We minimize over all \(Q\) to get
that
\(V^*_{\mathfrak{R}}(S)\geq V^*_{\mathfrak{R}^{\mathfrak{pr}}}(S)\).

\subsubsection{\texorpdfstring{Proof of
Proposition~\ref{prp-pruning-correct}}{Proof of Proposition~}}\label{sec-pruning-proofs-pruning-correct}

First,
\(\mathcal{K}\setminus\mathcal{L}\subseteq\mathcal{K}\setminus\mathcal{K}^{\mathfrak{pr}}\)
is trivial: a \(k\) such that
\(\zeta_{k} \geq  \sum_{k'\in Succ_k} Vec_{k'}\) obviously satisfies the
condition of Definition~\ref{def-pruning} to be pruned.

Now let \((i,i')\in \mathcal{K}\setminus\mathcal{K}^{\mathfrak{pr}}\) an
element that is pruned by Definition~\ref{def-pruning}, so there exists
\(p\geq2\) and integers \(i_1,\dotsc,i_{p-1}\) such that, when setting
\(i_0=i\) and \(i_{p}=i'+1\), the sequence \((i_0,\dotsc,i_{p})\) is
strictly increasing, \((i_{j-1},i_{j}-1)\in\mathcal{K}\) for all
\(1\leq j\leq p\) and finally
\(\zeta_{(i,i')}=\zeta_{(i_0,i_{p}-1)}\geq \sum_{j=1}^{p} \zeta_{(i_{j-1}, i_{j}-1)}\).
Then by the proof of Theorem 1 of \citet{MR4178188} but applied to
\(S=R_{(i,i')}\) we have that
\(\sum_{j=1}^{p} \zeta_{(i_{j-1}, i_{j}-1)}\geq  \sum_{k'\in Succ_{(i,i')}} Vec_{k'}\)
(see the unnumbered line just above Equation (A4) in that paper) and so
\(\zeta_{(i,i')}\geq \sum_{k'\in Succ_{(i,i')}} Vec_{k'}\) hence
\((i,i')\) is pruned by Algorithm~\ref{algo-pruning} and
\(\mathcal{K}\setminus\mathcal{K}^{\mathfrak{pr}}\subseteq\mathcal{K}\setminus\mathcal{L}\).

In the end,
\(\mathcal{K}\setminus\mathcal{K}^{\mathfrak{pr}}=\mathcal{K}\setminus\mathcal{L}\)
so \(\mathcal{K}^{\mathfrak{pr}}=\mathcal{L}\).

\subsection{\texorpdfstring{Proof of
Theorem~\ref{thm-curve-path}}{Proof of Theorem~}}\label{sec-proof}

In this section, every reference to a line is a reference to a line of
Algorithm~\ref{algo-formal-curve}.

\subsubsection{\texorpdfstring{Derivation of
\eqref{eq_vstar_equal_sum_eta}}{Derivation of }}\label{derivation-of}

We first derive \eqref{eq_vstar_equal_sum_eta} from
\eqref{eq_vstar_inter_Rk_equal_eta} and \eqref{eq_Pt_good_partition}.
First note that for all \(Q\in\mathfrak P\), \begin{equation}
Q=\bigcup_{k\in\mathcal{K}^1}\{ k' \in Q : R_{k'}\subseteq R_k\}
\label{eq_part_K1}
\end{equation} and the union is disjoint. From
\eqref{eq_vstar_Qpartition}, let \(Q^*\in\mathfrak P\) such that
\(V^*_{\mathfrak{R}}(S_t) = \sum_{k'\in Q^*} \zeta_{k'}\wedge |S_t\cap R_{k'}|\).
Then by \eqref{eq_part_K1}, \begin{align}
V^*_{\mathfrak{R}}(S_t) &= \sum_{k'\in Q^*} \zeta_{k'}\wedge |S_t\cap R_{k'}|\notag\\
&=\sum_{k\in\mathcal{K}^1}  \sum_{\substack{k'\in Q^*\\ R_{k'}\subseteq R_k}} \zeta_{k'}\wedge |S_t\cap R_{k'}|\notag \\
&=\sum_{k\in\mathcal{K}^1}  \sum_{\substack{k'\in Q^*\\ R_{k'}\subseteq R_k}} \zeta_{k'}\wedge |S_t\cap (R_{k}\cap R_{k'})| \notag\\
&=\sum_{k\in\mathcal{K}^1}  \sum_{k'\in Q^*} \zeta_{k'}\wedge |(S_t\cap R_{k})\cap R_{k'}| \label{eq_delicate} \\
&\geq \sum_{k\in\mathcal{K}^1} V^*_{\mathfrak{R}}(S_t\cap R_k),\label{eq_delicate_vstar}
\end{align} where the equality in \eqref{eq_delicate} comes from the
fact that if \(R_{k'}\not\subseteq R_k\), then
\(R_{k'}\cap R_k=\varnothing\), that is, \(R_{k}\subsetneq R_{k'}\) is
impossible because \(k\in\mathcal K^1\). Furthermore,
\eqref{eq_delicate_vstar} holds again by \eqref{eq_vstar_Qpartition}.

Because \(\mathcal{K}^1\subseteq\mathcal{K}_t\), by
\eqref{eq_Pt_good_partition},
\(V^*_{\mathfrak{R}}(S_t\cap R_k) = \sum_{\substack{k'\in \mathcal{P}^t\\ R_{k'}\subseteq R_k}} \zeta_{k'}\wedge|S_t \cap R_{k'}|\)
for all \(k\in\mathcal{K}^1\). Then, \begin{align*}
 \sum_{k\in\mathcal{K}^1} V^*_{\mathfrak{R}}(S_t\cap R_k)&=  \sum_{k\in\mathcal{K}^1}\sum_{\substack{k'\in \mathcal{P}^t\\ R_{k'}\subseteq R_k}} \zeta_{k'}\wedge|S_t \cap R_{k'}|\\
 &=\sum_{k\in\mathcal{P}^t} \zeta_{k}\wedge |S_t\cap R_{k}| \text{ by \eqref{eq_part_K1}}\\
 &\geq V^*_{\mathfrak{R}}(S_t) \text{ by \eqref{eq_vstar_Qpartition}}.\\
\end{align*} So we proved that
\(V^*_{\mathfrak{R}}(S_t)= \sum_{k\in\mathcal{P}^t} \zeta_{k}\wedge |S_t\cap R_{k}|= \sum_{k\in\mathcal{K}^1} V^*_{\mathfrak{R}}(S_t\cap R_k)\)
and finally
\(V^*_{\mathfrak{R}}(S_t)=\sum_{k\in\mathcal{K}^1} V^*_{\mathfrak{R}}(S_t\cap R_k)= \sum_{k\in\mathcal{K}^1}  \eta_k^t\)
by \eqref{eq_vstar_inter_Rk_equal_eta}, again because
\(\mathcal{K}^1\subseteq\mathcal{K}_t\). Every equality in
\eqref{eq_vstar_equal_sum_eta} is proven.

\subsubsection{\texorpdfstring{Proof that
\(\mathcal{P}^t\in\mathfrak P\)}{Proof that \textbackslash mathcal\{P\}\^{}t\textbackslash in\textbackslash mathfrak P}}\label{proof-that-mathcalptinmathfrak-p}

By completeness, \(F\subseteq \mathcal{K}\) and so
\(\mathcal{P}^0\subseteq \mathcal{K}\). First we show that
\(\mathbb{N}_m^*=\bigcup_{k\in\mathcal{P}^0}R_k\). Let
\(j\in\mathbb{N}_m^*\) and \(i\in \mathbb{N}_N^*\) such that
\(j\in R_{(i,i)}\). Let
\(G=\left\{k\in\mathcal{K}_0^-: j\in R_k\right\}\). If
\(G=\varnothing\), for any \(k\in\mathcal{K}_0^-\),
\(R_{(i,i)}\subseteq R_k\) would imply that \(j\in R_k\) and \(k\in G\),
hence a contradiction, and so \((i,i)\in F\) and
\(j\in \bigcup_{k\in\mathcal{P}^0}R_k\). If \(G\neq\varnothing\), for
any \(k,k'\in G\), \(j\in R_{k}\cap R_{k'}\) so, by forest structure,
\(R_{k}\subseteq R_{k'}\) or \(R_{k'}\subseteq R_{k}\), hence
\(\subseteq\) is a total order on the finite, non-empty set
\(\{R_k : k\in G\}\), so the latter has a maximum and there exists a
unique \(k^*\in G\) such that \(R_{k^*}=\max_{k\in G}R_k\). Let us show
that, as an element of \(\mathcal{K}_0^-\), \(k^*\) is maximal, which
will imply that \(k^*\in E\) and that
\(j\in \bigcup_{k\in\mathcal{P}^0}R_k\). Let any
\(k'\in \mathcal{K}_0^-\) such that \(R_{k^*}\subseteq R_{k'}\), then
\(j\in R_{k'}\), so \(k'\in G\) and so
\(R_{k'}\subseteq \max_{k\in G}R_k=R_{k^*}\). Hence \(R_{k^*}= R_{k'}\)
and \(k^*=k'\) (see Remark~\ref{rem-distinct}), \(k^*\) is indeed
maximal, \(k^*\in E\) and \(j\in \bigcup_{k\in\mathcal{P}^0}R_k\). This
proves that \(\mathbb{N}_m^*=\bigcup_{k\in\mathcal{P}^0}R_k\).

Now let us prove that the elements \(\mathcal{P}^0\) index disjoints
sets. Let \(k, k'\in \mathcal{P}^0\) such that there exists
\(j\in R_{k}\in R_{k'}\). By forest structure, \(R_{k}\subseteq R_{k'}\)
or \(R_{k'}\subseteq R_{k}\). Now we separate four cases. The case
\(k\in F, k'\in E\) is impossible, because \(k\) would be the index of
an atom, so we would imperatively have \(R_{k}\subseteq R_{k'}\), which
would contradict the definition of \(F\). Similarly, \(k\in F', k\in E\)
is impossible. If \(k, k'\in E\), then \(k=k'\) by the very definition
of \(E\). Finally, if \(k\in F, k'\in F\), \(R_k\) and \(R_{k'}\) are
both atoms, and because the atoms realize a partition of
\(\mathbb{N}_m^*\), then \(k=k'\). In all cases, \(k=k'\) which
concludes.

The \(R_k, k\in \mathcal{P}^0\), are disjoint, non-empty (see
Remark~\ref{rem-distinct}), and cover \(\mathbb{N}_m^*\), so they form a
partition of \(\mathbb{N}_m^*\), in other words,
\(\mathcal{P}^0\in\mathfrak P\).

We then show that \(\mathcal{P}^t\in\mathfrak P\) by recursion. We just
showed that \(\mathcal{P}^0\in\mathfrak P\). Let
\(t\in\{ 0,\dotsc,m-1\}\) and assume that
\(\mathcal{P}^t\in\mathfrak P\). In many cases,
\(\mathcal{P}^{t+1}=\mathcal{P}^t\) and so
\(\mathcal{P}^{t+1}\in\mathfrak P\) by the recursion hypothesis.
Otherwise, \(\mathcal{P}^{t+1}\) is given by the adjustment in line 17,
in which case we have \begin{equation}\label{eq_rel_rec_p}
\mathcal{P}^{t+1}=\left(\mathcal{P}^t \setminus\{k\in\mathcal{P}^t, R_k\subseteq R_{  k^{(t+1,h^f_{t+1})}   }   \}  \right)\cup \{ k^{(t+1,h^f_{t+1})} \}.
\end{equation} Note that this imply that
\(i_{t+1}\not\in\bigcup_{\kappa \in \mathcal{K}^-_t} R_{\kappa}\) (see
lines 6 and 10).

Let \(j\in\mathbb{N}_m^*\) and \(k\in\mathcal{P}^t\) such that
\(j\in R_k\). If \(R_k\subseteq R_{  k^{(t+1,h^f_{t+1})}   }\), then
\(j\in R_{  k^{(t+1,h^f_{t+1})}   }\subseteq \bigcup_{\kappa\in \mathcal{P}^{t+1}} R_{\kappa}\),
and if not, \(k\in\mathcal{P}^{t+1}\) and again
\(j\in \bigcup_{\kappa\in \mathcal{P}^{t+1}} R_{\kappa}\).

Now let \(k,k'\in\mathcal{P}^{t+1}\), \(k\neq k'\). If both are
different from \(k^{(t+1,h^f_{t+1})}\), then they are both in
\(\mathcal{P}^t\in\mathfrak P\) so \(R_k\cap R_{k'}=\varnothing\).
Assume that \(k'=k^{(t+1,h^f_{t+1})}\). By the forest structure,
\(R_k\cap R_{k'}=\varnothing\), or \(R_{k^{(t+1,h^f_{t+1})}}\), or
\(R_k\). By \eqref{eq_rel_rec_p}, the latter is impossible. It remains
to show that \(R_k\cap R_{k'}=R_{  k^{(t+1,h^f_{t+1})}}\) is also not
possible, in other words that we can't have
\(R_{  k^{(t+1,h^f_{t+1})}} \subsetneq  R_{k}\). If that was the case,
because of the strict inclusion, we would have \(k\not\in F\) because
\(R_{k}\) could not be an atom, so we would have either
\(k\in E\subseteq \mathcal{K}^-_0 \subseteq \mathcal{K}^-_t\) or
\(k\in\mathcal{P}^t\setminus\mathcal{P}^0\). In the second case, \(k\)
would have been added to \(\mathcal{P}^{t'}\) at a previous step
\(t'\leq t\) of the algorithm, but in that case it would also have been
added to \(\mathcal{K}_{t'}^-\subseteq \mathcal{K}_{t}^-\) (see lines 17
and 18). So in the end, in both cases, we would have \begin{equation*}
i_{t+1}\in R_{  k^{(t+1,h^f_{t+1})}}  \subsetneq  R_{k} \subseteq \bigcup_{\kappa\in\mathcal{K}^-_{t}}R_{\kappa}
\end{equation*} which is a contradiction with the fact that
\(i_{t+1}\not\in\bigcup_{\kappa \in \mathcal{K}^-_t} R_{\kappa}\), and
so \(R_k\cap R_{k'}=\varnothing\), and finally
\(\mathcal{P}^{t+1}\in\mathfrak P\).

\subsubsection{\texorpdfstring{Proof of
\eqref{eq_vstar_inter_Rk_equal_eta} and
\eqref{eq_Pt_good_partition}}{Proof of  and }}\label{proof-of-and}

We show the remainder of the statements with a strong recursion over
\(t\). We have \(\mathcal{P}^0\in\mathfrak P\) by previous section, and
given that \(S_0=\varnothing\) and \(\eta^0_k=0\) for all
\(k\in\mathcal{K}\) (recall that \(\mathcal{K}_0=\mathcal{K}\)),
everything is equal to 0 in \eqref{eq_vstar_inter_Rk_equal_eta} and
\eqref{eq_Pt_good_partition}.

So we let \(t\in\{ 0,\dotsc,m-1\}\), and assume that
\(\mathcal{P}^{t'}\in\mathfrak P\) and that
\eqref{eq_vstar_inter_Rk_equal_eta} and \eqref{eq_Pt_good_partition}
hold for all \(t'\leq t\). In all the following, \(\bar k\) is the
element of \(\mathcal{P}^t\) such that \(i_{t+1}\in R_{\bar k}\). We
will distinguish two cases: if
\(i_{t+1}\in\bigcup_{\kappa\in\mathcal{K}^-_{t}}R_{\kappa}\) or not.
First we show an inequality that will be used in both cases. We have,
for all \(k\in\mathcal{K}_t\), \begin{align}
V^*_{\mathfrak{R}}(S_{t+1}\cap R_k)&\leq \sum_{\substack{k'\in\mathcal{P}^{t}\\R_{k'}\subseteq R_k}} \zeta_{k'}\wedge |S_{t+1}\cap R_{k'}|.\label{eq_ineq_both_cases}
\end{align} Indeed, by \eqref{eq_vstar_Qpartition}, \begin{align*}
V^*_{\mathfrak{R}}(S_{t+1}\cap R_k)&\leq  \sum_{k'\in\mathcal{P}^{t}} \zeta_{k'}\wedge |S_{t+1}\cap R_k \cap R_{k'}|.
\end{align*} For any \(k'\in\mathcal{P}^{t}\), we have either
\(R_{k'}\cap R_k=\varnothing\), in which case
\(|S_{t+1}\cap R_k \cap R_{k'}|=0\), either \(R_{k'}\subseteq R_k\), in
which case \(|S_{t+1}\cap R_k \cap R_{k'}|=|S_{t+1} \cap R_{k'}|\), but
\(R_{k}\subsetneq R_{k'}\) is impossible. Indeed, by definition of
\(\mathcal{K}_t\), there exists \(\tilde k\in \mathcal{P}^t\) such that
\(R_{\tilde k}\subseteq R_k\), so \(R_{k}\subsetneq R_{k'}\) would
entail \(R_{\tilde k}\subsetneq R_{k'}\) which is impossible since
\(k', \tilde k\in \mathcal{P}^t\in\mathfrak P\) and so \(R_{\tilde k}\)
and \(R_{k'}\) are part of a partition of \(\mathbb{N}_m^*\). This gives
\eqref{eq_ineq_both_cases}.

\paragraph{\texorpdfstring{First case:
\(i_{t+1}\in\bigcup_{\kappa\in\mathcal{K}^-_{t}}R_{\kappa}\)}{First case: i\_\{t+1\}\textbackslash in\textbackslash bigcup\_\{\textbackslash kappa\textbackslash in\textbackslash mathcal\{K\}\^{}-\_\{t\}\}R\_\{\textbackslash kappa\}}}\label{first-case-i_t1inbigcup_kappainmathcalk-_tr_kappa}

In this case, \(\mathcal{P}^{t+1}=\mathcal{P}^t\) and
\(\mathcal{K}_{t+1}=\mathcal{K}_t\). For any \(k\in\mathcal{K}_{t+1}\)
such that \(i_{t+1}\not\in R_k\) (or, equivalently, such that
\(S_{t+1}\cap R_k=S_t\cap R_k\)), \begin{align*}
\sum_{\substack{k'\in\mathcal{P}^{t+1} \\ R_{k'}\subseteq R_k}} \zeta_{k'}\wedge |S_{t+1}\cap R_{k'}|&=\sum_{\substack{k'\in\mathcal{P}^{t} \\ R_{k'}\subseteq R_k}} \zeta_{k'}\wedge |S_{t}\cap R_{k'}| \\
&=V^*_{\mathfrak{R}}(S_t\cap R_k) \text{ by~\eqref{eq_Pt_good_partition}}\\
&=\eta_k^t \text{ by~\eqref{eq_vstar_inter_Rk_equal_eta}}\\
&=\eta_k^{t+1}
\end{align*} because \(\eta_k^t=\eta_k^{t+1}\) for all
\(k\in\mathcal{K}\). Furthermore \(S_{t+1}\cap R_k=S_t\cap R_k\) so
\(V^*_{\mathfrak{R}}(S_{t+1}\cap R_k)=V^*_{\mathfrak{R}}(S_t\cap R_k)\).
So everything is proved for such a \(k\).

Now we let \(k\in\mathcal{K}_{t+1}\) such that \(i_{t+1}\in R_k\) or,
equivalently, such that \(R_{\bar k}\subseteq R_k\). We first need to
show that \(\zeta_{\bar k}\leq |S_t\cap R_{\bar k}|\), and for that we
need to distinguish two subcases: if \(\bar k\) has been added to
\(\mathcal{P}^t\) during a previous step of the algorithm (see line 17),
or if not. Note that \(\bar k\) being added to \(\mathcal{P}^t\) during
a previous step means that there exists \(t'\), \(1\leq t'\leq t\), such
that \(h^f_{t'}\) is defined and \(\bar k= k^{(t',h^f_{t'})}\). The
contrary means that for all \(t'\), \(1\leq t'\leq t\) such that
\(h^f_{t'}\) is defined, \(\bar k\neq k^{(t',h^f_{t'})}\), and also that
\(\bar k\in\mathcal{P}^{t'}\) for all \(t'\in\{0,\dotsc,t\}\).

\subparagraph{\texorpdfstring{First subcase: \(\bar k\) has never been
added during the process of line
17}{First subcase: \textbackslash bar k has never been added during the process of line 17}}\label{first-subcase-bar-k-has-never-been-added-during-the-process-of-line-17}

As remarked just above, \(\bar k\in\mathcal{P}^{0}=E\cup F\). Our goal
is to show that \(\bar k\not\in F\). This will imply that
\(\bar k\in E\subseteq \mathcal{K}_0^-\) and so that
\(\zeta_{\bar k}=0\leq| S_{t}\cap R_{\bar k} |\) as desired.

Let us assume that \(\bar k\in F\) and find a contradiction. \(\bar k\)
is then the index of an atom that contains \(i_{t+1}\), and
\(i_{t+1}\in\bigcup_{\kappa\in\mathcal{K}^-_{t}}R_{\kappa}\), so there
exists \(k'\in\mathcal{K}^-_{t}\) such that \(i_{t+1}\in R_{k'}\) and
then, by atomicity, \(R_{\bar k}\subseteq R_{k'}\). By definition of
\(F\), \(k'\) cannot be in \(\mathcal{K}_0^-\), so there exists
\(t'\in\{1,\dotsc,t\}\) such that \(k'\) has been added to
\(\mathcal{K}^-_{t'}\) by the process of line 18, and so, by lines 16
and 17, \(k'=k^{(t',h^f_{t'})}\) and \(k'\in\mathcal{P}^{t'}\). But we
also have \(\bar k\in\mathcal{P}^{t'}\), and so by \(\mathcal{P}^{t'}\)
realizing a partition, the inclusion \(R_{\bar k}\subseteq R_{k'}\) is
actually an egality, so \(\bar k = k' = k^{(t',h^f_{t'})}\), which is
not possible by assumption of this first subcase.

\subparagraph{\texorpdfstring{Second subcase: \(\bar k\) has been added
to \(\mathcal{P}^t\) at a previous
step}{Second subcase: \textbackslash bar k has been added to \textbackslash mathcal\{P\}\^{}t at a previous step}}\label{second-subcase-bar-k-has-been-added-to-mathcalpt-at-a-previous-step}

Let \(t'\leq t\) be this step. This means that
\(\bar k = k^{(t',h^f_{t'})}\) and that at that step
\(\eta^{t'}_{\bar k}\geq \zeta_{\bar k}\), because the if condition in
line 13 failed. Also \(\bar k\in \mathcal{P}^{t'}\) so
\(\bar k\in \mathcal{K}_{t'}\) so we can write \begin{align*}
\zeta_{\bar k}&\leq\eta^{t'}_{\bar k}\\
&\quad=V^*_{\mathfrak{R}}(S_{t'}\cap R_{\bar k})  \text{ by~\eqref{eq_vstar_inter_Rk_equal_eta}} \\
&\quad\leq | S_{t'}\cap R_{\bar k} |\\
&\quad\leq  | S_{t}\cap R_{\bar k} |.\\
\end{align*}

This concludes the two subcases dichotomy:
\(\zeta_{\bar k}\leq |S_t\cap R_{\bar k}|\) and we can go back to our
\(k\in\mathcal{K}_{t+1}\) such that \(i_{t+1}\in R_k\) and
\(R_{\bar k}\subseteq R_k\).

\subparagraph{End of the first case}\label{end-of-the-first-case}

We write the following chain of claims: \begin{align*}
V^*_{\mathfrak{R}}(S_{t+1}\cap R_k)&\leq \sum_{\substack{k'\in\mathcal{P}^{t}\\R_{k'}\subseteq R_k}} \zeta_{k'}\wedge |S_{t+1}\cap R_{k'}|\text{ by~\eqref{eq_ineq_both_cases} and }\mathcal{K}_{t+1}\subseteq\mathcal{K}_t\\
&=\sum_{\substack{k'\in\mathcal{P}^{t}\\R_{k'}\subseteq R_k\\ k'\neq\bar k}} \zeta_{k'}\wedge |S_{t+1}\cap R_{k'}| + \zeta_{\bar k}\wedge |S_{t+1}\cap R_{\bar k}| \\
&=\sum_{\substack{k'\in\mathcal{P}^{t}\\R_{k'}\subseteq R_k\\ k'\neq\bar k}} \zeta_{k'}\wedge |S_{t}\cap R_{k'}| + \zeta_{\bar k}\wedge( |S_{t}\cap R_{\bar k}| +1)\\
&=\sum_{\substack{k'\in\mathcal{P}^{t}\\R_{k'}\subseteq R_k\\ k'\neq\bar k}} \zeta_{k'}\wedge |S_{t}\cap R_{k'}| + \zeta_{\bar k}\wedge|S_{t}\cap R_{\bar k}|\text{ because $\zeta_{\bar k}\leq |S_t\cap R_{\bar k}|$}\\
&=\sum_{\substack{k'\in\mathcal{P}^{t}\\R_{k'}\subseteq R_k}} \zeta_{k'}\wedge |S_{t}\cap R_{k'}|=\sum_{\substack{k'\in\mathcal{P}^{t+1}\\R_{k'}\subseteq R_k}} \zeta_{k'}\wedge |S_{t}\cap R_{k'}| \\
&=V^*_{\mathfrak{R}}(S_t\cap R_k)  \text{ by~\eqref{eq_Pt_good_partition}} \\
&= \eta_k^t  \text{ by~\eqref{eq_vstar_inter_Rk_equal_eta}} \\
&= \eta_k^{t+1}.\\
\end{align*} But on the other hand, \(S_t\subseteq S_{t+1}\) and so
\eqref{eq_vstar_Qpartition} also gives
\(V^*_{\mathfrak{R}}(S_t\cap R_k) \leq V^*_{\mathfrak{R}}(S_{t+1}\cap R_k)\)
and so in the end we have the desired outcome: \begin{equation*}
V^*_{\mathfrak{R}}(S_{t+1}\cap R_k) =  \eta_k^{t+1} =  \sum_{\substack{k'\in\mathcal{P}^{t+1}\\R_{k'}\subseteq R_k}} \zeta_{k'}\wedge |S_{t+1}\cap R_{k'}| ,
\end{equation*} which concludes this first case.

\paragraph{\texorpdfstring{Second case:
\(i_{t+1}\not\in\bigcup_{\kappa\in\mathcal{K}^-_{t}}R_{\kappa}\)}{Second case: i\_\{t+1\}\textbackslash not\textbackslash in\textbackslash bigcup\_\{\textbackslash kappa\textbackslash in\textbackslash mathcal\{K\}\^{}-\_\{t\}\}R\_\{\textbackslash kappa\}}}\label{second-case-i_t1notinbigcup_kappainmathcalk-_tr_kappa}

Like in the first case, considering a
\(k\in\mathcal{K}_{t+1}\subseteq \mathcal{K}_t\) such that
\(i_{t+1}\not\in R_k\) is not problematic, because in that case \(k\) is
not visited at all by the algorithm at step \(t+1\) :
\(\eta^{t+1}_k=\eta^{t}_k\),
\(\{k'\in\mathcal{P}^{t+1}\,:\,R_{k'}\subseteq R_k\}=\{k'\in\mathcal{P}^{t}\,:\,R_{k'}\subseteq R_k\}\),
and for all \(k'\in \mathcal{K}\) such that \(R_{k'}\subseteq R_k\),
\(S_{t+1}\cap R_{k'}=S_{t}\cap R_{k'}\). Hence, from \begin{equation*}
V^*_{\mathfrak{R}}(S_{t}\cap R_k) =  \eta_k^{t} =  \sum_{\substack{k'\in\mathcal{P}^{t}\\R_{k'}\subseteq R_k}} \zeta_{k'}\wedge |S_{t+1}\cap R_{k'}| ,
\end{equation*} we directly have \begin{equation*}
V^*_{\mathfrak{R}}(S_{t+1}\cap R_k) =  \eta_k^{t+1} =  \sum_{\substack{k'\in\mathcal{P}^{t+1}\\R_{k'}\subseteq R_k}} \zeta_{k'}\wedge |S_{t+1}\cap R_{k'}| .
\end{equation*}

So we now focus on the \(k\in\mathcal{K}_{t+1}\) such that
\(i_{t+1}\in R_k\). Note that for such \(k\), \begin{equation*}
\eta^{t+1}_k=\eta^t_k+1=V^*_{\mathfrak{R}}(S_t\cap R_k)+1=\sum_{\substack{k'\in \mathcal{P}^t\\ R_{k'}\subseteq R_k}}\zeta_{k'}\wedge|S_t\cap R_{k'}|+1
\end{equation*} by construction, by \eqref{eq_vstar_inter_Rk_equal_eta}
and by \eqref{eq_Pt_good_partition}. Indeed, such a \(k\) is equal to a
\(k^{(t+1,h)}\) with \(h\leq h_{max}(t+1)\), and even
\(h\leq h^f_{t+1}\) if the latter exists.

Also, similarly to the first case, for all \(k\in\mathcal{K}_{t+1}\)
such that \(i_{t+1}\in R_k\) (recall that this is equivalent to
\(R_{\bar k}\subseteq R_k\)), we can write: \begin{align}
V^*_{\mathfrak{R}}(S_{t+1}\cap R_k)&\leq \sum_{\substack{k'\in\mathcal{P}^{t}\\R_{k'}\subseteq R_k}} \zeta_{k'}\wedge |S_{t+1}\cap R_{k'}| \text{ by \eqref{eq_ineq_both_cases} and }\mathcal{K}_{t+1}\subseteq\mathcal{K}_t\notag\\
&=\sum_{\substack{k'\in\mathcal{P}^{t}\\R_{k'}\subseteq R_k\\ k'\neq\bar k}} \zeta_{k'}\wedge |S_{t+1}\cap R_{k'}| + \zeta_{\bar k}\wedge |S_{t+1}\cap R_{\bar k}| \notag\\
&=\sum_{\substack{k'\in\mathcal{P}^{t}\\R_{k'}\subseteq R_k\\ k'\neq\bar k}} \zeta_{k'}\wedge |S_{t}\cap R_{k'}| + \zeta_{\bar k}\wedge( |S_{t}\cap R_{\bar k}| +1)\notag\\
&\leq \sum_{\substack{k'\in\mathcal{P}^{t}\\R_{k'}\subseteq R_k\\ k'\neq\bar k}} \zeta_{k'}\wedge |S_{t}\cap R_{k'}| + \zeta_{\bar k}\wedge|S_{t}\cap R_{\bar k}| +1\notag\\
&=\sum_{\substack{k'\in\mathcal{P}^{t}\\R_{k'}\subseteq R_k}} \zeta_{k'}\wedge |S_{t}\cap R_{k'}|  + 1\notag\\
&=V^*_{\mathfrak{R}}(S_t\cap R_k) +1  \text{ by \eqref{eq_Pt_good_partition}}.\label{eq_ineq_}
\end{align}

Note that by the joint construction of \(\mathcal{K}^-_t\) and
\(\mathcal{P}^t\) on lines 17 and 18, the fact that
\(i_{t+1}\not\in\bigcup_{\kappa\in\mathcal{K}^-_{t}}R_{\kappa}\) implies
that \(\bar k\in F\), so \(\bar k\) is the index of an atom, so actually
\(h_{\max}(t+1)=\phi(\bar k)\), \(\bar k = k^{(t+1,\phi(\bar k))}\) and
the \(R_k\), \(k\in\mathcal{K}_{t}\), such that
\(R_{\bar k}\subseteq R_k\) are nested and are exactly indexed by the
\(k^{(t+1,h)}\), \(1\leq h\leq \phi(\bar k)\). We now prove that for all
of them,
\(V^*_{\mathfrak{R}}(S_{t+1}\cap R_k)\geq V^*_{\mathfrak{R}}(S_{t}\cap R_k)+1\),
which will be true in particular for the ones that are in
\(\mathcal{K}_{t+1}\), given that
\(\mathcal{K}_{t+1}\subseteq \mathcal{K}_t\). We do that by constructing
some sets \(A_h\) with good properties with a descending recursion on
\(h\), starting from \(\phi(\bar k)\). We only give the first two steps
of the construction, because every other step is exactly the same as the
second one, which contains the recursive arguments. We go back to the
real definition of \(V^*_{\mathfrak{R}}\) to do so, for any
\(S\subseteq \mathbb{N}_m\): \begin{equation}
\label{eq_Vstar_au_max} V^*_{\mathfrak{R}}(S)=\max_{\substack{A\subseteq \mathbb{N}_m\\\forall k'\in\mathcal{K}, |A\cap R_{k'}|\leq \zeta_{k'}}} |A\cap S| =\max_{\substack{A\subseteq S\\\forall k'\in\mathcal{K}, |A\cap R_{k'}|\leq \zeta_{k'}}} |A|  . 
\end{equation}

By \eqref{eq_Vstar_au_max}, we have that
\(V^*_{\mathfrak{R}}(S_t \cap R_{k^{(t+1,\phi(\bar k))}})=|A_{\phi(\bar k)}|\)
for a given
\(A_{\phi(\bar k)}\subseteq S_t \cap R_{k^{(t+1,\phi(\bar k))}}\) and
such that \(|A_{\phi(\bar k)}\cap R_{k'}|\leq \zeta_{k'}\) for all
\(k'\in\mathcal{K}\). Now for the second set, we construct
\(A_{\phi(\bar k)-1}\). Note that
\(V^*_{\mathfrak{R}}(S_t \cap R_{k^{(t+1,\phi(\bar k)-1)}})=|B|\) for
some \(B\subseteq S_t \cap R_{k^{(t+1,\phi(\bar k)-1)}}\) and such that
\(|B\cap R_{k'}|\leq \zeta_{k'}\) for all \(k'\in\mathcal{K}\). By
reductio ad absurdum, if there are strictly less than
\(V^*_{\mathfrak{R}}(S_t \cap R_{k^{(t+1,\phi(\bar k)-1)}}) - V^*_{\mathfrak{R}}(S_t \cap R_{k^{(t+1,\phi(\bar k))}})=|B|-|A_{\phi(\bar k)}|\)
elements in
\(S_t\cap R_{k^{(t+1,\phi(\bar k)-1)}} \setminus S_t\cap R_{k^{(t+1,\phi(\bar k))}}\),
then
\(|B|+|S_t \cap R_{k^{(t+1,\phi(\bar k))}}|-|S_t \cap R_{k^{(t+1,\phi(\bar k)-1)}}|>|A_{\phi(\bar k)}|=V^*_{\mathfrak{R}}(S_t \cap R_{k^{(t+1,\phi(\bar k))}})\).
Given that
\(B\cup (S_t\cap R_{k^{(t+1,\phi(\bar k))}})\subseteq S_t \cap R_{k^{(t+1,\phi(\bar k)-1)}}\),
this entails
\(|B\cap S_t\cap R_{k^{(t+1,\phi(\bar k))}}| =|B|+|S_t \cap R_{k^{(t+1,\phi(\bar k))}}| -| B\cup (S_t\cap R_{k^{(t+1,\phi(\bar k))}})| > V^*_{\mathfrak{R}}(S_t \cap R_{k^{(t+1,\phi(\bar k))}})\)
which contradicts the maximality of \(A_{\phi(\bar k)}\) in
\eqref{eq_Vstar_au_max}.

So we construct \(A_{\phi(\bar k)-1}\) by taking the disjoint union of
\(A_{\phi(\bar k)}\) and
\(V^*_{\mathfrak{R}}(S_t \cap R_{k^{(t+1,\phi(\bar k)-1)}}) - V^*_{\mathfrak{R}}(S_t \cap R_{k^{(t+1,\phi(\bar k))}})\)
elements of
\(S_t\cap R_{k^{(t+1,\phi(\bar k)-1)}} \setminus S_t\cap R_{k^{(t+1,\phi(\bar k))}}\).
We now establish the properties of \(A_{\phi(\bar k)-1}\). First,
\(A_{\phi(\bar k)-1}\subseteq S_t \cap R_{k^{(t+1,\phi(\bar k)-1)}}\),
and
\(|A_{\phi(\bar k)-1}|=V^*_{\mathfrak{R}}(S_t \cap R_{k^{(t+1,\phi(\bar k)-1)}})\).
For all \(k'\in\mathcal{K}\) such that
\(R_{k^{(t+1,\phi(\bar k)-1)}} \cap R_{k'}=\varnothing\), we have
\(|A_{\phi(\bar k)-1}\cap R_{k'}|=0\leq \zeta_k'\). Furthermore,
\begin{align*}
|A_{\phi(\bar k)-1}\cap R_{   k^{(t+1,\phi(\bar k))}   }|&=|A_{\phi(\bar k)}\cap R_{   k^{(t+1,\phi(\bar k))}   }|\\
&\leq \zeta_{ k^{(t+1,\phi(\bar k))} }
\end{align*} by construction of \(A_{\phi(\bar k)}\). Finally, for all
\(k'\) such that \(R_{k^{(t+1,\phi(\bar k)-1)}}\subseteq R_{k'}\),
\(|A_{\phi(\bar k)-1}\cap R_{k'}|=|A_{\phi(\bar k)-1}|=V^*_{\mathfrak{R}}(S_t \cap R_{k^{(t+1,\phi(\bar k)-1)}})=|B|\)
with the previously defined \(B\), in particular
\(|B\cap R_{k'}|\leq \zeta_{k'}\), but given that
\(B\subseteq S_t \cap R_{k^{(t+1,\phi(\bar k)-1)}}\),
\(|B\cap R_{k'}|=|B|\). Wrapping all those equalities, it comes that
\(|A_{\phi(\bar k)-1}\cap R_{k'}|\leq \zeta_{k'}\). In the end,
\(|A_{\phi(\bar k)-1}\cap R_{k'}|\leq \zeta_{k'}\) for all
\(k'\in\mathcal{K}\), so \(A_{\phi(\bar k)-1}\) realizes the maximum in
\eqref{eq_Vstar_au_max} for \(S_t \cap R_{k^{(t+1,\phi(\bar k)-1)}}\).

By applying exactly the same method, we recursively construct a
non-increasing sequence \(A_{\phi(\bar k)}\subseteq\dotsb\subseteq A_1\)
such that for all \(\ell\in\{1,\dotsc, \phi(\bar k)\}\) and
\(k'\in\mathcal{K}\), \(A_\ell\subseteq S_t\cap R_{k^{(t+1,\ell)}}\),
\(V^*_{\mathfrak{R}}(S_t\cap R_{k^{(t+1,\ell)}})=|A_\ell|\), and
\(|A_\ell\cap R_{k'}|\leq \zeta_{k'}\). Furthermore for \(\ell'>\ell\),
\(A_{\ell}\setminus A_{\ell'}\subseteq S_t\cap R_{k^{(t+1,\ell)}}\setminus S_t\cap R_{k^{(t+1,\ell')}}\).
Also note that the fact that
\(i_{t+1}\not\in\bigcup_{\kappa\in\mathcal{K}^-_{t}}R_{\kappa}\) implies
that \(\eta^t_{k^{(t+1,\ell)}}<\zeta_{k^{(t+1,\ell)}}\) for all
\(\ell\in\{1,\dotsc, \phi(\bar k)\}\). So by
\eqref{eq_vstar_inter_Rk_equal_eta},
\(|A_\ell|<\zeta_{k^{(t+1,\ell)}}\).

Let, for any \(\ell\in\{1,\dotsc, \phi(\bar k)\}\),
\(\widetilde A_\ell=A_\ell \cup \{ i_{t+1}\}\). Given that
\(A_\ell\subseteq S_t\cap R_{k^{(t+1,\ell)}}\) and that
\(i_{t+1}\in S_{t+1}\setminus S_t\),
\(\widetilde A_\ell\subseteq S_{t+1}\cap R_{k^{(t+1,\ell)}}\),
\(|\widetilde A_\ell|=| A_\ell|+1\), and for all
\(\ell'\in\{1,\dotsc, \phi(\bar k)\}\),
\(|\widetilde A_\ell \cap R_{k^{(t+1,\ell')}} |=| A_\ell\cap R_{k^{(t+1,\ell')}}|+1\).
Note that if, furthermore, \(\ell\geq \ell'\), then
\(A_\ell\subseteq A_{\ell'}\), so \begin{align*}
|\widetilde A_\ell \cap R_{k^{(t+1,\ell')}} |&=| A_\ell\cap R_{k^{(t+1,\ell')}}|+1\\
&\leq | A_{\ell'}\cap R_{k^{(t+1,\ell')}}|+1\\
&= | A_{\ell'}|+1\\
&<\zeta_{k^{(t+1,\ell')}}+1.\\
\end{align*} On the contrary, if \(\ell< \ell'\), we write that
\begin{align*}
|\widetilde A_\ell \cap R_{k^{(t+1,\ell')}} |&=| A_\ell\cap R_{k^{(t+1,\ell')}}|+1\\
&= | (A_{\ell}\setminus A_{\ell'}) \cap R_{k^{(t+1,\ell')}}   | + | A_{\ell'}\cap R_{k^{(t+1,\ell')}}|+1\\
&<  0 +  \zeta_{k^{(t+1,\ell')}} +1,
\end{align*} because
\(A_{\ell}\setminus A_{\ell'} \subseteq   R_{k^{(t+1,\ell)}}\setminus  R_{k^{(t+1,\ell')}}\)
hence
\((A_{\ell}\setminus A_{\ell'}) \cap R_{k^{(t+1,\ell')}}  =\varnothing\).
In both cases,
\(|\widetilde A_\ell \cap R_{k^{(t+1,\ell')}} |< \zeta_{k^{(t+1,\ell')}} +1\)
so
\(|\widetilde A_\ell \cap R_{k^{(t+1,\ell')}} |\leq \zeta_{k^{(t+1,\ell')}}\).
Additionally, for all \(k'\in\mathcal{K}\) such that
\(i_{t+1}\not\in R_{k'}\),
\(|\widetilde A_\ell \cap R_{k'} |=| A_\ell \cap R_{k'} |\leq \zeta_{k'}\).

In the end, \(|\widetilde A_\ell \cap R_{k'} |\leq \zeta_{k'}\) for all
\(k'\in\mathcal{K}\), so \begin{align*}
V^*_{\mathfrak{R}}( S_{t+1}\cap R_{k^{(t+1,\ell)}})&\geq | \widetilde A_\ell | \text{ by \eqref{eq_Vstar_au_max}}\\
&=  |  A_\ell | +1\\
&= V^*_{\mathfrak{R}}( S_{t}\cap R_{k^{(t+1,\ell)}}) +1.
\end{align*} So, as we wanted,
\(V^*_{\mathfrak{R}}(S_{t+1}\cap R_k)\geq V^*_{\mathfrak{R}}(S_{t}\cap R_k)+1\)
for all \(k\in\mathcal{K}_{t}\) such that \(i_{t+1}\in R_k\) and so for
all such \(k\) that are in \(\mathcal{K}_{t+1}\). So every inequality in
\eqref{eq_ineq_} becomes an equality and we have proven that
\begin{equation*}
V^*_{\mathfrak{R}}(S_{t+1}\cap R_k) = V^*_{\mathfrak{R}}(S_{t}\cap R_k)+1 =\eta^t_k+1= \eta^{t+1}_k,
\end{equation*} that is, \eqref{eq_vstar_inter_Rk_equal_eta} is true at
\(t+1\). Looking at the first line of \eqref{eq_ineq_} , we also proved
that \begin{equation}
V^*_{\mathfrak{R}}(S_{t+1}\cap R_k) =  \sum_{\substack{k'\in\mathcal{P}^{t}\\R_{k'}\subseteq R_k}} \zeta_{k'}\wedge |S_{t+1}\cap R_{k'}| . \label{eq_Pt_instead}
\end{equation} The only thing left to prove is that
\eqref{eq_Pt_instead} is also true with \(\mathcal{P}^{t+1}\) instead of
\(\mathcal{P}^{t}\), that is that \eqref{eq_Pt_good_partition} also
holds at \(t+1\), or, put differently, that
\begin{equation}\label{eq_last_goal}
\sum_{\substack{k'\in\mathcal{P}^{t}\\R_{k'}\subseteq R_k}} \zeta_{k'}\wedge |S_{t+1}\cap R_{k'}|=\sum_{\substack{k'\in\mathcal{P}^{t+1}\\R_{k'}\subseteq R_k}} \zeta_{k'}\wedge |S_{t+1}\cap R_{k'}|.
\end{equation} If \(h^f_{t+1}\) does not exist, meaning that we didn't
break the loop, \(\mathcal{P}^{t+1}=\mathcal{P}^t\) so there is nothing
to prove.

Now assume that \(h^f_{t+1}\) exists. So \eqref{eq_rel_rec_p} holds. We
will split each term in \eqref{eq_last_goal} in a sum of two terms.
First, note that by \eqref{eq_rel_rec_p}, for any \(k'\in\mathcal{K}\)
such that \(R_{k'}\cap  R_{  k^{(t+1,h^f_{t+1})} } = \varnothing\), we
have that \(k'\in\mathcal{P}^{t+1}\) if and only if
\(k'\in \mathcal{P}^{t}\). And so, \begin{align*}
 \sum_{\substack{k'\in\mathcal{P}^{t+1}\\R_{k'}\subseteq R_k}} \zeta_{k'}\wedge |S_{t+1}\cap R_{k'}| &=  \sum_{\substack{k'\in\mathcal{P}^{t+1}\\R_{k'}\cap  R_{  k^{(t+1,h^f_{t+1})} } = \varnothing \\R_{k'}\subseteq R_k }} \zeta_{k'}\wedge |S_{t+1}\cap R_{k'}| +  \zeta_{ k^{(t+1,h^f_{t+1})}}\wedge |S_{t+1}\cap R_{ k^{(t+1,h^f_{t+1})}}|\\ 
 &=  \sum_{\substack{k'\in\mathcal{P}^{t}\\R_{k'}\cap  R_{  k^{(t+1,h^f_{t+1})} } = \varnothing \\R_{k'}\subseteq R_k }} \zeta_{k'}\wedge |S_{t+1}\cap R_{k'}| +  \zeta_{ k^{(t+1,h^f_{t+1})}}\wedge |S_{t+1}\cap R_{ k^{(t+1,h^f_{t+1})}}|.
\end{align*}

Recall that we already proved that there is no \(k'\in \mathcal{P}^t\)
such that \(R_{  k^{(t+1,h^f_{t+1})} }\subsetneq R_{k'}\), so for any
\(k'\in \mathcal{P}^t\), either
\(R_{k'}\cap  R_{  k^{(t+1,h^f_{t+1})} } = \varnothing\) or
\(R_{k'}\subseteq  R_{  k^{(t+1,h^f_{t+1})}}\). Hence the split
\begin{align*}
 \sum_{\substack{k'\in\mathcal{P}^{t}\\R_{k'}\subseteq R_k}} \zeta_{k'}\wedge |S_{t+1}\cap R_{k'}| &=  \sum_{\substack{k'\in\mathcal{P}^{t}\\R_{k'}\cap  R_{  k^{(t+1,h^f_{t+1})} } = \varnothing \\R_{k'}\subseteq R_k }} \zeta_{k'}\wedge |S_{t+1}\cap R_{k'}|  \;   + \sum_{\substack{k'\in\mathcal{P}^{t}\\R_{k'}\subseteq  R_{  k^{(t+1,h^f_{t+1})}}\\ R_{k'}\subseteq R_k}} \zeta_{k'}\wedge |S_{t+1}\cap R_{k'}|\\
 &=\sum_{\substack{k'\in\mathcal{P}^{t}\\R_{k'}\cap  R_{  k^{(t+1,h^f_{t+1})} } = \varnothing \\R_{k'}\subseteq R_k }} \zeta_{k'}\wedge |S_{t+1}\cap R_{k'}|  \;   + \sum_{\substack{k'\in\mathcal{P}^{t}\\R_{k'}\subseteq  R_{  k^{(t+1,h^f_{t+1})}}}} \zeta_{k'}\wedge |S_{t+1}\cap R_{k'}|,
\end{align*} where the last equality comes from the fact that
\(R_{  k^{(t+1,h^f_{t+1})} }\subseteq R_k\), because
\(k\in\mathcal{K}_{t+1}\), \(i_{t+1}\in R_k\), and
\(k^{(t+1,h^f_{t+1})}\in \mathcal{P}^{t+1}\).

Given the two previously made splits, it remains to prove that
\begin{equation*}
\sum_{\substack{k'\in\mathcal{P}^{t}\\R_{k'}\subseteq  R_{  k^{(t+1,h^f_{t+1})}} }} \zeta_{k'}\wedge |S_{t+1}\cap R_{k'}| =  \zeta_{ k^{(t+1,h^f_{t+1})}}\wedge |S_{t+1}\cap R_{ k^{(t+1,h^f_{t+1})}}|.
\end{equation*} Interestingly, this does not depend on \(k\) anymore.

Let us show that \begin{equation}
\eta^{t+1}_{k^{(t+1,h^f_{t+1})} }= \zeta_{ k^{(t+1,h^f_{t+1})}}.
\label{zeta_equal_eta}
\end{equation} Due to \(h^f_{t+1}\) existing, we broke the loop in line
19, so the condition in line 13 was false, so
\(\eta^{t+1}_{k^{(t+1,h^f_{t+1})} }\geq \zeta_{ k^{(t+1,h^f_{t+1})}}\).
We show by recursion over \(t'\in\{0,\dotsc,t\}\) that
\(\eta^{t'}_{k^{(t+1,h^f_{t+1})} }< \zeta_{ k^{(t+1,h^f_{t+1})}}\).
Given that
\(i_{t+1}\not\in\bigcup_{\kappa\in\mathcal{K}^-_{t}}R_{\kappa}\) and
\(i_{t+1}\in R_{k^{(t+1,h^f_{t+1})}}\),
\(k^{(t+1,h^f_{t+1})} \not\in \mathcal{K}^-_{t}\) and in particular
\(k^{(t+1,h^f_{t+1})} \not\in \mathcal{K}^-_{0}\) so
\(\eta^{0}_{k^{(t+1,h^f_{t+1})} }=0< \zeta_{ k^{(t+1,h^f_{t+1})}}\). Now
let \(t'<t\) and assume that
\(\eta^{t'}_{k^{(t+1,h^f_{t+1})} }< \zeta_{ k^{(t+1,h^f_{t+1})}}\). We
distinguish two cases. If \(i_{t'+1}\not\in R_{k^{(t+1,h^f_{t+1})}}\),
\(k^{(t+1,h^f_{t+1})}\) is not visited at step \(t'+1\) of the
algorithm, so
\(\eta^{t'+1}_{k^{(t+1,h^f_{t+1})} }=\eta^{t'}_{k^{(t+1,h^f_{t+1})} }< \zeta_{ k^{(t+1,h^f_{t+1})}}\)
by the recursion hypothesis. If \(i_{t'+1}\in R_{k^{(t+1,h^f_{t+1})}}\),
all \(k^{(t'+1,h)}\), \(h\leq h^f_{t+1}\), are visited, and for all of
them the condition in line 13 is true, otherwise we would have, for some
\(h\leq h^f_{t+1}\),
\(k^{(t'+1,h)}\in\mathcal{K}^-_{t'+1}\subseteq\mathcal{K}^-_{t}\).
Noting that, necessarily, \(k^{(t'+1,h)}=k^{(t+1,h)}\), we would finally
have
\(i_{t+1}\in R_{k^{(t+1,h^f_{t+1})}}\subseteq R_{k^{(t+1,h)}}\subseteq \bigcup_{\kappa\in\mathcal{K}^-_{t}}R_{\kappa}\)
which is a contradiction. This completes the recursion. Specifically for
\(t'=t\),
\(\eta^{t}_{k^{(t+1,h^f_{t+1})} }\leq \zeta_{ k^{(t+1,h^f_{t+1})}}-1\)
and so
\(\eta^{t+1}_{k^{(t+1,h^f_{t+1})} }\leq\eta^{t}_{k^{(t+1,h^f_{t+1})} }+1\leq \zeta_{ k^{(t+1,h^f_{t+1})}}\)
and so \eqref{zeta_equal_eta} holds.

Also note that
\(k^{(t+1,h^f_{t+1})}\in\mathcal{P}^{t+1}\subseteq\mathcal{K}_{t+1}\),
which implies two things. Firstly, because
\eqref{eq_vstar_inter_Rk_equal_eta} holds at \(t+1\),
\(\eta^{t+1}_{k^{(t+1,h^f_{t+1})} }=V^*_{\mathfrak{R}}( S_{t+1}\cap  R_{  k^{(t+1,h^f_{t+1})}  } )\).
Secondly, by applying \eqref{eq_Pt_instead}, we get that
\(\sum_{\substack{k'\in\mathcal{P}^{t}\\R_{k'}\subseteq  R_{  k^{(t+1,h^f_{t+1})}} }} \zeta_{k'}\wedge |S_{t+1}\cap R_{k'}| = V^*_{\mathfrak{R}}( S_{t+1}\cap  R_{  k^{(t+1,h^f_{t+1})}  } )\).
Wrapping all these assertions: \begin{align*}
\sum_{\substack{k'\in\mathcal{P}^{t}\\R_{k'}\subseteq  R_{  k^{(t+1,h^f_{t+1})}} }} \zeta_{k'}\wedge |S_{t+1}\cap R_{k'}| &= V^*_{\mathfrak{R}}( S_{t+1}\cap  R_{  k^{(t+1,h^f_{t+1})}  } )\\
&=V^*_{\mathfrak{R}}( S_{t+1}\cap  R_{  k^{(t+1,h^f_{t+1})}  } )\wedge  |S_{t+1}\cap R_{ k^{(t+1,h^f_{t+1})}}|\\
&=\eta^{t+1}_{k^{(t+1,h^f_{t+1})} }\wedge  |S_{t+1}\cap R_{ k^{(t+1,h^f_{t+1})}}|\\
&=  \zeta_{ k^{(t+1,h^f_{t+1})}}\wedge |S_{t+1}\cap R_{ k^{(t+1,h^f_{t+1})}}| \text{ by \eqref{zeta_equal_eta}},
\end{align*} which achieves the second case and so the proof of
Theorem~\ref{thm-curve-path}.

\subsection{\texorpdfstring{Proof of
Proposition~\ref{prp-cardinals}}{Proof of Proposition~}}\label{sec-cardinals}

Recall that the \(R_k\) are assumed to be all non-empty and distinct by
Remark~\ref{rem-distinct}.

The \(P_n\), \(n\in \mathbb N_N^*\), form a partition of
\(\mathbb{N}_m^*\), so \(N\leq m\).

There is a sequence \(R^{(H)}\subsetneq\dotsb\subsetneq R^{(1)}\). There
exists \(n_H\in\mathbb N_N^*\) such that \(P_{n_H}\subseteq R^{(H)}\)
and for \(1\leq h \leq H-1\), there exists \(n_h\in\mathbb N_N^*\) such
that \(P_{n_h}\subseteq R^{(h)}\setminus R^{(h+1)}\). For \(h_1, h_2\)
such that \(1\leq h_1<h_2 \leq H\), we have
\(P_{n_{h_2}}\subseteq R^{(h_2)}\) and
\(P_{n_{h_1}}\subseteq R^{(h_1)}\setminus R^{(h_1+1)}\subseteq R^{(h_1)}\setminus R^{(h_2)}\)
so \(P_{n_{h_1}}\cap P_{n_{h_2}}=\varnothing\) and in particular
\(n_{h_1}\neq n_{h_2}\), so \(h\mapsto n_h\) is an injection and
\(H\leq N\).

To show that the three bounds can be achieved simultaneously, we let
\(P_i=\{i\}\) for all \(i\in\mathbb{N}_m^*\),
\(\mathcal{K}=\{(i,i), i\in\mathbb{N}_m^*\}\cup\{(1,i), i\in\mathbb{N}_m^*\}\),
and, as usual, for \((i,j)\in\mathcal{K}\),
\(R_{(i,j)}=P_{i:j}=\bigcup_{i\leq n\leq j}P_n\).

Finally we show by induction over \(N\geq1\) that for any family
\((R_k)_{k\in\mathcal{K}}\) with a forest structure and \(N\) leaves,
\(|\mathcal{K}|\leq 2N-1\). For \(N=1\) it is trivial, necessarily
\(P_1=\mathbb{N}_m^*\) and then the only possible set \(R_k\) is also
\(\mathbb{N}_m^*\) so \(|\mathcal{K}|=1=2N-1\). Let \(N\geq 1\). Assume
that for any family \((R_k)_{k\in\mathcal{K}}\) with a forest structure
and \(N\) compatible leaves, \(|\mathcal{K}|\leq 2N-1\). Let
\((R_k)_{k\in\mathcal{K}}\) a forest structure with \(N+1\) leaves, let
\(H\) its maximum depth, let \(P_1, \dotsc P_{N+1}\) the leaves. If
\(H=1\), all the regions are two-by-two disjoint so there is an
injection from the regions to the leaves and so
\(|\mathcal{K}|\leq N+1\leq 2(N+1)-1\). From now on we assume \(H\geq2\)
and we distinguish two cases.

In the first case, assume that there exists \(\hat k\) of depth \(H\),
that is \(\phi(\hat k)=H\), such that \(R_{\hat k}\) is comprised of at
least two leaves: there exist \(\tilde\imath\) and \(\tilde\jmath\) with
\(\tilde\jmath\geq\tilde\imath+1\) such that
\(R_{\hat k}=\bigcup_{\tilde\imath\leq n\leq \tilde\jmath}P_n\). Let
\(\mathcal{K}^-=\mathcal{K}\setminus\{(\tilde\imath, \tilde\imath), (\tilde\imath+1, \tilde\imath+1)\}\),
\((R_k)_{k\in\mathcal{K}^-}\) has also a forest structure, and we show
that
\(P_1,\dotsc,P_{\tilde\imath-1}, P_{\tilde\imath}\cup P_{\tilde\imath+1},P_{\tilde\imath+2},\dotsc,P_{N+1}\)
is a sequence of \(N\) leaves that are compatible with this family.
First note that they well define a partition of \(\mathbb{N}_m^*\). Let
\(k\in\mathcal{K}^-\), we just have to prove that if
\(P_{\tilde\imath}\subseteq R_k\) or
\(P_{\tilde\imath+1}\subseteq R_k\), then
\(P_{\tilde\imath}\cup P_{\tilde\imath+1}\subseteq R_k\). If that's the
case, then \(R_k\cap R_{\hat k}\neq\varnothing\), and by the forest
structure property of \(\mathcal{K}\), \(R_k\subsetneq R_{\hat k}\) or
\(R_{\hat k}\subseteq R_k\), but actually \(R_k\subsetneq R_{\hat k}\)
is impossible because \(\phi(\hat k)=H\) which is the maximal depth. So
\(R_{\hat k}\subseteq R_k\), noticing that
\(P_{\tilde\imath}\cup P_{\tilde\imath+1}\subseteq R_{\hat k}\), we get
the desired result.

On the contrary, in the second case, assume that for all
\(k\in\mathcal{K}\) of height \(H\), \(R_k\) is a leaf. Let
\(\hat k\in\mathcal{K}\) of depth \(H\), that is \(\phi(\hat k)=H\), and
let \(\tilde k\in\mathcal{K}\) the element of depth \(H-1\) such that
\(R_{\hat k}\subsetneq R_{\tilde k}\). \(\tilde k\) exists because
\(H\geq2\). Identify \(\tilde k\) to
\((\tilde\imath, \tilde\jmath)\in \left(\mathbb N_N^* \right)^2\) such
that \(R_{\tilde k}=\bigcup_{\tilde\imath\leq n\leq \tilde\jmath}P_n\).
If \(\tilde\jmath=\tilde\imath\), then
\(R_{\tilde k}=P_{\tilde\imath}\), then we have also
\(R_{\hat k}=P_{\tilde\imath}\), and there is a contradiction with the
fact that \(R_{\hat k}\subsetneq R_{\tilde k}\). So
\(\tilde\jmath\geq\tilde\imath+1\). Let again
\(\mathcal{K}^-=\mathcal{K}\setminus\{(\tilde\imath, \tilde\imath), (\tilde\imath+1, \tilde\imath+1)\}\),
and let us show again that
\(P_1,\dotsc,P_{\tilde\imath-1}, P_{\tilde\imath}\cup P_{\tilde\imath+1},P_{\tilde\imath+2},\dotsc,P_{N+1}\)
is a sequence of \(N\) leaves that are compatible with
\((R_k)_{k\in\mathcal{K}^-}\). The reasoning is the same as in the first
case, but working with \(\tilde k\) instead of \(\hat k\). Let
\(k\in\mathcal{K}^-\), such that \(P_{\tilde\imath}\subseteq R_k\) or
\(P_{\tilde\imath+1}\subseteq R_k\). Then
\(R_k\cap R_{\tilde k}\neq\varnothing\), and by the forest structure
property of \(\mathcal{K}\), \(R_k\subsetneq R_{\tilde k}\) or
\(R_{\tilde k}\subseteq R_k\). But actually
\(R_k\subsetneq R_{\tilde k}\) is impossible, because this implies that
\(\phi(k)=H\), so \(R_k\) is a leaf, so necessarily
\(R_k=P_{\tilde\imath}\) or \(R_k=P_{\tilde\imath+1}\), but
\(k\in\mathcal{K}^-\) so \(k\neq(\tilde\imath, \tilde\imath)\) and
\(k\neq(\tilde\imath+1, \tilde\imath+1)\), hence a contradiction. So
\(R_{\tilde k}\subseteq R_k\), noticing that
\(P_{\tilde\imath}\cup P_{\tilde\imath+1}\subseteq R_{\tilde k}\), we
get the desired result.

In both cases, we have constructed a forest structure
\((R_k)_{k\in\mathcal{K}^-}\) with \(N\) compatible leaves. By the
induction hypothesis, \(|\mathcal{K}^-|\leq 2N-1\) and so
\(|\mathcal{K}|\leq|\mathcal{K}^-|+2\leq 2(N+1)-1\) which concludes.

A direct, alternative proof that \(|\mathcal{K}|\leq 2m-1\) is given in
next section.

\subsubsection{\texorpdfstring{Direct proof that
\(|\mathcal{K}|\leq 2m-1\)}{Direct proof that \textbar\textbackslash mathcal\{K\}\textbar\textbackslash leq 2m-1}}\label{direct-proof-that-mathcalkleq-2m-1}

We show by induction on \(m\geq1\) that, for a family of subsets
\((R_k)_{k\in\mathcal{K}}\) with a forest structure such that the
\(R_k\) are all non-empty and distinct, \(|\mathcal{K}|\leq 2m-1\). For
\(m=1\) it is trivial, the only subset possible is \(\{1\}\). Now let
\(m\geq 1\) and assume that the result is true for \(m\). Let
\((R_k)_{k\in\mathcal{K}}\) a family of non-empty and distinct subsets
of \(\mathbb N_{m+1}^*\) with a forest structure.

Let \(k_1, \dotsc, k_D\), \(D\leq H\), the indices of the regions
including \(m+1\) (possibly non-existent, in which case \(D=0\)),
ordered such that \(R_{k_1}\subsetneq\dotsb\subsetneq R_{k_D}\). Let
\(\widetilde{\mathcal{K}}=\mathcal{K}\setminus\{k_1, \dotsc, k_D\}\),
and let
\(\mathcal{K}'=\mathcal{K}\setminus\{k_1, k_2\}=\widetilde{\mathcal{K}}\cup\{k_3, \dotsc, k_D\}\).
For \(k\in \widetilde{\mathcal{K}}\), we let \(R'_k=R_k\), and for
\(k\in\{k_3, \dotsc, k_D\}\), we let \(R'_k=R_k\setminus\{m+1\}\). The
rest of the proof consists in proving that \((R'_k)_{k\in\mathcal{K}'}\)
is a family of non-empty and distinct subsets of \(\mathbb{N}_m^*\) with
a forest structure. Once this is proven, by induction hypothesis we will
have \(|\mathcal{K}'|\leq 2m-1\), and finally
\(|\mathcal{K}|\leq |\mathcal{K}'|+2\leq 2m-1+2=2(m+1)-1\).

First, any \(R'_k\), \(k\in\mathcal{K}'\), is non-empty, because if
\(k\in\widetilde{\mathcal{K}}\), \(R_k'=R_k\neq\varnothing\), and if
\(k=k_d\) with \(d\geq3\), \(R_{k_1}\subsetneq R_{k_2}\subsetneq R_k\)
so \(|R_k|\geq 3\) and then \(|R'_k|=|R_k\setminus\{m+1\}|\geq2\).

To prove that \((R'_k)_{k\in\mathcal{K}'}\) is a family of distinct
subsets of \(\mathbb{N}_m^*\) with a forest structure, we need to take
\(k, k'\in\mathcal{K}'\), \(k\neq k'\), and show that
\(R'_k\neq R'_{k'}\) and
\(R'_k\cap R'_{k'}\in\{\varnothing, R'_k, R'_{k'}\}\).

If \(|\widetilde{\mathcal{K}}|\geq2\), let
\(k, k'\in\widetilde{\mathcal{K}}\), \(k\neq k'\). We have \(R'_k=R_k\)
and \(R'_{k'}=R_{k'}\), so \(R'_k\neq R'_{k'}\) and
\(R'_k\cap R'_{k'}\in\{\varnothing, R_k, R_{k'}\}=\{\varnothing, R'_k, R'_{k'}\}\).

If \(D\geq 4\), let \(i, j\in\{3,\dotsc, D\}\), \(i<j\). We have
\(R'_{k_i}=R_{k_i}\setminus\{m+1\}\) and
\(R'_{k_j}=R_{k_j}\setminus\{m+1\}\) with \(R_{k_i}\subsetneq R_{k_j}\),
so \(R'_{k_i}\neq R'_{k_j}\) and
\(R'_{k_i}\cap R'_{k_j}=R_{k_i}\setminus\{m+1\}=R'_{k_i}\).

If \(D\geq 3\) and \(|\widetilde{\mathcal{K}}|\geq1\), let
\(i\in\{3,\dotsc, D\}\) and \(k\in\widetilde{\mathcal{K}}\). We have
\(R'_{k_i}=R_{k_i}\setminus\{m+1\}\) and \(R'_k=R_k\). \begin{align*}
R'_{k_i}\cap R'_k&=(R_{k_i}\setminus\{m+1\})\cap R_k \\
&=R_{k_i}\cap R_k \text{ because } m+1\not\in R_k\\
&\in\{\varnothing, R_{k_i}, R_k\} \text{ by the property of forest structure}
\end{align*} Given that \(R'_{k_i}\subsetneq R_{k_i}\),
\(R'_{k_i}\cap R'_k=R_{k_i}\) is impossible, so
\(R'_{k_i}\cap R'_k\in\{\varnothing, R_k\}=\{\varnothing, R'_k\}\) so
the only thing remaining to prove is that \(R'_{k_i}\) and \(R'_k\) are
distinct. We prove that by showing that if \(R'_k = R'_{k_i}\), there is
a contradiction. Indeed, then \(R_k = R_{k_i}\setminus\{m+1\}\), and we
can study \(R_{k_2}\cap R_k\). On the one hand,
\(R_{k_2}\cap R_k\in\{\varnothing, R_{k_2}, R_k\}=\{\varnothing, R_{k_2}, R_{k_i}\setminus\{m+1\}\}\)
by forest structure. On the other hand, \begin{align*}
R_{k_2}\cap R_k&=(R_{k_2}\setminus\{m+1\})\cap R_k \text{ because }m+1\not\in R_k \\
&=(R_{k_2}\setminus\{m+1\})\cap (R_{k_i}\setminus\{m+1\})\\
&=R_{k_2}\setminus\{m+1\}.
\end{align*} So \(R_{k_2}\cap R_k=R_{k_2}\) is impossible. Furthermore,
\(R_{k_2}\cap R_k=R_{k}\) is also impossible because
\(R_{k_2}\subsetneq R_{k_i}\) and \(m+1\in R_{k_2}\) hence
\(R_{k_2}\setminus\{m+1\}\subsetneq R_{k_i}\setminus\{m+1\}\). So
\(R_{k_2}\cap R_k=\varnothing\), that is
\(R_{k_2}\setminus\{m+1\}=\varnothing\), so \(R_{k_2}=\{m+1\}\) and the
contradiction is the following:
\(\{m+1\}\subseteq R_{k_1}\subsetneq R_{k_2}=\{m+1\}\).

\section*{Acknowledgements}\label{acknowledgements}
\addcontentsline{toc}{section}{Acknowledgements}

This work has been supported by the research grants ANR-20-IDEES-0002
(PIA), ANR-19-CHIA-0021 (BISCOTTE), ANR-23-CE40-0018 (BACKUP) and
ANR-21-CE23-0035 (ASCAI). Thanks to Romain Périer for being the first to
extensively use the new implemented algorithms. Thanks to Pierre Neuvial
for his valuable feedback. Thanks to the three anonymous reviewers whose
suggestions and comments greatly improved this manuscript.

\section*{References}\label{references}
\addcontentsline{toc}{section}{References}

\renewcommand{\bibsection}{}
\bibliography{algo-curve.bib}

\newcommand{\noop}[1]{}
\begin{thebibliography}{24}
\providecommand{\natexlab}[1]{#1}
\providecommand{\url}[1]{\texttt{#1}}
\expandafter\ifx\csname urlstyle\endcsname\relax
  \providecommand{\doi}[1]{doi: #1}\else
  \providecommand{\doi}{doi: \begingroup \urlstyle{rm}\Url}\fi

\bibitem[Benjamini and Hochberg(1995)]{MR1325392}
Yoav Benjamini and Yosef Hochberg.
\newblock Controlling the false discovery rate: a practical and powerful
  approach to multiple testing.
\newblock \emph{J. Roy. Statist. Soc. Ser. B}, 57\penalty0 (1):\penalty0
  289--300, 1995.
\newblock ISSN 0035-9246.
\newblock URL \url{https://www.jstor.org/stable/2346101}.

\bibitem[Blain et~al.(2022)Blain, Thirion, and Neuvial]{blain22notip}
Alexandre Blain, Bertrand Thirion, and Pierre Neuvial.
\newblock {Notip: Non-parametric True Discovery Proportion control for brain
  imaging}.
\newblock \emph{Neuroimage}, 260, October 2022.
\newblock URL \url{https://doi.org/10.1016/j.neuroimage.2022.119492}.

\bibitem[Blain et~al.(2023)Blain, Thirion, Grisel, and
  Neuvial]{NEURIPS2023_f6712d51}
Alexandre Blain, Bertrand Thirion, Olivier Grisel, and Pierre Neuvial.
\newblock False discovery proportion control for aggregated knockoffs.
\newblock In A.~Oh, T.~Naumann, A.~Globerson, K.~Saenko, M.~Hardt, and
  S.~Levine, editors, \emph{Advances in Neural Information Processing Systems},
  volume~36, pages 78193--78204. Curran Associates, Inc., 2023.
\newblock URL
  \url{https://proceedings.neurips.cc/paper_files/paper/2023/file/f6712d5191d2501dfc7024389f7bfcdd-Paper-Conference.pdf}.

\bibitem[Blanchard et~al.(2020)Blanchard, Neuvial, and Roquain]{MR4124323}
Gilles Blanchard, Pierre Neuvial, and Etienne Roquain.
\newblock Post hoc confidence bounds on false positives using reference
  families.
\newblock \emph{Ann. Statist.}, 48\penalty0 (3):\penalty0 1281--1303, 2020.
\newblock ISSN 0090-5364.
\newblock \doi{10.1214/19-AOS1847}.
\newblock URL \url{https://doi.org/10.1214/19-AOS1847}.

\bibitem[Bogdan et~al.(2015)Bogdan, van~den Berg, Sabatti, Su, and
  Cand\`es]{MR3418717}
Ma{\l}gorzata Bogdan, Ewout van~den Berg, Chiara Sabatti, Weijie Su, and
  Emmanuel~J. Cand\`es.
\newblock S{LOPE}---adaptive variable selection via convex optimization.
\newblock \emph{Ann. Appl. Stat.}, 9\penalty0 (3):\penalty0 1103--1140, 2015.
\newblock ISSN 1932-6157,1941-7330.
\newblock \doi{10.1214/15-AOAS842}.
\newblock URL \url{https://doi.org/10.1214/15-AOAS842}.

\bibitem[Chang et~al.(2025)Chang, Cheng, Allaire, Sievert, Schloerke, Xie,
  Allen, McPherson, Dipert, and Borges]{shiny}
Winston Chang, Joe Cheng, JJ~Allaire, Carson Sievert, Barret Schloerke, Yihui
  Xie, Jeff Allen, Jonathan McPherson, Alan Dipert, and Barbara Borges.
\newblock \emph{shiny: Web Application Framework for R}, 2025.
\newblock URL \url{https://CRAN.R-project.org/package=shiny}.
\newblock R package version 1.11.0.

\bibitem[Durand et~al.(2020)Durand, Blanchard, Neuvial, and Roquain]{MR4178188}
Guillermo Durand, Gilles Blanchard, Pierre Neuvial, and Etienne Roquain.
\newblock Post hoc false positive control for structured hypotheses.
\newblock \emph{Scand. J. Stat.}, 47\penalty0 (4):\penalty0 1114--1148, 2020.
\newblock ISSN 0303-6898.
\newblock \doi{10.1111/sjos.12453}.
\newblock URL \url{https://doi.org/10.1111/sjos.12453}.

\bibitem[Dvoretzky et~al.(1956)Dvoretzky, Kiefer, and Wolfowitz]{MR0083864}
A.~Dvoretzky, J.~Kiefer, and J.~Wolfowitz.
\newblock Asymptotic minimax character of the sample distribution function and
  of the classical multinomial estimator.
\newblock \emph{Ann. Math. Statist.}, 27:\penalty0 642--669, 1956.
\newblock ISSN 0003-4851.
\newblock \doi{10.1214/aoms/1177728174}.
\newblock URL \url{https://doi.org/10.1214/aoms/1177728174}.

\bibitem[Enjalbert~Courrech(2024)]{enjalbertcourrech:tel-05034928}
Nicolas Enjalbert~Courrech.
\newblock \emph{{Post-selection inference for transcriptomic data analysis}}.
\newblock Theses, {Universit{\'e} de Toulouse}, December 2024.
\newblock URL \url{https://theses.hal.science/tel-05034928}.

\bibitem[Enjalbert-Courrech and Neuvial(2022)]{10.1093/bioinformatics/btac693}
Nicolas Enjalbert-Courrech and Pierre Neuvial.
\newblock Powerful and interpretable control of false discoveries in two-group
  differential expression studies.
\newblock \emph{Bioinformatics}, 38\penalty0 (23):\penalty0 5214--5221, 10
  2022.
\newblock ISSN 1367-4803.
\newblock \doi{10.1093/bioinformatics/btac693}.
\newblock URL \url{https://doi.org/10.1093/bioinformatics/btac693}.

\bibitem[{Enjalbert Courrech} and Neuvial(2025)]{IIDEA}
Nicolas {Enjalbert Courrech} and Pierre Neuvial.
\newblock \emph{IIDEA: Interactive Inference for Differential Expression
  Analyses}, 2025.
\newblock R package version 0.0.1.0.

\bibitem[Genovese and Wasserman(2006)]{MR2279468}
Christopher~R. Genovese and Larry Wasserman.
\newblock Exceedance control of the false discovery proportion.
\newblock \emph{J. Amer. Statist. Assoc.}, 101\penalty0 (476):\penalty0
  1408--1417, 2006.
\newblock ISSN 0162-1459.
\newblock \doi{10.1198/016214506000000339}.
\newblock URL \url{https://doi.org/10.1198/016214506000000339}.

\bibitem[Goeman and Solari(2011)]{MR2951390}
Jelle~J. Goeman and Aldo Solari.
\newblock Multiple testing for exploratory research.
\newblock \emph{Statist. Sci.}, 26\penalty0 (4):\penalty0 584--597, 2011.
\newblock ISSN 0883-4237.
\newblock \doi{10.1214/11-STS356}.
\newblock URL \url{https://doi.org/10.1214/11-STS356}.

\bibitem[Marcus et~al.(1976)Marcus, Peritz, and Gabriel]{MR468056}
Ruth Marcus, Eric Peritz, and K.~R. Gabriel.
\newblock On closed testing procedures with special reference to ordered
  analysis of variance.
\newblock \emph{Biometrika}, 63\penalty0 (3):\penalty0 655--660, 1976.
\newblock ISSN 0006-3444.
\newblock \doi{10.1093/biomet/63.3.655}.
\newblock URL \url{https://doi.org/10.1093/biomet/63.3.655}.

\bibitem[Massart(1990)]{MR1062069}
P.~Massart.
\newblock The tight constant in the {D}voretzky-{K}iefer-{W}olfowitz
  inequality.
\newblock \emph{Ann. Probab.}, 18\penalty0 (3):\penalty0 1269--1283, 1990.
\newblock ISSN 0091-1798,2168-894X.
\newblock URL
  \url{http://links.jstor.org/sici?sici=0091-1798(199007)18:3<1269:TTCITD>2.0.CO;2-Q&origin=MSN}.

\bibitem[Meah et~al.(2024)Meah, Blanchard, and Roquain]{JMLR:v25:23-1025}
Iqraa Meah, Gilles Blanchard, and Etienne Roquain.
\newblock False discovery proportion envelopes with m-consistency.
\newblock \emph{Journal of Machine Learning Research}, 25\penalty0
  (270):\penalty0 1--52, 2024.
\newblock URL \url{http://jmlr.org/papers/v25/23-1025.html}.

\bibitem[Meijer et~al.(2015)Meijer, Krebs, and Goeman]{MR3305943}
Rosa~J. Meijer, Thijmen J.~P. Krebs, and Jelle~J. Goeman.
\newblock A region-based multiple testing method for hypotheses ordered in
  space or time.
\newblock \emph{Stat. Appl. Genet. Mol. Biol.}, 14\penalty0 (1):\penalty0
  1--19, 2015.
\newblock ISSN 2194-6302.
\newblock \doi{10.1515/sagmb-2013-0075}.
\newblock URL \url{https://doi.org/10.1515/sagmb-2013-0075}.

\bibitem[Meinshausen(2006)]{MR2279639}
Nicolai Meinshausen.
\newblock False discovery control for multiple tests of association under
  general dependence.
\newblock \emph{Scand. J. Statist.}, 33\penalty0 (2):\penalty0 227--237, 2006.
\newblock ISSN 0303-6898.
\newblock \doi{10.1111/j.1467-9469.2005.00488.x}.
\newblock URL \url{https://doi.org/10.1111/j.1467-9469.2005.00488.x}.

\bibitem[Mersmann(2024)]{microbenchmark}
Olaf Mersmann.
\newblock \emph{microbenchmark: Accurate Timing Functions}, 2024.
\newblock URL \url{https://CRAN.R-project.org/package=microbenchmark}.
\newblock R package version 1.5.0.

\bibitem[Neuvial et~al.(2024)Neuvial, Blanchard, Durand, Enjalbert-Courrech,
  and Roquain]{sanssouci}
Pierre Neuvial, Gilles Blanchard, Guillermo Durand, Nicolas Enjalbert-Courrech,
  and Etienne Roquain.
\newblock \emph{sanssouci: Post Hoc Multiple Testing Inference}, 2024.
\newblock URL \url{https://sanssouci-org.github.io/sanssouci}.
\newblock R package version 0.13.0.

\bibitem[Perrot-Dock{\`e}s et~al.(2023)Perrot-Dock{\`e}s, Blanchard, Neuvial,
  and Roquain]{perrot2023selective}
Marie Perrot-Dock{\`e}s, Gilles Blanchard, Pierre Neuvial, and Etienne Roquain.
\newblock Selective inference for false discovery proportion in a hidden markov
  model.
\newblock \emph{TEST}, pages 1--27, 2023.

\bibitem[{R Core Team}(2024)]{R-base}
{R Core Team}.
\newblock \emph{R: A Language and Environment for Statistical Computing}.
\newblock R Foundation for Statistical Computing, Vienna, Austria, 2024.
\newblock URL \url{https://www.R-project.org/}.

\bibitem[Van~Rossum and Drake(2009)]{10.5555/1593511}
Guido Van~Rossum and Fred~L. Drake.
\newblock \emph{Python 3 Reference Manual}.
\newblock CreateSpace, Scotts Valley, CA, 2009.
\newblock ISBN 1441412697.

\bibitem[Vesely et~al.(2023)Vesely, Finos, and Goeman]{MR4731977}
Anna Vesely, Livio Finos, and Jelle~J. Goeman.
\newblock Permutation-based true discovery guarantee by sum tests.
\newblock \emph{J. R. Stat. Soc. Ser. B. Stat. Methodol.}, 85\penalty0
  (3):\penalty0 664--683, 2023.
\newblock ISSN 1369-7412,1467-9868.
\newblock \doi{10.1093/jrsssb/qkad019}.
\newblock URL \url{https://doi.org/10.1093/jrsssb/qkad019}.

\end{thebibliography}

\section*{Session information}\label{session-information}
\addcontentsline{toc}{section}{Session information}

\begin{verbatim}
R version 4.5.0 (2025-04-11)
Platform: aarch64-apple-darwin24.4.0
Running under: macOS Sequoia 15.7.4

Matrix products: default
BLAS:   /opt/homebrew/Cellar/openblas/0.3.29/lib/libopenblasp-r0.3.29.dylib 
LAPACK: /opt/homebrew/Cellar/r/4.5.0_1/lib/R/lib/libRlapack.dylib;  LAPACK version 3.12.1

locale:
[1] en_US.UTF-8/en_US.UTF-8/en_US.UTF-8/C/en_US.UTF-8/en_US.UTF-8

time zone: Europe/Paris
tzcode source: internal

attached base packages:
[1] stats     graphics  grDevices utils     datasets  methods   base     

other attached packages:
[1] sanssouci_0.14.2     microbenchmark_1.5.0

loaded via a namespace (and not attached):
 [1] digest_0.6.37     fastmap_1.2.0     xfun_0.52         Matrix_1.7-3     
 [5] lattice_0.22-6    matrixStats_1.5.0 knitr_1.50        htmltools_0.5.8.1
 [9] generics_0.1.4    rmarkdown_2.29    tinytex_0.57      cli_3.6.5        
[13] grid_4.5.0        matrixTests_0.2.3 compiler_4.5.0    rstudioapi_0.17.1
[17] tools_4.5.0       evaluate_1.0.3    Rcpp_1.0.14       yaml_2.3.10      
[21] rlang_1.1.6       jsonlite_2.0.0   
\end{verbatim}

\end{document}